\documentclass[a4paper, bibliography=totoc]{amsart}
\usepackage{amsmath}
\usepackage{amssymb}
\usepackage{enumerate}
\usepackage{amsthm}
\usepackage{appendix}
\usepackage{todonotes}
\usepackage{mathtools} 
\usepackage[UKenglish]{babel}
\usepackage{bbm}
\usepackage[T1]{fontenc}
\usepackage[utf8]{inputenc}
\usepackage{tikz}
\usepackage{thmtools} 
\usepackage{thm-restate} 
\usepackage{stackrel} 
\usepackage[mathlines]{lineno}
\usepackage[normalem]{ulem}

\usepackage[unicode]{hyperref} 

\numberwithin{equation}{section} 
\renewcommand{\theequation}{\arabic{section}.\arabic{equation}}

\declaretheorem[name=Theorem,numberwithin=section]{thm} 
\newtheorem*{thm*}{Theorem}
\newtheorem*{define*}{Definition}
\newtheorem{define}[thm]{Definition}

\newtheorem*{lemma*}{Lemma}
\newtheorem{lemma}[define]{Lemma}

\newtheorem{corollary}[define]{Corollary}

\newtheorem*{algorithm*}{Algorithm}

\newtheorem*{construction*}{Construction}

\newtheorem*{prop*}{Proposition}

\newtheorem*{obs*}{Observation}

\newtheorem*{fact*}{Fact}

\newtheorem*{remark*}{Remark}
\newtheorem{remark}[define]{Remark}

\newtheorem*{quest*}{Question}

\newtheorem*{cor*}{Corollary}
\newtheorem{cor}[define]{Corollary}

\newtheorem*{conjecture*}{Conjecture}

\newtheorem*{question*}{Question}

\newtheorem*{example*}{Example}

\newcommand{\R}{\mathbb{R}}
\newcommand{\Q}{\mathbb{Q}}

\newcommand{\N}{\mathbb{N}}
\newcommand{\Lip}{\mathbb{L}}
\newcommand{\Span}{\operatorname{Span}}
\newcommand{\const}{2} 

\newcommand{\dist}{\operatorname{dist}}
\newcommand{\leb}{\mc{L}}
\newcommand{\SLA}{\operatorname{SLA}}

\newcommand{\skp}[1]{\langle{#1}\rangle}
\newcommand{\lip}{\operatorname{Lip}}

\newcommand{\inter}{\operatorname{Int}}
\newcommand{\diam}{\operatorname{diam}}

\newcommand{\abs}[1]{\left|#1\right|}
\newcommand{\norm}[1]{\left\|#1\right\|}
\newcommand{\supnorm}[1]{\left\|#1\right\|_{\infty}}
\newcommand{\opnorm}[1]{\left\|#1\right\|_{\operatorname{op}}}
\newcommand{\lnorm}[2]{\left\|#2\right\|_{#1}}

\newcommand{\mc}[1]{\mathcal{#1}}

\newcommand{\mb}[1]{\mathbf{#1}}

\newcommand{\set}[1]{\left\{#1\right\}}
\newcommand{\br}[1]{\left(#1\right)}
\newcommand{\bbr}[1]{\bigl(#1\bigr)}
\newcommand{\Bbr}[1]{\Bigl(#1\Bigr)}
\newcommand{\sqbr}[1]{\left[#1\right]}

\newcommand{\cl}[1]{\overline{#1}}
\newcommand{\upp}[1]{^{(#1)}}
\newcommand{\dow}[1]{_{#1}}
\newcommand{\supp}{\operatorname{supp}}
\newcommand{\wt}[1]{\widetilde{#1}}
\newcommand{\indi}[1]{{#1}_{1}}
\newcommand{\indii}[1]{{#1}_{2}}
\newcommand{\indiii}[1]{{#1}_{3}}

\DeclarePairedDelimiter{\inner}{\langle}{\rangle}
\newcommand{\id}{\operatorname{Id}}

\newcommand{\Ball}{\mathbb{B}}
\newcommand{\Balli}[1]{\mathbb{B}_{#1}}
\newcommand{\Sph}{\mathbb{S}}

\newcommand{\1}{\mathbbm{1}}
\newcommand{\eqc}{C}
\newcommand{\cyl}{\mathfrak{C}}
\newcommand{\bas}{\mathcal{W}}
\newcommand{\Gat}{G\^ ateaux }
\newcommand{\Fre}{Fr\' echet }
\newcommand{\myV}{H}
\newcommand{\Jint}{\int_{P_{n}}}

\newcommand{\Jintm}{\int_{P_{m}}}
\newcommand{\inside}[2]{{#1}_{#2}}
\newcommand{\lipone}{\lip_{1}}
\newcommand{\setG}{E}
\newcommand{\op}{T}
\newcommand{\vp}{\mb{v}_{P}}

\title{Extreme non-differentiability of typical Lipschitz mappings}

\author{Michael Dymond}
\address[M. Dymond]{School of Mathematics, University of Birmingham, Edgbaston, Birmingham, B15 2TT, U.K.}
\email{{\tt M.Dymond@bham.ac.uk}}
\address[M.D., former address]{Mathematisches Institut, Universit\"at Leipzig, PF 10 09 02, 04109 Leipzig, Germany}

\author{Olga Maleva}
\address[O. Maleva]{School of Mathematics, University of Birmingham, Edgbaston, Birmingham, B15 2TT, U.K.}
\email{\tt O.Maleva@bham.ac.uk}

\begin{document}
\begin{abstract} 
We show that no matter what subset of a normed space is given, a typical 1-Lipschitz mapping into a Banach space is non-differentiable at a typical point of the set in a very strong sense: the derivative ratio approximates, on arbitrary small scales, every  linear operator of norm at most 1. 

For subsets of finite-dimensional normed spaces which can be covered by a countable union of closed purely unrectifiable sets this extreme non-differenti\-a\-bility holds for a typical Lipschitz mapping at every point. 

Both results are new even for Lipschitz mappings with a finite-dimensional co-domain.
\end{abstract} 
\maketitle
\section{Introduction}
The purpose of this paper is to present a striking (non-)differentiability property of typical Lipschitz mappings. We show that a typical, in Baire category sense, $1$-Lipschitz mapping between a normed space $X$ and a Banach space $Y$ is not differentiable at a typical point of a given set $E\subseteq X$ in a most extreme way: its derivative ratios can approximate all linear operators $X\to Y$ of norm at most $1$. Moreover, if the dimension of $X$ is finite and $E$ is $F_\sigma$ purely unrectifiable, the above holds for all (not just residually many) points of $E$.

Differentiability of Lipschitz mappings is the focus of mathematical research in an array of settings including Euclidean spaces (see e.g.~\cite{fitzpatrick1984differentiation}, \cite{Preiss_1990}, \cite{Dore_Maleva1}, \cite{preiss_speight2013}, \cite{zahorski}), Hilbert and Banach spaces (see e.g.~\cite{benyamini1998geometric}, \cite{LPT}, \cite{Dore_Maleva3}) and geodesic metric spaces (see e.g.~\cite{Kirchheim94}, \cite{pinamonti_speight2017uds}). A starting point for these investigations is Rademacher's theorem, which guarantees that the set of non-differentiability points of 
a Lipschitz mapping $\R^d\to \R^l$ is of Lebesgue measure zero. Versions of Rademacher's Theorem are also available beyond finite-dimensional spaces: under reasonable assumptions on Banach spaces $X,Y$, a Lipschitz mapping $X\to Y$ is \Gat differentiable everywhere except an Aronszajn null set, see~\cite[Theorem~6.42]{benyamini1998geometric}. Furthermore, a celebrated result by Preiss~\cite{Preiss_1990} says that a Lipschitz function defined on a Banach space $X$ which is Asplund (i.e.\ every separable subspace has a separable dual) is differentiable on a dense subset of $X$. Thus, in many settings, Lipschitz mappings constitute a class of mappings which on the whole have good differentiability properties, but crucially have the flexibility of pathological behaviour on a null set. There are various notions of null or exceptional sets, to which this may refer, but even Lebesgue null sets in Euclidean spaces are a diverse class with important tools such as category and fractal dimension which distinguish between them.

Typical differentiability as an object of interest dates back to Banach's famous 1931 result~\cite[Satz~1]{Banach1931} that a typical continuous function on an interval is nowhere differentiable. 
	Such a result would be impossible for a typical Lipschitz mapping between Euclidean spaces, by Rademacher's theorem. The extent to which a typical Lipschitz mapping is differentiable has been investigated recently in~\cite{preiss_tiser94}, \cite{Loewen_Wang_typical_2000}, \cite{dymond_maleva_2020}, \cite{dymond2019typical}, \cite{merlo}.    
Recall also~\cite{Preiss_1990,Dore_Maleva2,Dore_Maleva3,dymond_maleva2016} that Banach spaces with separable dual of dimension $2$ or more have \textit{universal differentiability sets} (UDS) which are ``very small'' but contain a point of differentiability of every $\R$-valued Lipschitz function. 
One naturally asks how ``big'' is the set of points where a typical Lipschitz function is not differentiable. 
In a sense, an opposite of UDS are sets where a typical $1$-Lipschitz $\R$-valued function is nowhere differentiable.
In~\cite{preiss_tiser94}, the class of analytic subsets of $[0,1]\subseteq\R$ with this property has been shown to coincide with the $F_\sigma$ null sets, and~\cite{dymond_maleva_2020} extended this to the case of $[0,1]^d\subseteq \R^d$ by showing that the relevant condition must be ``$F_\sigma$ purely unrectifiable.''
Against this background, the present paper achieves the following significant advances:

\begin{itemize}
	\item 
	In Theorem~\ref{thm:typical_nonsep}, we establish that inside any given set $S$ the set of non-differentiability points of a typical Lipschitz mapping is a residual subset of $S$, no matter  on what (bounded, normed) domain (containing $S$) the space of Lipschitz mappings is considered.
	This holds for vector-valued ($Y$-valued) mappings. The norms on $X$ and $Y$ are arbitrary, as long as $Y$ is Banach.
	Non-differentiability can be strengthened to extreme non-differentiability which still holds at residually many points.
	\item In Theorem~\ref{thm:typical_nondiff-dash} we determine that inside any given $F_{\sigma}$ purely unrectifiable set, a typical $Y$-valued Lipschitz mapping is nowhere differentiable in the extreme sense.  
\end{itemize}

Versions of these two results were known only in the special case of functions from $\R^d$ to $\R$, see~\cite[Theorem~4]{Loewen_Wang_typical_2000} and~\cite{dymond_maleva_2020}. We point out straight away that typical behaviour of scalar-valued Lipschitz functions has no direct implications for vector-valued mappings, even from $\R^d$ to $\R^l$ when $l>1$; see~\cite[Theorem~6.1]{dymond_maleva_2020}. 
The results of the present paper also go significantly beyond the premises of~\cite{dymond_maleva_2020} by
allowing arbitrary norms into consideration and replacing directional non-differentiability by extreme non-differentiability. 

This article further acts as an Erratum to~\cite[Remarks~2.9~and~3.18]{dymond_maleva_2020}. The authors hereby retract these two remarks, which are shown to be invalid by the present article; see Corollary~\ref{cor}. We further note that the content of those two remarks does not affect any of the results or indeed anything at all in the rest of the paper~\cite{dymond_maleva_2020}. 

\subsection{Main results.}\label{sec:typ}
Given a topological space $Z$, we say that a typical element of $Z$ possesses a certain property if the collection of those elements of $Z$ having that property forms a residual subset of $Z$. If $Z$ is a complete metric space, then its residual subsets are dense in $Z$, by the Baire Category theorem, hence a condition satisfied by a typical element is satisfied by elements of a dense $G_{\delta}$ subset of the space.

It is therefore important to note that the meaning of `typical' depends on the ambient space. We thus need to clarify, as we do in Section~\ref{sec:Lip}, whether a typical Lipschitz mapping is to be understood relative to the space of all Lipschitz mappings, or only those with Lipschitz constant bounded by $L$ for certain $L>0$; whether the mappings are defined on the whole space or on a certain subset; and what topology (or metric) is used on the space of Lipschitz mappings. 

We presently state the first main result of this paper; see Sections~\ref{sec2.1} and~\ref{sec:Lip} for detailed explanation of the notation involved. 
\begin{thm}\label{thm:typical_nonsep}
	Let $X$ be a normed space, $Y$ be a Banach space, $W$ be a separable subspace of $\mc{L}(X,Y)$, $Q$ be a bounded subset of $X$ and $\setG\subseteq \inter Q$. Then there is a residual subset $\mc{F}$ of $(\lipone(Q,Y),\lnorm{\infty}{\cdot})$ such that for every $f\in \mc{F}$ the set
	\begin{equation*}
	\mc{N}_{f,W}:=\set{\mb{x}\in \setG\colon \mc{D}_{f}(\mb{x})\supseteq \Balli{W}}
	\end{equation*}
	is residual in $\setG$.
\end{thm}
Here $\Balli{W}$ is the closed unit ball of $W$, $\mc{L}(X,Y)$ is the space of bounded linear operators $X\to Y$ and $\mc{D}_{f}(\mb{x})$ is the collection of those operators which, in a specific sense, behave like a derivative of $f$ at $\mb{x}$; for the precise definition of $\mc{D}_{f}(\mb{x})$ see~\eqref{eq:Dfx} below. The inclusion $\mc{D}_{f}(\mb{x})\supseteq \Balli{W}$ for non-trivial $W$ and points $\mb{x}\in \mc{N}_{f,W}$ in Theorem~\ref{thm:typical_nonsep} implies, in particular, that $\mc{N}_{f,W}$ is contained in the set of G\^ateaux non-differentiability points of $f$. However, this condition should be interpreted as a very strong form of non-differentiability. We elaborate on this presently.

A particularly strong failure of differentiability  of a mapping at a point, considered in~\cite[Theorem~1.9]{maleva_preiss2018}, happens when many different linear mappings simultaneously behave like a derivative of the mapping. Note that if $f$ is G\^ateaux differentiable then either $\mc{D}_{f}(\mb{x})$ is empty, or it is the singleton set containing only the G\^ateaux derivative of $f$; if $f$ is 
Fr\'echet differentiable at $\mb{x}$, then $\mc{D}_{f}(\mb{x})=\set{Df(\mb{x})}$.
Accordingly, the size of the set $\mc{D}_{f}(\mb{x})$ is a measure of the severity of non-differentiability of $f$ at $\mb{x}$. In the case when $\mc{L}(X,Y)$ is separable (for example, when $X$ is finite-dimensional and $Y$ is separable), the most extreme form of non-differentiability of a $1$-Lipschitz $f$ at $\mb{x}$ occurs if $\mc{D}_{f}(\mb{x})=\Balli{\mc{L}(X,Y)}$, the closed unit ball of $\mc{L}(X,Y)$. 

When $\mc{L}(X,Y)$ is non-separable, it is impossible to achieve $\mc{D}_f(\mb{x})=\Balli{\mc{L}(X,Y)}$, as we show, in Lemma~\ref{lemma:Df_separable} below, that $\mc{D}_{f}(\mb{x})$ is always separable. Qualitatively, the strongest form of non-differentiability of a $1$-Lipschitz $f$ that may hold in such case is $\mc{D}_{f}(\mb{x})\supseteq\Balli{W}$ for an infinite-dimensional, separable subspace $W$ of $\mc{L}(X,Y)$. For any such $W$, 
this is what we achieve for a typical $1$-Lipschitz $f$.

In light of~\cite[Theorem~2.2]{dymond_maleva_2020}, concerning non-differentiability of a typical (real-valued) Lipschitz function at every point of an arbitrary fixed $F_\sigma$ purely unrectifiable subset of $\R^{d}$, one asks if the conclusion of Theorem~\ref{thm:typical_nonsep}, where the settings are much more general, can be strengthened for such sets. We answer this for finite-dimensional domains in the affirmative in our second main result: \begin{restatable}{thm}{typicalnondiff}\label{thm:typical_nondiff-dash}
	Let $X$ be a finite-dimensional normed space, $Y$ be a Banach space, $W$ be a separable subspace of $\mc{L}(X,Y)$, $Q\subseteq X$ be bounded and $E\subseteq \inter(Q)$ be an $F_\sigma$ purely unrectifiable set. 
	Then $\mc{D}_{f}(\mb{x})\supseteq\Ball_{W}$ for a typical $f\in (\lipone(Q,Y),\lnorm{\infty}{\cdot})$ and every $\mb{x}\in E$.
\end{restatable}
Theorem~\ref{thm:typical_nondiff-dash} also strengthens previous results in this direction obtained in~\cite[Theorem~2.7]{dymond_maleva_2020} and in~\cite{merlo}: significant gains being that it caters for infinite-dimensional spaces $Y$, any norms on $X$ and $Y$, as long as $Y$ is Banach, and the derivative ratios `see' all possible linear operators $T\in W$ on arbitrarily small scales. 

\section{Preliminaries and Notation.}
\subsection{General notation and differentiability notions.}\label{sec2.1}
Given a normed vector space $X$, we let $\Balli{X}$ denote its closed unit ball and $\mathbb{S}_X$ its unit sphere. An open ball in $X$ with centre $\mb{x}$ and radius $r$ will be written as $B_{X}(\mb{x},r)$ and for closed balls we write $\cl{B}_{X}$ instead of $B_{X}$. The origin in $X$ will be denoted by $\mb{0}_{X}$. If $Y$ is an additional normed vector space, we let $\mc{L}(X,Y)$ denote the space of bounded linear operators $X\to Y$. The operator norm on $\mc{L}(X,Y)$ is denoted by $\opnorm{-}$. For a subset $Q$ of a topological space, we let $\inter Q$ denote the interior of $Q$. 
For a mapping $f\colon Q\subseteq X\to Y$ and $\mb{x}\in \inter Q$ we let 
\begin{equation}\label{eq:Dfx}
\mc{D}_{f}(\mb{x}):=\set{\op\in \mc{L}(X,Y)\colon \liminf_{r\to 0+}\sup_{\mb{u}\in\cl{B}(\mb{0}_{X},r)}\frac{\lnorm{Y}{f(\mb{x}+\mb{u})-f(\mb{x})-\op \mb{u}}}
	{r}=0}.	
\end{equation}
Observe that if $f$ is $1$-Lipschitz, we have $\mc{D}_{f}(\mb{x})\subseteq \Balli{\mc{L}(X,Y)}$ for every $\mb{x}\in \inter Q$.  

Finally, we will refer to a subset $\Gamma$ of a metric space $(M,d)$ as \emph{uniformly separated} if $\inf\set{d(\mb{x},\mb{y})\colon \mb{x},\mb{y}\in\Gamma,\,\mb{x}\neq \mb{y}}>0$. For such a set $\Gamma$ and $s>0$, we call $\Gamma$ \emph{$s$-separated} if $\inf\set{d(\mb{x},\mb{y})\colon \mb{x},\mb{y}\in\Gamma,\,\mb{x}\neq \mb{y}}\geq s$.

\subsection{Optimality of Theorem~\ref{thm:typical_nonsep} and comparison with previous results.}\label{sec:2.2}
Taking $W$ in Theorems~\ref{thm:typical_nonsep} and~\ref{thm:typical_nondiff-dash} as any non-trivial, separable subspace of $\mc{L}(X,Y)$ ensures that the set $\mc{N}_{f,W}$ is contained in the set of points of \Gat non-differenti\-a\-bility of $f$. Moreover, when $\mc{L}(X,Y)$ is itself separable, we may take $W=\mc{L}(X,Y)$. In this case we get that for a typical $f\in \lipone(Q,Y)$ the set $\mc{D}_{f}(\mb{x})$ is maximum possible, as it is equal to $\Balli{W}=\Balli{\mc{L}(X,Y)}$, at a typical point $\mb{x}$ of $\setG$ (and at every $\mb{x}\in E$ in the premises of Theorem~\ref{thm:typical_nondiff-dash}). In the following lemma we show that it is not possible to omit the separability condition on $W$ in Theorems~\ref{thm:typical_nonsep} and~\ref{thm:typical_nondiff-dash}. 

\begin{lemma}\label{lemma:Df_separable}
	Let $X$ and $Y$ be normed spaces, $Q\subseteq X$, $\mb{x}\in \inter Q$ and $f\colon Q\to Y$ be a mapping. Then the set $\mc{D}_{f}(\mb{x})$ is separable and closed in $(\mc{L}(X,Y),\opnorm{\cdot})$.
\end{lemma}
\begin{proof}
	Let $r>0$ be small enough so that $B_{X}(\mb{x},r)\subseteq Q$. For each rational $q\in \Q\cap (0,r)$ and each $n\in\N$ choose ${\op}_{q,n}\in{\mc{L}(X,Y)}$ such that
	\begin{multline*}
	\sup_{\mb{u}\in \cl{B}_{X}(\mb{0}_{X},q)}\frac{\lnorm{Y}{f(\mb{x}+\mb{u})-f(\mb{x})-{\op}_{q,n}\mb{u}}}{q}\\
	\leq \inf_{{\op}\in\mc{L}(X,Y)}\sup_{\mb{u}\in \cl{B}_{X}(\mb{0}_{X},q)}\frac{\lnorm{Y}{f(\mb{x}+\mb{u})-f(\mb{x})-{\op}\mb{u}}}{q}+\frac{1}{n}.
	\end{multline*}
	We show that $\mc{D}_{f}(\mb{x})\subseteq \cl{\set{T_{q,n}\colon q\in \Q\cap (0,r),\, n\in\N}}$, where the closure is taken with respect to the operator norm. 
	
	Indeed, consider arbitrary ${\op}_0\in \mc{D}_{f}(\mb{x})$ and $\varepsilon>0$. Let $n>3/\varepsilon$ and choose $q\in\Q\cap (0,r)$ so that
	\begin{equation*}
	\sup_{\mb{u}\in\cl{B}_{X}(\mb{0}_X,q)}\frac{\lnorm{Y}{f(\mb{x}+\mb{u})-f(\mb{x})-{\op}_0\mb{u}}}{q}\leq \frac{\varepsilon}{3}.
	\end{equation*}  
	Then, for every $\mb{u}\in \cl{B}_{X}(\mb{0}_X,q)$ we have
	\begin{align*}
	\lnorm{Y}{({\op}_{q,n}-{\op}_0)\mb{u}}&\leq \lnorm{Y}{{\op}_{q,n}\mb{u}+f(\mb{x})-f(\mb{x}+\mb{u})}+\lnorm{Y}{f(\mb{x}+\mb{u})-f(\mb{x})-{\op}_0\mb{u}}\\
	&\leq \frac{\varepsilon q}{3}+\frac{q}{n}+\frac{\varepsilon q}{3}\leq \varepsilon q,
	\end{align*}
	which implies $\opnorm{{\op}_{q,n}-{\op}_0}\leq \varepsilon$.
	
	To show that $\mc{D}_{f}(\mb{x})$ is closed, assume $T_k\in\mc{D}_{f}(\mb{x})$ converge in the operator norm to $T_0$. Fix an arbitrary $\varepsilon>0$ and choose $n\ge1$ and $0<\rho<\varepsilon$ such that $\opnorm{T_n-T_0}<\varepsilon/2$ and $\frac1\rho\lnorm{Y}{f(\mb{x}+\mb{u})-f(\mb{x})-T_n\mb{u}}<\varepsilon/2$ whenever $\lnorm{X}{\mb{u}}\le \rho$. Then for all $\mb{u}\in \cl B_X(\mb{0}_X,\rho)$
	\[
	\tfrac1\rho\lnorm{Y}{f(\mb{x}+\mb{u})-f(\mb{x})-T_0\mb{u}}\le
	\tfrac1\rho\lnorm{Y}{f(\mb{x}+\mb{u})-f(\mb{x})-T_n\mb{u}}+\opnorm{T_n-T_0}\tfrac{\lnorm{X}{\mb{u}}}{\rho}<\varepsilon.
	\]
	From the arbitrariness of $\varepsilon>0$ we conclude $T_0\in\mc{D}_f(\mb{x})$.	
\end{proof}

Let us record a simple comparison, in the case of real-valued $f$, of the set $\mc{D}_{f}(\mb{x})$ with the Dini subgradient $\hat{\partial}f(\mb{x})$ of $f$ at $\mb{x}$, considered in~\cite{Loewen_Wang_typical_2000}. Let $f\colon X\to \R$ be a function, $\mb{x}\in X$ and for each $\mb{v}\in X$ consider the lower Dini directional derivative
\begin{equation*}
f_{+}(\mb{x};\mb{v}):=\liminf_{t\to 0+}\frac{f(\mb{x}+t\mb{v})-f(\mb{x})}{t}.
\end{equation*}
The Dini subgradient of $f$ at $\mb{x}$ is then defined by
\begin{equation*}
\hat{\partial}f(\mb{x}):=\set{\mb{x}^{*}\in X^{*}\colon f_{+}(\mb{x};\mb{v})\geq \inner{\mb{x}^{*},\mb{v}}\,\forall \mb{v}\in X}.
\end{equation*}
\begin{lemma}\label{lemma:Dini}
	Let $X$ be a normed space, $f\colon X\to \R$ be a $1$-Lipschitz function, $\mb{x},\mb{v}\in X$ and $\mb{y}^{*},\mb{z}^{*}\in \mc{D}_{f}(\mb{x})$ 
	be such that
	$\mb{z}^{*}(\mb{v})<0<\mb{y}^{*}(\mb{v})$. Then $\hat{\partial}f(\mb{x})=\emptyset$.
\end{lemma}
\begin{proof}
	Observe that $\mb{z}^{*}\in\mc{D}_{f}(\mb{x})$ and $\mb{z}^{*}(\mb{v})<0$ implies $f_{+}(\mb{x};\mb{v})\leq \mb{z}^{*}(\mb{v})<0$. Similarly $\mb{y}^{*}\in\mc{D}_{f}(\mb{x})$ and $\mb{y}^{*}(-\mb{v})<0$ implies $f_{+}(\mb{x};-\mb{v})<0$. Thus, we have $f_{+}(\mb{x};\mb{v})<0$ and $f_{+}(\mb{x};-\mb{v})<0$, which implies $\hat{\partial}f(\mb{x})=\emptyset$.
\end{proof}
\begin{remark}
	
	Note that, by the Hahn-Banach theorem, given any $\mb{v}\in X\setminus\set{\mb{0}_X}$, we may choose the functionals $\mb{y}^{*},\mb{z}^*$ of norm $1$ such that $\mb{y}^*(\mb{v})=\lnorm{X}{\mb{v}}$ and $\mb{z}^*(\mb{v})=-\lnorm{X}{\mb{v}}$. Therefore, Lemma~\ref{lemma:Dini} may be applied whenever $\mc{D}_{f}(\mb{x})\supseteq \mathbb{S}_{X^{*}}$, the unit sphere of $X^*$.
	
	It can also be easily seen from the example of $f(\mb{x})=-\norm{\mb{x}}\colon X\to\R$ that a $1$-Lipschitz function may have
	$\mc{D}_f(\mb{0}_X)=\hat{\partial}f(\mb{0}_X)=\emptyset$. Hence, there is no reverse implication to the statement of Lemma~\ref{lemma:Dini}: emptiness of the Dini subgradient $\hat{\partial}f(\mb{x})$ does not imply any type of largeness of the set $\mc{D}_{f}(\mb{x})$.
	
	Thus Theorem~\ref{thm:typical_nonsep}, which proves $\mc{D}_{f}(\mb{x})\supseteq \Balli{X^{*}}\supseteq \mathbb{S}_{X^{*}}$ is stronger than~\cite[Theorem~4]{Loewen_Wang_typical_2000} even in the case when $X$ is finite-dimensional and $Y=\R$.
\end{remark}

\begin{remark}
	We also note that Theorem~\ref{thm:typical_nonsep} is stronger than results obtained in~\cite{Loewen_Wang_typical_2000} in another aspect. The paper~\cite{Loewen_Wang_typical_2000}
	requires $\setG$ to be equal to the whole space $X$ while Theorem~\ref{thm:typical_nonsep} allows any $E\subseteq\inter(Q)$. This is crucial in order to refute~\cite[Remarks~2.9~and~3.18]{dymond_maleva_2020} discussed in the Introduction. 
	We do that in the next corollary. 
\end{remark}
\begin{corollary}\label{cor}
	Let $d\ge1$ and $S\subseteq(0,1)^d$ be arbitrary.
	Then there is a residual subset $\mc{F}$ of $\lipone([0,1]^d,\R)$ such that for every $f\in \mc{F}$ the set
	of non-differentiability points of $f$ in $S$ is residual in $S$.	
\end{corollary}
\begin{proof}
	Apply Theorem~\ref{thm:typical_nonsep} to $X=\R^d$ equipped with the Euclidean norm, $Y=\R$, $W=\mc{L}(\R^d,\R)$, $Q=[0,1]^d$ and $\setG=S$.
\end{proof}

\subsection{Lipschitz mappings.}\label{sec:Lip}
Given a metric space $Q$ and a Banach space $Y$, we denote by $\lipone(Q,Y)$ the set of Lipschitz mappings $f\colon Q\to Y$ with $\lip(f)\leq 1$.  
If $Q$ is a bounded metric space, then $\lipone(Q,Y)$ is a closed subset of the Banach space $C_b(Q,Y)$ of $Y$-valued continuous bounded functions on $Q$, with the norm 
\begin{equation*}
\lnorm{\infty}{f}= \sup\set{\lnorm{Y}{f(\mb{x})}\colon \mb{x}\in Q}.
\end{equation*}
We require completeness of $(\lipone(Q,Y),\lnorm{\infty}{\cdot})$ in order for residual subsets of $\lipone(Q,Y)$ to be dense in $\lipone(Q,Y)$, see Section~\ref{sec:typ}.

Note that if $X$ is a normed space and $Q\subseteq X$ is not bounded, one could still consider the space $\lipone(Q,Y)$ 
as a complete metric space with metric
\[
\rho(f,g)=\sum_{n=1}^{\infty}2^{-n}
\min\set{1,\lnorm{\infty}{
		f|_{Q\cap n\Balli{X}}-g|_{Q\cap n\Balli{X}}
	}
}. 
\]
This is the approach chosen in~\cite{Loewen_Wang_typical_2000}. In the present work, we elect to work only with bounded $Q$ in order to be consistent with the papers~\cite{preiss_tiser94} and~\cite{dymond_maleva_2020}. However, we note that
the proofs given in the present paper may be easily 
modified to obtain the same results for $\lipone(Q,Y)$ in the case $Q$ is unbounded.
\subsection{The Banach-Mazur Game}\label{sec:BMgame}

To prove that a set is residual, i.e.\ a complement of the set of first Baire category, we will make use of the Banach-Mazur game, see~\cite[8.H]{kechris2012classical}.

Given a topological space $Z$ and its subset $H\subseteq Z$, the Banach-Mazur game in $Z$ with target $H$ is played by two players, Player~I and Player~II, as follows: The game starts by Player~I selecting a non-empty open subset $U_{1}$ of $Z$. Player~II must then respond by nominating a non-empty open subset $V_{1}$ of $Z$ with $V_{1}\subseteq U_{1}$. In the $k$-th round of the game, with $k\geq 2$, Player~I chooses a non-empty open set $U_{k}\subseteq V_{k-1}$ and Player~II returns a non-empty open set $V_{k}\subseteq U_{k}$. Thus, a run of the game is described by an infinite sequence of open sets
\begin{equation*}
U_{1}\supseteq V_{1}\supseteq U_{2}\supseteq V_{2}\supseteq \ldots\supseteq U_{k}\supseteq V_{k}\supseteq\ldots,
\end{equation*}
where the sets $U_{k}$ are the choices of Player~I and the sets $V_{k}$ are those of Player~II. Player~II wins the game if 
\begin{equation*}
\bigcap_{k\in\N}V_{k}\subseteq H.
\end{equation*}
Otherwise Player~I wins.

The Banach-Mazur game can be used to determine whether a subset of a topological space is residual. More precisely, for any non-empty topological space $Z$ and any subset $H$ of $Z$ it holds that $H$ is a residual subset of $Z$ if and only if Player~II has a winning strategy in the Banach Mazur game in $Z$ with target set $H$; see~\cite[Theorem~8.33]{kechris2012classical}.

In the case that $Z$ is a metric space (as will be the case in our setting), open balls may be used in place of the open sets $U_{k}$ and $V_{k}$ above, see also~\cite[Theorem~3.16]{dymond_maleva_2020}. Thus, the moves of Player~I and Player~II effectively become a choice of pairs $(\mb{x},r)$ where $\mb{x}\in Z$ prescribes the centre of the ball and $r>0$ the radius. In the special case when Player~II is always able to ensure that the intersection $\bigcap_{k\in\N} V_k=\bigcap_{k\in\N}B(y_k,s_k)$ of their choices is a singleton $y_0\in Z$, to conclude that Player~II wins it would be enough to verify $y_0\in H$ for any run of the game.

\section{A typical Lipschitz mapping is extremely non-differentiable at a typical point of a set}
In this section we prove Theorem~\ref{thm:typical_nonsep}. For the proof of subsequent auxiliary lemmata we follow the convention that the infimum of the empty set is $+\infty$. We also note that in any normed space a bounded non-empty subset has a non-empty boundary.

The next lemma is a generalisation of~\cite[Lemma~3.1]{mynewpaper} for normed spaces instead of convex sets.
\begin{lemma}\label{lemma:general}
	Let $X$ and $Z$  be normed spaces, $0<a<b$, and let $f_1,f_2\colon X\to Z$ be Lipschitz mappings such that $\lip(f_1)+\lip(f_2)\le1$ and $f_1(\mb{0}_{X})=f_2(\mb{0}_{X})=\mb{0}_{Z}$. 
	Then there exists a Lipschitz mapping $\Phi=\Phi(a,b,f_1,f_2)\colon X\to Z$ such that
	\begin{enumerate}[(i)]
		\item\label{Phi1a} $\Phi(\mb{x})=f_1(\mb{x})$ whenever $\lnorm{X}{\mb{x}}\le a$.
		\item\label{Phi2a} $\Phi(\mb{x})=f_2(\mb{x})$ whenever $\lnorm{X}{\mb{x}}\ge b$.
		\item\label{Phi3a} $\lip(\Phi)\leq 1+\frac{a}{b-a}$.
		\item \label{Phi5}
		If $f_1=\mb{0}_{Z}$ is the constant $\mb{0}_{Z}$ mapping, then 
		$\lnorm{Z}{\Phi(\mb{x})-f_2(\mb{x})}\le a\lip(f_2)$ for all $\mb{x}\in X$. 
		\item\label{Phi4a}
		If $f_2=\mb{0}_{Z}$ is the constant $\mb{0}_{Z}$ mapping, then
		$\lnorm{Z}{\Phi(\mb{x})}\leq b\lip(f_1)$ for all $\mb{x}\in X$.
	\end{enumerate}
	
\end{lemma}
\begin{proof}
	Define $\Phi\colon X\to Z$ by 
	\begin{equation*}
	\Phi(\mb{x})=\begin{cases}
	f_1(\mb{x}) & \text{ if }\lnorm{X}{\mb{x}}\le a ,\\
	\displaystyle\frac{b-\lnorm{X}{\mb{x}}}{b-a}f_1(\mb{x})+
	\frac{b\bigl(\lnorm{X}{\mb{x}}-a\bigr)}{\lnorm{X}{\mb{x}}(b-a)}f_2(\mb{x})
	& \text{ if }a<\lnorm{X}{\mb{x}}<b,\\
	f_2(\mb{x}) & \text{ if }\lnorm{X}{\mb{x}}\ge b.
	\end{cases}
	\end{equation*}
	Clearly, $\Phi$ satisfies~\eqref{Phi1a} and~\eqref{Phi2a}.
	Observe that $\Phi$ is a continuous mapping $X\to Z$. 
	Moreover, since $f_1,f_2\in\lipone$, in order to show that $\Phi$ is Lipschitz and to check~\eqref{Phi3a}, it is enough to verify 
	\begin{equation}\label{eq:Lip_bound}
	\lnorm{Z}{\Phi(\mb{y})-\Phi(\mb{x})}\le {\left(1+\frac{a}{b-a}\right)}\lnorm{X}{\mb{y}-\mb{x}}
	\end{equation}
	whenever $\lnorm{X}{\mb{x}},\lnorm{X}{\mb{y}}\in (a,b)$. To show this, fix such $\mb{x},\mb{y}\in X$ and note first that
	\begin{equation}\label{eq.Phi-b}
	\begin{aligned}
	\lnorm{Z}{\Phi(\mb{y})-\Phi(\mb{x})}&\le
	\lnorm{Z}{
		\tfrac{b-\lnorm{X}{\mb{y}}}{b-a}f_1(\mb{y})-
		\tfrac{b-\lnorm{X}{\mb{x}}}{b-a}f_1(\mb{x})
	}\\
	&\hphantom{AAAAAAA}+
	\lnorm{Z}{
		\tfrac{b(\lnorm{X}{\mb{y}}-a)}{\lnorm{X}{\mb{y}}(b-a)}f_2(\mb{y})
		-
		\tfrac{b(\lnorm{X}{\mb{x}}-a)}{\lnorm{X}{\mb{x}}(b-a)}f_2(\mb{x})
	}.
	\end{aligned}	
	\end{equation}	
	In several estimates that follow we will use that 
		\begin{equation}\label{eq:radial_lip_est}
		\lnorm{Z}{f_{i}(\mb{u})}\leq \lip(f_{i})\lnorm{X}{\mb{u}} \qquad\text{for all $\mb{u}\in X$ and $i=1,2$.}
		\end{equation}
		 This holds due to the condition $f_{i}(\mb{0}_{X})=\mb{0}_{Z}$ for $i=1,2$. Assuming without loss of generality that $\lnorm{X}{\mb{y}}\ge\lnorm{X}{\mb{x}}$, the first term of~\eqref{eq.Phi-b} is  bounded above by 
	\begin{align*}	
	&\frac{b-\lnorm{X}{\mb{y}}}{b-a}
	\lnorm{Z}{
		f_1(\mb{y})-f_1(\mb{x})}
	+
	\abs{\frac{b-\lnorm{X}{\mb{y}}}{b-a}-\frac{b-\lnorm{X}{\mb{x}}}{b-a}}
	\lnorm{Z}
	{f_1(\mb{x})}
	\\
	\\
	\le
	&\frac{b-\lnorm{X}{\mb{y}}}{b-a}
	\lip(f_1)\lnorm{X}{\mb{y}-\mb{x}}
	+
	\frac{\lnorm{X}{\mb{y}-\mb{x}}}{b-a}\lip(f_1)\lnorm{X}{\mb{x}}
	\\
	\\
	=
	&\frac{b-\lnorm{X}{\mb{y}}+\lnorm{X}{\mb{x}}}{b-a}\lip(f_1)\lnorm{X}{\mb{y}-\mb{x}}
	\le
	\left(1+\frac{a}{b-a}\right)\lip(f_1)\lnorm{X}{\mb{y}-\mb{x}}
	.
	\end{align*}	
	The second term of~\eqref{eq.Phi-b} is bounded above by
	\begin{align*}
	&\frac{b\bigl(\lnorm{X}{\mb{y}}-a\bigr)}{\lnorm{X}{\mb{y}}(b-a)}
	\lnorm{Z}{f_2(\mb{y})-f_2(\mb{x})}+
	\abs{\frac{b\bigl(\lnorm{X}{\mb{y}}-a\bigr)}{\lnorm{X}{\mb{y}}(b-a)}-
		\frac{b\bigl(\lnorm{X}{\mb{x}}-a\bigr)}{\lnorm{X}{\mb{x}}(b-a)}
	}\lnorm{Z}{f_2(\mb{x})}\\
	\le 
	&\frac{b\bigl(\lnorm{X}{\mb{y}}-a\bigr)}{\lnorm{X}{\mb{y}}(b-a)}
	\lip(f_2)\lnorm{X}{\mb{y}-\mb{x}}
	+\frac{ab}{\lnorm{X}{\mb{x}}\lnorm{X}{\mb{y}}(b-a)}\lnorm{X}{\mb{x}-\mb{y}}\lip(f_2)\lnorm{X}{\mb{x}}\\
	=
	&\left(1+\frac{a}{b-a}\right)\lip(f_2)\lnorm{X}{\mb{y}-\mb{x}}
	.
	\end{align*}
Summing the derived upper bounds for the two terms of~\eqref{eq.Phi-b} establishes~\eqref{eq:Lip_bound}.
	
	Finally, 
	if $f_1=\mb{0}_{Z}$, then 
	\begin{equation*}
		\lnorm{Z}{\Phi(\mb{x})-f_2(\mb{x})}	= \begin{cases}
				\lnorm{Z}{f_2(\mb{x})} & \text{if $\lnorm{X}{x}\leq a$},\\
				\frac{a(b-\lnorm{X}{\mb{x}})}{(b-a)\lnorm{X}{\mb{x}}}\lnorm{Z}{f_2(\mb{x})} & \text{if $a<\lnorm{X}{x}<b$},\\
				0 & \text{if $\lnorm{X}{x}\geq b$},
			\end{cases}
	\end{equation*}
 and if $f_2=\mb{0}_{Z}$, then 	\begin{equation*}
		\lnorm{Z}{\Phi(\mb{x})}=\begin{cases}
			\lnorm{Z}{f_1(\mb{x})} & \text{if $\lnorm{X}{x}\leq a$},\\
			\frac{b-\lnorm{X}{\mb{x}}}{b-a}\lnorm{Z}{f_1(\mb{x})} & \text{if $a<\lnorm{X}{x}<b$},\\
			0 & \text{if $\lnorm{X}{x}\geq b$}.
		\end{cases}
	\end{equation*}
Applying~\eqref{eq:radial_lip_est} to these formulae, we verify~\eqref{Phi5} and~\eqref{Phi4a}.
\end{proof}

The following lemma provides a construction which will be used to define a winning strategy for Player~II in the Banach-Mazur game in Lemma~\ref{lemma:BMgame1}. 
The property~\eqref{eq:g_formula} of $g$ ensures that this new $1$-Lipschitz mapping ``sees'' $L$ as its derivative in a small neighbourhood of the given set $\Gamma$.
\begin{lemma}\label{lemma:PlayerII_nonsep-1}
	Let $X$ and $Y$ be normed spaces and $Q\subseteq X$ be a bounded set with $\inter Q\neq \emptyset$.  Let $r\in (0,1)$, $L\in\mc{L}(X,Y)$ with $\opnorm{L}\leq 1-r$ and $f\in \lipone(Q,Y)$.  Let $\emptyset\ne\Gamma\subseteq \inter Q$ be a uniformly separated set with 
	\begin{equation}\label{eq:Gamma-cond-1}
	\inf_{\mb{x}\in\Gamma}\dist_{X}(\mb{x},\partial Q)>0. 
	\end{equation}
	Then there exist $\alpha\in(0,r)$ and $g\in\lipone(Q,Y)$ such that $\supnorm{g-f}<r$ and
	\begin{equation}\label{eq:g_formula}
	g(\mb{x}+\mb{u})=g(\mb{x})+L\mb{u}\qquad\text{for all }\mb{x}\in\Gamma\text{ and all }\mb{u}\in \cl{B}(\mb{0}_{X},\alpha).
	\end{equation}
\end{lemma}
\begin{proof}
	The approach we take to modify the mapping $f$ to arrive at $g$ is similar to that taken in~\cite[Lemma~3.3]{mynewpaper}.
	
	The conclusion of this lemma is valid for $f$ if and only if it is valid for any mapping of the form $f+p$, where $p\colon Q\to Y$ is a constant mapping. Therefore, we may assume that $\mb{0}_{Y}\in f(\Gamma)$. Lipschitz mappings $h\colon Q\to Y$ with the property $\mb{0}_{Y}\in h(Q)$ satisfy $\lnorm{Y}{h(\mb{x})}\leq \lip(h)\diam Q$ for all $\mb{x}\in Q$. This fact will be used later in the proof.
	
	Fix $s\in (0,1)$ small enough so that $\Gamma$ is $4s$-separated and the infimum of~\eqref{eq:Gamma-cond-1} is at least $4s$. Let 
	\begin{equation}\label{eq:beta0}
	\beta=\beta(r,s,Q)\in (0,s/2)
	\end{equation}
	be a parameter depending only on $r, s$ and $Q$ which will be determined at the end of the proof in~\eqref{eq:alpha_beta}. We define first a mapping $g_{0}\colon Q\to Y$ by
	\begin{equation*}
	g_{0}(\mb{z})=\begin{cases}
	f(\mb{z}) & \text{ if }\mb{z}\in Q\setminus \bigcup_{\mb{x}\in\Gamma}B_{X}(\mb{x},s),\\
	f(\mb{x}+\Phi(\mb{z}-\mb{x})) & \text{ if }\mb{z}\in B_{X}(\mb{x},s)\text{ and }\mb{x}\in\Gamma,
	\end{cases}
	\end{equation*}
	where $\Phi:=\Phi(\beta,s,\mb{0}_{X},\id_X)\colon X\to X$ is the mapping given by Lemma~\ref{lemma:general} applied to $X$, $Z=X$, $a=\beta$, $b=s$, $f_{1}=\mb{0}_{X}$ (the constant mapping $X\to X$ with value $\mb{0}_{X}$) and $f_{2}=\id_{X}$. Here we used Lemma~\ref{lemma:general}\eqref{Phi4a} to conclude $\lnorm{X}{\Phi(\mb{z}-\mb{x})}\le\beta<s$, so that $\mb{x}+\Phi(\mb{z}-\mb{x})\in Q$ whenever $\mb{x}\in\Gamma$ and $\mb{z}\in B_{X}(\mb{x},s)$. Using again $\beta<s$ and Lemma~\ref{lemma:general}\eqref{Phi1a}, we note that
	\begin{equation}\label{eq:g0_const}
	g_{0}(\mb{z})=g_{0}(\mb{x})=f(\mb{x})\qquad \text{whenever }
	\mb{x}\in \Gamma\text{ and }\mb{z}\in \cl{B}_{X}(\mb{x},\beta).
	\end{equation}
	It follows, in particular,
	\[
	\mb{0}_{Y}\in f(\Gamma)=g_{0}(\Gamma).
	\]
	Further, we note that Lemma~\ref{lemma:general}\eqref{Phi3a} and~\eqref{Phi5} imply $g_0$ is Lipschitz,
	\begin{equation*}
	\lip(g_{0})\leq 1+\frac{\beta}{s-\beta}
	\qquad 
	\text{and}
	\qquad 
	\supnorm{g_{0}-f}\leq \beta.
	\end{equation*}

	Let $T=T(r,s,Q,L)\in\mc{L}(X,Y)$ be a linear operator 
	with $\opnorm{T}\leq 1$. This operator will be used in construction of the target $1$-Lipschitz mapping $g$ such that~\eqref{eq:g_formula} is satisfied with a multiple of $T$ instead of $L$; this will determine how $T$ is defined, see~\eqref{eq:choice_T}. The choice of $T$ depends on $L,s$ and $\beta=\beta(r,s,Q)$, and we note that $L,Q,r,s$ are fixed from the start. Next we let $\alpha=\alpha(r,s,Q)\in (0,\beta)$ be a further parameter to be determined later in the proof in~\eqref{eq:alpha_beta} and define $g_{1}\colon Q\to Y$ by
	\begin{equation*}
	g_{1}(\mb{z})=\begin{cases}
			g_{0}(\mb{z}), & \text{ if }\mb{z}\in Q\setminus \bigcup_{\mb{x}\in\Gamma}B_{X}(\mb{x},\beta),\\
			g_0({\mb{x}}) +T(\Psi(\mb{z}-\mb{x})), & \text{ if }\mb{z}\in B_{X}(\mb{x},\beta)\text{ and }\mb{x}\in \Gamma,
	\end{cases}
\end{equation*}
	where $\Psi:=\Phi(\alpha,\beta,\id_{X},\mb{0}_{X})\colon X\to X$ is the mapping given by Lemma~\ref{lemma:general} applied to $X$, $Z=X$, $a=\alpha$, $b=\beta$, $f_{1}=\id_{X}$ and $f_{2}=\mb{0}_{X}$. The properties of $g_0$, $\Psi$ and~\eqref{eq:g0_const} ensure that $g_{1}$ is Lipschitz. We may estimate its Lipschitz constant as
	\begin{multline*}
	\lip(g_{1})\leq \max\set{\lip(g_{0}),1+\frac{\alpha}{\beta-\alpha}}\\
	\leq\max\set{1+\frac{\beta}{s-\beta}, 1+\frac{\alpha}{\beta-\alpha}}\leq 1+\frac{\beta}{s-\beta}
	=\frac{s}{s-\beta},
	\end{multline*}
	where the penultimate inequality is achieved by imposing the condition
	\begin{equation}\label{eq:alpha_condition}
	\alpha\leq \frac{\beta^{2}}{s}.
	\end{equation} 
	Moreover, we have 
	\begin{equation*}
	\supnorm{g_{1}-g_{0}}\leq \beta\qquad\text{and}\qquad \mb{0}_{Y}\in g_{0}(\Gamma)=g_{1}(\Gamma).
	\end{equation*}
	To verify the former, it is enough to observe, using~\eqref{eq:g0_const}, $\opnorm{T}\leq 1$, $\lip(\id_X)=1$ and Lemma~\ref{lemma:general}\eqref{Phi4a}, that for any $\mb{z}\in B_{X}(\mb{x},\beta)$ with $\mb{x}\in \Gamma$ 
	\begin{multline*}
	\lnorm{Y}{g_1(\mb{z})-g_0(\mb{z})}
	=\lnorm{Y}{g_0(\mb{x})+T(\Psi(\mb{z}-\mb{x}))-g_0(\mb{z})}\\
	=\lnorm{Y}{T(\Psi(\mb{z}-\mb{x}))}
	\leq \lnorm{X}{\Psi(\mb{z}-\mb{x})}
	\leq  \beta.
	\end{multline*}
	Finally we set 
	\begin{equation*}
	g=\frac{s-\beta}{s}\cdot g_{1},
	\end{equation*}
	so that $g\in\lipone (Q,Y)$ and
	\begin{equation*}
	\supnorm{g-g_{1}}\leq \frac{\beta \lip(g_{1})\diam Q}{s}\leq \frac{\beta \diam Q}{s-\beta}\leq \frac{2\beta \diam Q}{s},
	\end{equation*}
	using $\beta<s/2$ from~\eqref{eq:beta0}. We conclude that
	\begin{equation*}
	\supnorm{g-f}\leq \supnorm{g-g_{1}}+\supnorm{g_{1}-g_{0}}+\supnorm{g_{0}-f}
	\leq \frac{2\beta\diam Q}{s}+2\beta\leq \frac{4\beta(1+\diam Q)}{s}.
	\end{equation*}
	Thus, we achieve $\supnorm{g-f}\leq r$ by imposing the condition
	\begin{equation}\label{eq:beta_condition1}
	\beta \leq \frac{rs}{4(1+\diam Q)}.
	\end{equation}
	We are now ready to make the choice of linear operator $T=T(r,s,Q,L)\in \mc{L}(X,Y)$ with $\opnorm{T}\leq 1$. Indeed, the choice 
	\begin{equation}\label{eq:choice_T}
	T=\frac{s}{s-\beta}L, 
	\end{equation}
	establishes~\eqref{eq:g_formula}. 
	We note that the condition~\eqref{eq:beta_condition1} imposed on $\beta$ implies $\beta\leq rs$, which together with $\opnorm{L}\leq 1-r$ gives $\opnorm{T}\leq 1$. 
	
	It only remains to note that the choices
	\begin{equation}\label{eq:alpha_beta}
	\beta =\frac{rs}{4(1+\diam Q)},\qquad \alpha=\frac{r^{2}s}{16(1+\diam Q)^{2}}
	\end{equation}
	satisfy the required conditions~\eqref{eq:beta0}, \eqref{eq:alpha_condition} and~\eqref{eq:beta_condition1}.
\end{proof}

\begin{lemma}\label{lemma:Gamma1}
	Let $X$ be a normed space, $Q$ be a bounded subset of $X$ and $\setG\subseteq \inter Q$. Then there exists a sequence $(\Gamma_{k})_{k\in\N}$ of nested sets $\Gamma_k\subseteq \Gamma_{k+1}\subseteq \setG$ such that the union $\bigcup_{k\ge 1}\Gamma_k$ is dense in $\setG$ and each set $\Gamma_{k}$ satisfies the hypothesis of Lemma~\ref{lemma:PlayerII_nonsep-1}, that is, 
	$\Gamma_k$ satisfies~\eqref{eq:Gamma-cond-1} and is $\delta_k$-separated for some $\delta_k>0$.
\end{lemma}
\begin{proof}
	If $\setG=\emptyset$, let $\Gamma_k=\emptyset$ for all $k\in\mathbb N$.
	
	Assume $\setG\ne\emptyset$.
	Let $\setG_k=\set{\mb{x}\in \setG\colon\dist_{X}(\mb{x},\partial Q)\geq 2^{-k}}$. 
	Since $\setG\subseteq\inter Q$, we have that $\bigcup_{k\ge1}\setG_k=\setG$. Let $n\ge1$ be the smallest index such that $\setG_n\ne\emptyset$.  Set $\Gamma_k=\emptyset$ for any $0\le k\le n-1$. For any $k\ge n$, let us make an inductive choice of $\Gamma_k\supseteq \Gamma_{k-1}$ to
	be a non-empty maximal $2^{-k}$-separated subset of $\setG_k$. 
	Since for any $k\ge n$ the set $\Gamma_k\ne\emptyset$ is a $2^{-k}$-net of $\setG_k$,  we  conclude that $\bigcup_{k\ge n}\Gamma_k=\bigcup_{k\ge 1}\Gamma_k$ is dense in $\setG$.	
\end{proof}
The following lemma is the final step allowing us to prove Theorem~\ref{thm:typical_nonsep}. It shows that every bounded linear operator $L$ with $\opnorm{L}<1$ behaves like a derivative of a typical $1$-Lipschitz function, at a typical point of $\setG$.
\begin{lemma}\label{lemma:BMgame1}
	Let $X$ be a normed space and $Y$ be a Banach space, $Q$ be a bounded subset of $X$ with non-empty interior, $\setG\subseteq \inter Q$ and $\op\in \mc{L}(X,Y)$ with $\opnorm{\op}<1$. Then there is a residual subset $\mc{H}_{\op}$ of $(\lipone(Q,Y),\lnorm{\infty}{\cdot})$ such that for every $f\in \mc{H}_{\op}$ the set
	\begin{equation*}
	\mc{P}_{\op,f}:=\set{\mb{x}\in \setG\colon \op\in \mc{D}_{f}(\mb{x})}
	\end{equation*}
	is residual in $\setG$, where $\mc{D}_{f}(\mb{x})$ is defined according to~\eqref{eq:Dfx}.
\end{lemma}
\begin{proof}
We assume that $\setG\ne \emptyset$.	
Let 
	\begin{equation*}
	\mc{H}_{\op}:=\set{f\in\lipone(Q,Y)\colon \text{ the set }\set{\mb{x}\in \setG\colon \op\in \mc{D}_{f}(\mb{x})}\text{ is residual in $\setG$}}.
	\end{equation*}
	We prove that the set $\mc{H}_{\op}$ is residual in $\lipone(Q,Y)$ by describing a winning strategy for Player~II in the relevant Banach-Mazur game in $(\lipone(Q,Y),\lnorm{\infty}{\cdot})$ with the target $\mc{H}_\op$,
	in which Player I's choices are balls $B(f_k,r_k)$ and Player II’s choices
	are balls $B(g_k,s_k)$; see subsection~\ref{sec:BMgame} for details on the Banach-Mazur game.  Here and throughout the proof, given a mapping $\phi\in\lipone(Q,Y)$ and $\rho>0$ we abbreviate the notation $B_{\lipone(Q,Y)}(\phi,\rho)$, for the open ball in the metric space $(\lipone(Q,Y),\lnorm{\infty}{\cdot})$ with centre $\phi$ and radius $\rho$, to $B(\phi,\rho)$. 
	
	Before the game starts, let Player~II prepare by fixing a nested sequence $(\Gamma_{k})_{k\in \N}$ of sets $\Gamma_k\subseteq \Gamma_{k+1}\subseteq \setG$, given by Lemma~\ref{lemma:Gamma1}. 
	
	Let $k\in\N$, assume that $k-1$ rounds of the game have already completed, giving $f_i$, $r_i$, $g_i$ and $s_i$ for $i\leq k-1$ and let $f_{k}\in\lip_{1}(Q,Y)$ and $r_{k}>0$ denote the $k$-th move of Player~I. 
	Since nothing prevents Player~II from acting as if the radius $r_{k}$ was replaced
	by a smaller radius $\tilde r_{k}>0$,  
	we may assume that
	\begin{equation}\label{eq:xk-fk}
	r_{k}\leq 2^{-k}(1-\opnorm{\op}),
	\quad\text{in addition to}\quad 
\cl{B}(f_{k},r_{k})\subseteq B(g_{k-1},s_{k-1})\text{ if $k\geq 2$.}
\end{equation}
In order to define their response, Player~II applies Lemma~\ref{lemma:PlayerII_nonsep-1} to find a mapping $g_{k}\in \lipone(Q,Y)$ and $\alpha_{k}\in(0,r_k)$ satisfying $\supnorm{g_k-f_k}<r_{k}$ and
	\begin{equation}\label{eq:gk_formula-1}
	g_{k}(\mb{x}+\mb{u})=
	g_k(\mb{x})+\op\mb{u},\qquad  \text{whenever } 
	\mb{x}\in\Gamma_k\,\text{ and }\mb{u}\in\cl{B}(\mb{0}_{X},\alpha_{k}).\\
	\end{equation}
	Finally, Player~II chooses 
	\begin{equation}\label{eq:sk1}
	0<s_{k}<\alpha_k/(4k)
	\end{equation}
	small enough so that 
	\begin{equation}\label{eq:sk2}
	\cl{B}(g_{k},s_{k})\subseteq B(f_{k},r_{k}) 
	\end{equation}
	and declares $g_{k}\in\lipone(Q,Y)$ and $s_{k}>0$ as their $k$-th move.
	
	Due to the conditions~\eqref{eq:xk-fk}  and~\eqref{eq:sk2}, the intersection
	\begin{equation*}
	\bigcap_{k=1}^{\infty}B(f_{k},r_{k})=\bigcap_{k=1}^{\infty}B(g_{k},s_{k})
	\end{equation*}
	is a singleton set containing only the Lipschitz mapping $g:=\lim_{k\to \infty}g_{k}\in\lipone(Q,Y)$.
		
	To complete the proof, we show that Player~II wins the game, that is, that $g\in \mc{H}_{T}$, see subsection~\ref{sec:BMgame}. Consider the sequence $U_{k}:=\bigcup_{\mb{x}\in\Gamma_{k}} B_{X}(\mb{x},s_{k})$ of open sets in $X$
	and the set $J:=\setG\cap \bigcap_{n=1}^{\infty}\bigcup_{k=n}^{\infty}U_{k}\subseteq\setG$. Clearly, $J$ is a relatively $G_{\delta}$ subset of $\setG$. Moreover, for each $n\ge1$, $\bigcup_{k\ge n}U_{k}\supseteq \bigcup_{k\ge n}\Gamma_{k}=\bigcup_{k\ge 1}\Gamma_{k}$, as $\Gamma_k$ are nested, and the latter is a dense subset of $\setG$ by Lemma~\ref{lemma:Gamma1}; thus $J\supseteq \bigcup_{k\ge 1}\Gamma_{k}$ is dense in $\setG$.
	We conclude that $J$ is a relatively residual subset of $\setG$.
	
	To prove $g\in\mathcal{H}_T$, we verify $J\subseteq\{\mb{x}\in E\colon T\in D_g({x})\}$. Let $\mb{x}\in J$ and $\varepsilon>0$. Choose $k\in\N$ with $k\geq {1}/{\varepsilon}$ such that $\mb{x}\in U_{k}$ and $\alpha_{k}<\varepsilon$. Let $\mb{x}_{k}\in\Gamma_{k}$ be such that $\mb{x}\in B_{X}(\mb{x}_{k},s_{k})$; let  $\mb{u}\in\cl{B}_{X}(\mb{0}_{X},\alpha_{k})$ be arbitrary. Then, applying~\eqref{eq:gk_formula-1}, we get
	$g_{k}(\mb{x}_{k}+\mb{u})=g_{k}(\mb{x}_{k})+\op\mb{u}$. 
	Using this identity, we derive
	\begin{multline*}
	\lnorm{Y}{g(\mb{x}+\mb{u})-g(\mb{x})-\op\mb{u}}\leq\lnorm{Y}{g(\mb{x}+\mb{u})-g_{k}(\mb{x}+\mb{u})}+\lnorm{Y}{g_{k}(\mb{x}+\mb{u})-g_{k}(\mb{x}_{k}+\mb{u})}\\
	+\lnorm{Y}{g_{k}(\mb{x}_{k}+\mb{u})-g_{k}(\mb{x}_{k})-\op\mb{u}}+\lnorm{Y}{g_{k}(\mb{x}_{k})-g_{k}(\mb{x})}+\lnorm{Y}{g_{k}(\mb{x})-g(\mb{x})}\\
	\leq  2\supnorm{g_{k}-g}+2\lnorm{X}{\mb{x}_{k}-\mb{x}}+0\leq 4
	s_k\leq \frac{\alpha_{k}}{k},
	\end{multline*}	
	where the last inequality is due to Player~II's choice~\eqref{eq:sk1} of $s_{k}$. This argument verifies
	\[
	\sup_{\mb{u}\in\cl{B}_{X}(\mb{0}_X,\alpha_{k})}\frac{\lnorm{Y}{g(\mb{x} + \mb{u})-g(\mb{x})-\op\mb{u}}}{\alpha_k}
	\le\frac1k\leq \varepsilon
	.
	\]
	and subsequently $\op\in\mc{D}_{g}(\mb{x})$. 
\end{proof}
We are now ready to prove Theorem~\ref{thm:typical_nonsep}.
\begin{proof}[Proof of Theorem~\ref{thm:typical_nonsep}]
	Fix a dense sequence $(\op_{n})_{n\in\N}$ in the closed unit ball $\Balli{W}$ with $\opnorm{\op_{n}}<1$ for all $n\in\N$. 
	Let the sets $\mc{H}_{\op_{n}}\subseteq \lipone(Q,Y)$ be given by the conclusion of Lemma~\ref{lemma:BMgame1}.
	Define a residual subset $\mc{F}=\bigcap_{n\in\N}\mc{H}_{\op_{n}}$ of $\lipone(Q,Y)$ 	
	and let $f\in\mc{F}$ be arbitrary.
	Let the sets $\mc{P}_{\op_{n},f}$ be given by the conclusion of Lemma~\ref{lemma:BMgame1}, 
	consider the residual subset $\mc{P}_f=\bigcap_{n\in\N}\mc{P}_{\op_{n},f}$ of $\setG$
	and let $\mb{x}\in\mc{P}_f$ be arbitrary. Since $T_n\in \mc{D}_{f}(\mb{x})$ for all $n\ge1$ and	
	$\mc{D}_{f}(\mb{x})$ is closed in $(\mc{L}(X,Y),\opnorm{\cdot})$, by Lemma~\ref{lemma:Df_separable},  we have $\mc{D}_{f}(\mb{x})\supseteq \overline{\set{T_n\colon n\in\N}}^{\opnorm{\cdot}}= \Balli{W}$.
\end{proof}

\section{Sets in which a typical Lipschitz mapping is everywhere extremely non-differentiable}\label{sec:nondiff-pu}
In this section we prove our second main result, Theorem~\ref{thm:typical_nondiff-dash}. Some of the proofs which appear in this section follow the scheme employed in~\cite[Sections~2,3]{maleva_preiss2018}, yet a lot of intricate work is required to make the arguments work in the much more general situation where the norm on $X$ is no longer Euclidean, the domain is not the whole space and the mappings are $Y$-valued for an arbitrary Banach space~$Y$.

\begin{lemma}\label{lemma:one_direction_steep_function}
	Let $X$ be a normed space, $G\subseteq X$, $P\in X^{*}$ be a norm-attaining functional, $\vp\in \Sph_{X}$ be such that $P(\vp)=\lnorm{X^{*}}{P}$
	and $\alpha\in (0,1)$. Suppose that the quantity
	\begin{equation*}
		\begin{split}
			\xi(G,P,\alpha)&:=\sup\left\{\mc{H}^{1}\br{G\cap \gamma(\R)}\colon \gamma\in \lip(\R,X),\right.\\
			&\qquad\qquad\qquad \left.P\br{\gamma'(t)}\geq \alpha\lnorm{X}{\gamma'(t)}\lnorm{X^{*}}{P}\text{ whenever $\gamma'(t)$ exists}\right\}
		\end{split}
	\end{equation*}
	is finite. Then there exists a function $g\colon X\to\R$ such that
	\begin{enumerate}[(i)]
		\item\label{gsmall} $0\leq g(\mb{x})\leq \lnorm{X^{*}}{P}\xi(G,P,\alpha)$ for all $\mb{x}\in X$.
		\item\label{gperp} $\abs{g(\mb{x}+\mb{y})-g(\mb{x})}\leq \frac{\alpha\lnorm{X^{*}}{P}}{1-\alpha}\lnorm{X}{\mb{y}}$ for all $\mb{x}\in X$ and $\mb{y}\in \ker P$.
		\item\label{gdecay} For every pair $\mb{x},\mb{w}\in X$ there exists $\lambda=\lambda(\mb{x},\mb{w})\in[0,1]$ such that 
		\begin{equation*}
			\abs{g(\mb{x}+\mb{w})-(g(\mb{x})+\lambda P(\mb{w}))}\leq \frac{2\alpha \lnorm{X^{*}}{P}}{1-\alpha}\lnorm{X}{\mb{w}}.
		\end{equation*}
		Moreover, $\lambda(\mb{x},\mb{w})=1$ if either $P=0$ or $P\neq 0$ and $\sqbr{\mb{x},\mb{x}+\frac{P(\mb{w})}{\lnorm{X^{*}}{P}}\vp}\subseteq G$.
		\item\label{glip}\label{last}
		The function $g\colon X\to\R$ is $\br{1+\frac{2\alpha}{1-\alpha}}\lnorm{X^*}{P}$-Lipschitz. 
	\end{enumerate}
\end{lemma}
\begin{proof}
	If $P=0$ we may take $g$ as the constant zero function $X\to \R$. So we may assume $P\neq 0$.
	
	Observe that if $g\colon X\to\R$ satisfies conditions~\eqref{gsmall}--\eqref{last} for the functional $P=\frac{Q}{\lnorm{X^{*}}{Q}}$, for some $Q\in X^{*}\setminus \set{0}$, then the function $\lnorm{X^{*}}{Q}g\colon X\to \R$ satisfies the conditions~\eqref{gsmall}--\eqref{last} for $P=Q$. Hence, we may assume that $\lnorm{X^{*}}{P}=1$.
	
	Define $g\colon X\to \R$ by
	\begin{equation}\label{eq:defg}
		\begin{split}
			g(\mb{x}) &= \sup\left\{ \mc{H}^{1} \br{G\cap \gamma((-\infty,b])}-s\colon \gamma\in \lip(\R,X),\, b\in\R,\,s\geq 0,\,\right. \\
			&\left. \gamma(b)=\mb{x}+s \vp,\, P\br{\gamma'(t)}\geq \alpha\lnorm{X}{\gamma'(t)}\text{ whenever $\gamma'(t)$ exists} \right\}
		\end{split}
	\end{equation}
	We will now check that $g$ satisfies~\eqref{gsmall}--\eqref{last} and, in addition, the following two properties:
	\begin{enumerate}[(A)]
		\item\label{ggengrow} $g(\mb{x})\leq g(\mb{x}+r\vp)\leq g(\mb{x})+r$ whenever $\mb{x}\in X$ and $r\geq 0$.  
		\item\label{ggrow} $g(\mb{x}')-g(\mb{x})=t$ whenever  $\mb{x}'-\mb{x}=t\vp$, $t\in\R$ and $[\mb{x},\mb{x}']\subseteq G$.
	\end{enumerate}
	
	The fact that $g$ is well-defined and the first inequality of~\eqref{gsmall} are witnessed by the triple $\gamma(t)=\mb{x}$ for all $t\in\R$, $b=0$ and $s=0$. The second inequality of~\eqref{gsmall} is immediate from $s\ge0$ for any admissible triple $(\gamma,b,s)$ in the definition of $g(\mb{x})$. 
	
	In the proof of remaining parts we will use the following notation. If $\mb{z}\in X$ and $\eta>0$ then by $(\gamma_{\mb{z}},s_{\mb{z}},b_{\mb{z}})$ we denote an admissible triple in the sense of~\eqref{eq:defg}, i.e.\ such that $\gamma_{\mb{z}}\in\lip(\R,X)$, $b_{\mb{z}}\in\R$, $s_{\mb{z}}\ge0$, $\gamma_{\mb{z}}(b_{\mb{z}})=\mb{z}+s_{\mb{z}}\vp$, $P(\gamma_{\mb{z}}'(t))\ge\alpha\lnorm{X}{\gamma_{\mb{z}}'(t)}$ whenever $\gamma_{\mb{z}}'(t)$ exists, with the additional property that 
	\begin{equation}\label{eq:approx_sup}
		\mc{H}^{1}\br{G\cap \gamma_{\mb{z}}\bigl((-\infty,b_{\mb{z}}]\bigr)}-s_{\mb{z}}>g(\mb{z})-\eta.
	\end{equation}
	If, in addition, $\mb{u}\in\Sph_{X}$ and $P(\mb{u})\ge\alpha$,  we define 
	\begin{equation}\label{eq:curve_mod}
		\gamma_{\mb{z},\mb{u}}(t)
		=\begin{cases}
			\gamma_{\mb{z}}(t) & \text{if }t\leq b_{\mb{z}},\\
			\gamma_{\mb{z}}(b_{\mb{z}})+(t-b_{\mb{z}})\mb{u} & \text{if }t>b_{\mb{z}}.
		\end{cases}
	\end{equation}
	In both cases we suppress $\eta$ in the notation although the objects we define also depend on $\eta$. 	Note that $\gamma_{\mb{z},\mb{u}}(t)\in\lip(\R,X)$ and $P(\gamma_{\mb{z},\mb{u}}'(t))\ge
	\alpha\lnorm{X}{\gamma_{\mb{z},\mb{u}}'(t)}$ whenever $\gamma_{\mb{z},\mb{u}}'(t)$ exists. In particular, we have that
	$P\circ\gamma_{\mb{z},\mb{u}}$ is monotone increasing and hence 
	\begin{equation}\label{eq:curve_hausd}
		\gamma_{\mb{z},\mb{u}}\br{(-\infty,b_{\mb{z}}]}\cap\gamma_{\mb{z},\mb{u}}\br{[b_{\mb{z}},+\infty)}=\{\gamma_{\mb{z},\mb{u}}(b_{\mb{z}})\} 
		.
	\end{equation}
	
	To prove~\eqref{gperp}, consider an arbitrary $\mb{x}\in X$. Since~\eqref{gperp} is trivially satisfied for $\mb{y}=\mb{0}_X$, assume $\mb{y}\in \ker P\setminus\{\mb{0}_X\}$. Fix any $\eta\in(0,1-\alpha)$, an admissible triple $(\gamma_{\mb{x}},s_{\mb{x}},b_{\mb{x}})$
	and $\beta\in (0,\infty)$ so that $\mb{u}:={\frac{\alpha}{1-\eta} \vp+\beta \frac{\mb{y}}{\lnorm{X}{\mb{y}}}}\in\Sph_{X}$ noting
	\begin{equation}\label{eq:beta}
		\beta \geq 1-\frac{\alpha}{1-\eta}.
	\end{equation}		
	Consider the mapping $\indi{\gamma}=\gamma_{\mb{x},\mb{u}}$, as defined in~\eqref{eq:curve_mod}, and let 
	\begin{equation*}
		\indi{s}:=s_{\mb{x}}+\frac{\alpha\lnorm{X}{\mb{y}}}{(1-\eta)\beta},\qquad \indi{b}:=b_{\mb{x}}+\frac{\lnorm{X}{\mb{y}}}{\beta}.
	\end{equation*}
	Observe that $\indi{s}\geq 0$, $\indi\gamma(\indi{b})=\gamma_{\mb{x}}(b_{\mb{x}})+\frac{\alpha\lnorm{X}{\mb{y}} }{\beta(1-\eta)}\vp+\mb{y}
	=
	\mb{x}+\mb{y}+s_1\vp$ and $P(\mb{u})=\frac{\alpha}{1-\eta}>\alpha$, thus
	we conclude that
	\begin{multline*}
		g(\mb{x}+\mb{y})\geq \mc{H}^{1}\br{G\cap\indi{\gamma}\br{(-\infty,\indi{b}]}}-\indi{s}\\
		\geq \mc{H}^{1}\br{G\cap\gamma_{\mb{x}}\br{(-\infty,b_{\mb{x}}]}}-s_{\mb{x}}-\frac{\alpha\lnorm{X}{\mb{y}}}{(1-\eta)\beta}>g(\mb{x})-\eta-\frac{\alpha\lnorm{X}{\mb{y}}}{1-\eta-\alpha},
	\end{multline*}
	using~\eqref{eq:approx_sup} and~\eqref{eq:beta} for the last inequality.
	Letting $\eta\to 0$, we obtain $g(\mb{x}+\mb{y})\geq g(\mb{x})-\frac{\alpha\lnorm{X}{\mb{y}}}{1-\alpha}$ and applying the above argument to the pair $\mb{\wt{x}}:=\mb{x}+\mb{y}$ and $\mb{\wt{y}}:=-\mb{y}$ in place of $\mb{x}$ and $\mb{y}$ delivers the reverse inequality $g(\mb{x})\geq g(\mb{x}+\mb{y})-\frac{\alpha\lnorm{X}{\mb{y}}}{1-\alpha}$. This finishes the proof of~\eqref{gperp}.
	
	We now turn our attention to the two additional properties. For~\eqref{ggengrow}, fix $\mb{x}\in X$ and $r\geq 0$. Let $\eta>0$ be arbitrary, and consider again $(\gamma_{\mb{x}},b_{\mb{x}},s_{\mb{x}})$; let  $\indii{\gamma}=\gamma_{\mb{x},\vp}$ be defined by~\eqref{eq:curve_mod}. 
	
	If $s_{\mb{x}}>r\ge0$, we note $\gamma_{\mb{x}}(b_{\mb{x}})=(\mb{x}+r\vp)+
	(s_{\mb{x}}-r)\vp$, so by~\eqref{eq:approx_sup}
	\begin{equation}\label{eq:pfC_s_more_r}
		g(\mb{x}+r\vp)\geq 
		\mc{H}^{1}\br{G\cap\gamma_{\mb{x}}((-\infty,b_{\mb{x}}])}-(s_{\mb{x}}-r)>g(\mb{x})-\eta+r\ge g(\mb{x})-\eta.
	\end{equation}
	If $0\le s_{\mb{x}}\leq r$, consider 
	$\indii{b}=b_{\mb{x}}+(r-s_{\mb{x}})\ge b_{\mb{x}}$ and $\indii{s}=0$ to deduce
	\begin{equation}\label{eq:gamma2}
		\indii{\gamma}(\indii{b})
		=\gamma_{\mb{x}}(b_{\mb{x}})+(r-s_{\mb{x}})\vp
		=\mb{x}+s_{\mb{x}}\vp+(r-s_{\mb{x}})\vp
		=\mb{x}+r\vp+\indii{s}\vp.
	\end{equation}
	Therefore by~\eqref{eq:approx_sup}
	\begin{multline}\label{eq:pfC_s_less_r}
		g(\mb{x}+r\vp)\geq \mc{H}^{1}\br{G\cap \indii{\gamma}((-\infty,\indii{b}])}-\indii{s}=\mc{H}^{1}\br{G\cap \indii{\gamma}((-\infty,\indii{b}])}
		\\
		\geq \mc{H}^{1}\br{G\cap \indii\gamma((-\infty,{b}_{\mb{x}}])}-s_{\mb{x}}
		=
		\mc{H}^{1}\br{G\cap \gamma_{\mb{x}}((-\infty,{b}_{\mb{x}}])}-s_{\mb{x}}\ge g(\mb{x})-\eta.
	\end{multline}
	In either case, letting $\eta\to 0$, we establish the first inequality of~\eqref{ggengrow}. 
	To prove the second inequality of~\eqref{ggengrow}, let $\eta>0$ be arbitrary and consider
	$(\indiii{\gamma},\indiii{s},\indiii{b})=
	(\gamma_{\mb{x}+r\vp},s_{\mb{x}+r\vp},b_{\mb{x}+r\vp})$.
	We observe 	 $\indiii{\gamma}(\indiii{b})=\mb{x}+(r+\indiii{s})\vp$, so using~\eqref{eq:approx_sup} we get 
	\begin{equation*}
		g(\mb{x})\geq \mc{H}^{1}\br{G\cap \indiii{\gamma}((-\infty,\indiii{b}])}-(r+\indiii{s})
		>g(\mb{x}+r\vp)-\eta-r,
	\end{equation*}
	which implies the second inequality of~\eqref{ggengrow} when we let $\eta\to 0$. 
		
	Assume $\mb{x},\mb{x}'\in X$ satisfy the conditions of~\eqref{ggrow}.  
	We may assume without loss of generality that $\mb{x}'-\mb{x}=r\vp$, where $r\geq 0$. In light of~\eqref{ggengrow}, it is now enough to show $g(\mb{x}+r\vp)\ge g(\mb{x})+r$.
	Let us again fix an arbitrary $\eta>0$ and consider $(\gamma_{\mb{x}},b_{\mb{x}},s_{\mb{x}})$, $\indii{\gamma}=\gamma_{\mb{x},\vp}$ and $b_2=b_{\mb{x}}+(r-s_{\mb{x}})$. If $s_{\mb{x}}>r$, then~\eqref{eq:pfC_s_more_r} applies to give the desired inequality in the limit $\eta\to0$. If $0\le s_{\mb{x}}\le r$, we may improve~\eqref{eq:pfC_s_less_r}. Indeed, 
	using the first line of~\eqref{eq:pfC_s_less_r} for the first inequality, followed by $\indii\gamma\br{(b_{\mb{x}},\indii{b}]}=[\mb{x}+s_{\mb{x}}\vp,\mb{x}+r\vp]\subseteq[\mb{x},\mb{x}+r\vp]\subseteq G$ and~\eqref{eq:curve_hausd} for the equality, we conclude that
	\begin{equation*}
		g(\mb{x}+r\vp)
		\geq \mc{H}^{1}\br{G\cap \indii{\gamma}((-\infty,\indii{b}])}
		=\mc{H}^{1}\br{G\cap \gamma_{\mb{x}}((-\infty,{b}_{\mb{x}}])}	+(r-s_{\mb{x}})	
		>g(\mb{x})-\eta+r,
	\end{equation*}	
	which implies $g(\mb{x}+r\vp)\geq g(\mb{x})+r$ if we let $\eta\to0$.	
	
	To prove part~\eqref{gdecay}, assume that $\mb{x},\mb{w}\in X$ are given,
	consider $\mb{y}:=\mb{w}-P(\mb{w})\vp\in \ker P$ and note that $\lnorm{X}{\mb{y}}\leq 2\lnorm{X}{\mb{w}}$.
	If $P(\mb{w})=0$, then
	$\mb{w}=\mb{y}$ and, setting $\lambda(\mb{x},\mb{w}):=1$, we have that the inequality of~\eqref{gdecay} immediately follows from~\eqref{gperp}. We may therefore assume $P(\mb{w})\ne0$.  
	In this case we let
	\begin{equation*}
		\lambda(\mb{x},\mb{w}):=
		\frac{g(\mb{x}+P(\mb{w})\vp)-g(\mb{x})}{P(\mb{w})},
	\end{equation*}
	and observe that $\lambda(\mb{x},\mb{w})\in [0,1]$, as $g|_{\mb{x}+\R\vp}$ is a $1$-Lipschitz increasing function by~\eqref{ggengrow}, with $\lambda(\mb{x},\mb{w})=1$ if $\sqbr{\mb{x},\mb{x}+\frac{P(\mb{w})}{\lnorm{X^{*}}{P}}\vp}\subseteq G$, by~\eqref{ggrow}. We now use~\eqref{gperp} to complete verification of~\eqref{gdecay}:
	\begin{multline*}
		\abs{g(\mb{x}+\mb{w})-g(\mb{x})-\lambda(\mb{x},\mb{w}) P(\mb{w})}
		= 
		\abs{g(\mb{x}+\mb{w})-g(\mb{x}+P(\mb{w})\vp)}
		\\
		\leq \frac{\alpha}{1-\alpha}\lnorm{X}{\mb{y}}\leq \frac{2\alpha}{1-\alpha}\lnorm{X}{\mb{w}}.
	\end{multline*}
	
	To establish~\eqref{glip}, we note that~\eqref{gdecay} implies that for any $\mb{x},\mb{w}\in X$ \begin{equation*}
		\abs{g(\mb{x}+\mb{w})-g(\mb{x})}
		\leq
		\frac{2\alpha}{1-\alpha}\lnorm{X}{\mb{w}}+\abs{\lambda} \lnorm{X}{\mb{w}}
		\leq
		\br{1+\frac{2\alpha}{1-\alpha}}\lnorm{X}{\mb{w}}. 
		\qedhere	\end{equation*}
\end{proof}

The next lemma says that for a given compact purely unrectifiable subset $E$ of a finite-dimensional normed space $X$, and a given bounded linear operator $T$ on $X$ there exists a Lipschitz $g\colon X\to T(X)$ which, on a neighbourhood of the set $E$, has derivative approximately equal to $T$ and everywhere outside this neighbourhood has derivative zero. In other words, there are Lipschitz mappings which are constant except on a small neighbourhood of the given compact purely unrectifiable set, where their derivative is approximately whatever bounded linear operator you wish. It is natural to ask what is the optimal Lipschitz constant with which such a mapping $g\colon X\to T(X)$ can be found. The optimal result that could be hoped for is clearly $\lip(g)\leq \opnorm{T}+\theta$ for an arbitrarily small error term $\theta>0$ and indeed, in all `classical' settings (for example, when the norm on $T(X)$ is Euclidean), this is what Lemma~\ref{lemma:pu_mapping_with_prescribed_derivative} achieves; see Remark~\ref{rem:cyl_constant}. However, in the general setting, we identify a constant $\cyl(T)\geq \opnorm{T}$, so that the Lipschitz constant of $g$ may be bounded above by $\cyl(T)+\theta$. It would be of interest to determine whether this upper bound is optimal.

\begin{define}\label{def:cyl_constant}
	Let $X$ and $Y$ be normed vector spaces and $T\in\mc{L}(X,Y)$ be a nonzero bounded linear operator of finite rank $l$. We associate to $T$ the constant
	\begin{equation*}
	\cyl(T):=\min\set{\max_{1\leq j\leq l}\opnorm{\sum_{i=1}^{j}\mb{w}_{i}^{*}\circ T(\cdot)\,\mb{w}_{i}}\colon \mb{w}_{1},\ldots,\mb{w}_{l}\text{ is a basis of $T(X)$}}.
	\end{equation*}
	In the above $\mb{w}_{1}^{*},\ldots,\mb{w}_{l}^{*}\in T(X)^{*}$ is the basis of $T(X)^*$ dual to the basis $\mb{w}_{1},\ldots,\mb{w}_{l}$ of $T(X)$ so that, in particular, $T=\sum_{i=1}^{l}\mb{w}_{i}^{*}\circ T(\cdot)\,\mb{w}_{i}$. We also define $\cyl(\mb{0}_{\mc{L}(X,Y)})=0$.
\end{define}
\begin{remark}\label{rem:cyl_constant}
	\begin{enumerate}[(i)]
		\item \label{rem:cyl_constant:well}
		The constant $\cyl(T)$ is well-defined, that is, the minimum of the above set exists, due to Lemma~\ref{lemma:cyl_constant}.
		\item \label{rem:cyl_constant:ineq}
		The identity  $T=\sum_{i=1}^{l}(\mb{w}_{i}^{*}\circ T)(\cdot)\,\mb{w}_{i}$ for any choice of $\mb{w}_{1},\ldots,\mb{w}_{l}$ and $\mb{w}_{1}^{*},\ldots,\mb{w}_{l}^{*}$ as above implies $\cyl(T)\ge\opnorm{T}$.
		\item For any normed space $X$ and $(Y,\lnorm{Y}{\cdot})=(\ell_p,\lnorm{q}{\cdot})$, $1\le p\le q\le\infty$ (in particular, Hilbert and finite-dimensional Euclidean), and for every $T\in\mc{L}(X,Y)$, one has
		$\cyl(T)=\|T\|$. However for any $Y$ of dimension at least $3$ there is a norm $\lnorm{Y}{\cdot}$ so that for $X=Y$ and $\lnorm{X}{\cdot}=\lnorm{Y}{\cdot}$ one has $\cyl(\id_X)>1$.
	\end{enumerate}
\end{remark}
\begin{lemma}\label{lemma:pu_mapping_with_prescribed_derivative}
	Let $X$ and $Y$ be normed spaces, where $X$ is finite-dimensional.
	Let  $E\subseteq U\subseteq X$ be sets, where $E$ is compact and purely unrectifiable and $U$ is open, $\theta>0$, $T\in\mc{L}(X,Y)$. Then there exist a Lipschitz mapping $g\colon X\to T(X)$ and an open subset $H$ of $X$ such that the following statements hold:
	\begin{enumerate}[(a)]
		\item\label{support_in_U}\label{der_zero_outside} $\supp g\subseteq U$ and $Dg(\mb{x})=0$ for all $\mb{x}\in X\setminus U$. 
		\item\label{small} $\sup_{\mb{x}\in X}\lnorm{Y}{g(\mb{x})}\leq \theta$.
		\item\label{H_in_between_E_and_U} $E\subseteq H\subseteq \cl{H}\subseteq U$.
		\item\label{Dg_close_to_T} $\opnorm{Dg(\mb{x})-T}\leq\theta$ for Lebesgue almost all $\mb{x}\in H$.	
		\item\label{lipTglobal} $\lip(g)\leq \cyl(T)+\theta$, where the constant $\cyl(T)$ is given by Definition~\ref{def:cyl_constant}.
	\end{enumerate}
\end{lemma}
\begin{proof}
	The statement of the lemma is clear if $T=0$, so assume, without loss of generality, that $T\ne0$. 
	Also, although the proof below will work independently of whether $E$ is an empty or a non-empty set, it might be worth mentioning that in case $E=\emptyset$, it would be enough to take $H=\emptyset$ and $g\equiv0$.
	
	Let $1\le l\leq \dim X$ denote the rank of $T$ and $\mb{w}_{1},\ldots,\mb{w}_{l}\in T(X)\subseteq Y$ be a basis of $T(X)$ for which \begin{equation*}\max_{1\leq i\leq l}\opnorm{\sum_{i=1}^{j}\mb{w}_{i}^{*}\circ T(\cdot)\,\mb{w}_{i}}= \cyl(T),
	\end{equation*}
	where we adopt the notation of Definition~\ref{def:cyl_constant}; let $T_{i}=\mb{w}_{i}^{*}\circ T\in X^*$ for each $1\le i\le l$, so that 
\begin{equation}\label{eq:T}
	\sum_{i=1}^l T_i(\cdot)\mb{w}_i=T(\cdot).
\end{equation}
Let $U_0$ be an open set given by Remark~\ref{rem:lip-loc_glob}, satisfying $E\subseteq U_0\subseteq \cl{U_0}\subseteq U$ and $\partial U_0$ has Lebesgue measure zero. For each $i=1,\ldots,l$ we will construct sequences of smooth functions $\varphi_{k}\upp{i}\colon X\to \R$, positive numbers $\varepsilon_{k}\upp{i}$, sets $G_{k}\upp{i}\subseteq X$ and Lipschitz functions $g_{k}\upp{i}\colon X\to \R$ respectively, as well as positive integers $K_{i}\in\N$ and open sets $U_{i}\subseteq X$  such that the following conditions hold for each $i=1,2,\ldots,l$:
	\begin{enumerate}[(A)]
		\item\label{varphiik} $(\varphi_{k}\upp{i})_{k\in\N}$ is a smooth, locally finite partition of unity with supports contained in $U_{i-1}$.
		\item\label{epsik} For each $k\ge1$, we have $\varepsilon_k\upp{i}\in(0,1)$ and
		\begin{equation*}
		\sum_{k\in\N}\frac{\varepsilon_{k}\upp{i}\br{1+\lip(\varphi_{k}\upp{i})}}{1-\varepsilon_{k}\upp{i}}
		\leq
		\frac{\theta}{4l\br{1+{\lnorm{X^*}{T_{i}}}
			}\br{1+\lnorm{Y}{\mb{w}_{i}}}}.
		\end{equation*}
		\item\label{Gik} 
		For each $k\ge1$ the set $G_{k}\upp{i}$ is open and satisfies 
		\begin{equation*}
		E\subseteq G_{k}\upp{i}\subseteq U_{i-1}
		\end{equation*}
		and
		\begin{equation*}
		\begin{split}
		&\sup\left\{\mc{H}^{1}\br{G\upp{i}_{k}\cap \gamma(\R)}\colon \gamma\in \lip(\R,X),\,\right.\\ &\qquad\qquad\qquad \left.T_{i}\br{\gamma'(t)}\geq \varepsilon_{k}\upp{i}\lnorm{X}{\gamma'(t)}\lnorm{X^{*}}{T_{i}}\text{ whenever $\gamma'(t)$ exists}\right\}\leq \varepsilon_{k}\upp{i}.
		\end{split}
		\end{equation*}
		\item\label{gik} For each $k\geq 1$ the function $g_{k}\upp{i}\colon X\to\R$ satisfies
		the following conditions:
		\begin{enumerate}[(D1)]
			\item\label{gik-small}
			$g\upp{i}_{k}$ is Lipschitz and $0\leq g\upp{i}_{k}(\mb{x})\leq \lnorm{X^*}{T_i}\varepsilon_{k}\upp{i}$ for all $\mb{x}\in X$,
			
			\item\label{gik-lam} 
			For every $\mb{x},\mb{w}\in X$ and every $k\ge1$ there exists $\lambda_{k}\upp{i}=\lambda_{k}\upp{i}(\mb{x},\mb{w})\in [0,1]$ such that
			\begin{equation*}
			\abs{g_{k}\upp{i}(\mb{x}+\mb{w})-g_{k}\upp{i}(\mb{x})-\lambda_{k}\upp{i} T_{i}(\mb{w})}\leq \frac{\const\varepsilon_{k}\upp{i}\lnorm{X^{*}}{T_{i}}}{1-\varepsilon_{k}\upp{i}}\lnorm{X}{\mb{w}}.
			\end{equation*}
			\item\label{gik-lam-one} 	
			Whenever $B(\mb{x},r)\subseteq G_{k}\upp{i}$ and  $\lnorm{X}{\mb{w}}<r$, the inequality~\eqref{gik-lam} is satisfied with $\lambda_{k}\upp{i}(\mb{x},\mb{w})=1$.		
		\end{enumerate}	
		\item\label{KiUi} $K_{i}\in\N$ and $U_{i}$ is an open subset of $X$ such that $\partial U_{i}$ has Lebesgue measure zero,
		\begin{equation*}
		\supp(\varphi_{k}\upp{i})\cap U_i=\emptyset
		\quad\text{for all $k>K_i$}\qquad\text{and}\qquad
		E\subseteq U_{i}\subseteq\cl{U_{i}}\subseteq U_{i-1}\cap \bigcap_{k=1}^{K_{i}}G_{k}\upp{i}.
		\end{equation*}
	\end{enumerate}
	Suppose that $1\le j\le l$ and that we have constructed the above listed objects of levels $i=1,\ldots,j-1$ such that conditions~\eqref{varphiik}--\eqref{KiUi} are satisfied for $i=1,\ldots,j-1$. In the case $j=1$ no objects are yet constructed and all conditions are vacuous and therefore satisfied.
	
	We then proceed to construct the objects of level $j$ as follows: First we choose the sequences $(\varphi_{k}\upp{j})_{k\in\N}$, $(\varepsilon_{k}\upp{j})_{k\in\N}$ and $(G_{k}\upp{j})_{k\in\N}$, in that order, arbitrarily subject to the conditions~\eqref{varphiik}, \eqref{epsik} and~\eqref{Gik} respectively for $i=j$. To choose $G_{k}\upp{j}$ as in~\eqref{Gik} we are using that compact purely unrectifiable sets are uniformly purely unrectifiable. To make this more precise, the existence of the sets $G_{k}\upp{j}$ is given by a result of Alberti, Marchese~\cite{alberti2016differentiability}; see also Theorem~\ref{thm:compact_pu} in Appendix~\ref{sec:misc}. If $1\leq j\leq l$ with $T_{j}=0$, we let $g_{k}\upp{j}\colon X\to\R$ be defined as the constant $0$ function for every $k\in\N$. Observe that all parts of~\eqref{gik} are then trivially satisfied for such $j$. For the remaining $1\leq j\leq l$, those with $T_{j}\neq 0$, 
	we let $g_{k}\upp{j}\colon X\to\R$ for each $k\in\N$ be given by the conclusion of Lemma~\ref{lemma:one_direction_steep_function} for $G=G_{k}\upp{j}$, $P=T_{j}$ and $\mb{v}_{T_j}=\mb{v}_{j}\in \Sph_{X}$ is any vector satisfying the condition $T_{j}(\mb{v}_{j})=\lnorm{X^{*}}{T_{j}}$, 
	and $\alpha=\varepsilon_{k}\upp{j}$. Then Lemma~\ref{lemma:one_direction_steep_function} provides all of the stated properties in~\eqref{gik} for $i=j$. 
	In particular $\lambda_{k}\upp{j}(\mb{x},\mb{w})$ in~\eqref{gik-lam} can be chosen as $\lambda(\mb{x},\mb{w})$ from Lemma~\ref{lemma:one_direction_steep_function}\eqref{gdecay}. Note that $\lambda_{k}\upp{j}(\mb{x},\mb{w})=1$ if the additional conditions of~\eqref{gik-lam-one} are satisfied, since $\lnorm{X}{\frac{T_i(\mb{w})}{\lnorm{X^*}{T_i}}\mb{v}_{j}}\le\lnorm{X}{\mb{w}}$ for all $\mb{w}\in X$. Finally, we let $K_{j}$ and $U_{j}$ with $\partial U_{j}$ of Lebesgue measure $0$ be the pair given by the conclusion of Lemma~\ref{lemma:commpact_partition_unity} applied to $X$, $E$, $V=U_{j-1}$, $(G_{k})_{k\in\N}:=(G_{k}\upp{j})_{k\in\N}$ and $(\varphi_{k})_{k\in\N}:=(\varphi_{k}\upp{j})_{k\in\N}$. Such choice of $K_{j}$ and $U_{j}$ ensures that~\eqref{KiUi} is satisfied for $i=j$. This completes the construction of all above mentioned objects for levels $i=1,\ldots,l$ so that conditions~\eqref{varphiik}--\eqref{KiUi} are satisfied for $i=1,\ldots,l$.	
	
	We define the mapping $g\colon X\to T(X)$ by
	\begin{equation}\label{eq:def_g}
	g(\mb{x})=\sum_{i=1}^{l}\sum_{k\in\N}\varphi_{k}\upp{i}(\mb{x})g_{k}\upp{i}(\mb{x})\mb{w}_{i}
	\end{equation}
	and put $H:=U_{l}$. Then we have 
	\begin{equation}\label{eq:supp}
	\supp g \subseteq U_0\subseteq U\qquad\text{and}\qquad Dg(\mb{x})=0 \quad\text{for all $\mb{x}\in X\setminus \cl{U_{0}}\supseteq X\setminus U$.}
	\end{equation}
	due to~\eqref{varphiik} and the fact, coming from~\eqref{KiUi}, that all $U_{i}$'s are contained in $U_{0}\subseteq\cl{U_{0}}\subseteq U$. Moreover,
	\begin{equation*}
	\lnorm{Y}{g(\mb{x})}\leq \sum_{i=1}^{l}\sum_{k\in\N}\abs{g_{k}\upp{i}(\mb{x})}\lnorm{Y}{\mb{w}_{i}}\leq\sum_{i=1}^{l}\sum_{k\in\N}\lnorm{X^*}{T_i}\varepsilon_{k}\upp{i}\lnorm{Y}{\mb{w}_{i}}\leq  \theta
	\end{equation*}
	by the inequalities~\eqref{gik-small} and~\eqref{epsik}. This establishes~\eqref{support_in_U} and~\eqref{small}. Part~\eqref{H_in_between_E_and_U} is clear from the choice $H=U_{l}$, \eqref{KiUi} and $\cl{U_0}\subseteq U$.
	
	In order to show that $g$ is Lipschitz, we argue first that $g$ is a locally Lipschitz mapping. Fix an arbitrary $\mb{x}\in X$. For each $i=1,\dots,l$ the collection $(\varphi_{k}\upp{i})_{k\in\N}$ forms a locally finite partition of unity, hence there is an open ball $B=B(\mb{x},r)$ and an index $n\in\N$ such that $\sum_{k\in\N}\varphi_{k}\upp{i}(\mb{y})g_{k}\upp{i}(\mb{y})=\sum_{k=1}^{n}\varphi_{k}\upp{i}(\mb{y})g_{k}\upp{i}(\mb{y})
	$ for all $\mb{y}\in B$ and all $1\le i\le l$. Since by~\eqref{varphiik} all $\varphi_{k}\upp{i}$ are Lipschitz and bounded on $B$, and by~\eqref{gik-small} each $g_{k}\upp{i}$ is a Lipschitz bounded function as well, we conclude that $g|_B$ is Lipschitz too. We have thus established that $g$ is locally Lipschitz on $X$.

	We now derive  bounds on the norm in $Y$ of vectors of the form $g(\mb{x}+\mb{z})-g(\mb{x})$, aiming to get an estimate with the Lipschitz constant given in~\eqref{lipTglobal}. We will approximate this vector closely with an appropriate linear mapping evaluated at $\mb{z}$. The appropriate linear mapping to use will be determined by which sets in the nested sequence $U_{0}\supseteq U_{1}\supseteq \ldots\supseteq U_{l}=H$ contain the segment $[\mb{x},\mb{x}+\mb{z}]$.
	
	For any $\mb{x},\mb{z}\in X$ and any $\lambda\in\R$ we may use~\eqref{gik-small} and~\eqref{varphiik} to write
	\begin{align}\label{eq:slope_bound_g}
	&\abs{\br{\varphi_{k}\upp{i}(\mb{x}+\mb{z})g_{k}\upp{i}(\mb{x}+\mb{z})-\varphi_{k}\upp{i}(\mb{x})g_{k}\upp{i}(\mb{x})}-\lambda\varphi_{k}\upp{i}(\mb{x})T_{i}(\mb{z})}\notag\\
	&\leq \abs{\varphi_{k}\upp{i}(\mb{x})}\abs{\br{g_{k}\upp{i}(\mb{x}+\mb{z})-g_{k}\upp{i}(\mb{x})}-\lambda T_{i}(\mb{z})}+\abs{g_{k}\upp{i}(\mb{x}+\mb{z})}\abs{\varphi_{k}\upp{i}(\mb{x}+\mb{z})-\varphi_{k}\upp{i}(\mb{x})}\notag\\
	&\leq 
	\abs{\br{g_{k}\upp{i}(\mb{x}+\mb{z})-g_{k}\upp{i}(\mb{x})}-\lambda T_{i}(\mb{z})}+\lnorm{X^*}{T_i}\varepsilon\upp{i}_{k}\lip(\varphi\upp{i}_{k})\lnorm{X}{\mb{z}}
	.
	\end{align}
	Observe that the set $U_0\setminus \bigcup_{j=1}^{l}\partial U_{j}$ may be partitioned as a union of open sets
	\begin{equation*}
	U_0\setminus \bigcup_{j=1}^{l}\partial U_{j}=H\cup\bigcup_{j=1}^{l}(U_{j-1}\setminus \cl{U_{j}}).
	\end{equation*}
	Therefore, to estimate the Lipschitz constant of $g$ locally at every point of this set, we may consider an arbitrary $\mb{x}\in U_0\setminus \bigcup_{j=1}^{l}\partial U_{j}$ and distinguish two cases:
	\begin{equation}\label{eq:cases}
	\mb{x}\in U(\mb{x}):=H=U_{l}\quad \text{ or }\quad \mb{x}\in U(\mb{x}):= U_{j-1}\setminus \cl{U_{j}}\text{ for some $j\in\set{1,\ldots,l}$.}
	\end{equation}
Let $r=r(\mb{x})>0$ be sufficiently small so that $B_{X}(\mb{x},r)\subseteq U(\mb{x})$ and let $\mb{z}\in B_{X}(\mb{0}_X,r)$ be arbitrary. In the former case of~\eqref{eq:cases}, we have $\mb{x}\in B_X(\mb{x},r)\subseteq H\subseteq U_i\subseteq G_k\upp{i}$, for every $i=1,\ldots,l$ and $1\le k\le K_i$, by~\eqref{KiUi}, so we may apply~\eqref{gik-lam-one} and~\eqref{eq:slope_bound_g} with $\lambda=1$ to derive 	
	\begin{multline*}
	\abs{\br{\varphi_{k}\upp{i}(\mb{x}+\mb{z})g_{k}\upp{i}(\mb{x}+\mb{z})-\varphi_{k}\upp{i}(\mb{x})g_{k}\upp{i}(\mb{x})}-\varphi_{k}\upp{i}(\mb{x})T_{i}(\mb{z})}\\
	\leq\br{\frac{\const\varepsilon_{k}\upp{i}\lnorm{X^{*}}{T_{i}}}{1-\varepsilon_{k}\upp{i}}+\lnorm{X^*}{T_i}\varepsilon_{k}\upp{i}\lip(\varphi_{k}\upp{i})}\lnorm{X}{\mb{z}}
	\text{ for every }1\le i\le l\text{ and }1\le k\le K_i
	.
	\end{multline*}
	We also recall that for every $1\le i\le l$ and $k>K_i$, ~\eqref{KiUi} implies
	\begin{equation*}
	\abs{\br{\varphi_{k}\upp{i}(\mb{x}+\mb{z})g_{k}\upp{i}(\mb{x}+\mb{z})-\varphi_{k}\upp{i}(\mb{x})g_{k}\upp{i}(\mb{x})}-\varphi_{k}\upp{i}(\mb{x})T_{i}(\mb{z})}=0
	,
	\end{equation*}
	since both $\mb{x},\mb{x}+\mb{z}\in H\subseteq U_i$.
	Multiplying the expressions under the modulus sign on the left-hand side by $\mb{w}_i$ and summing over respective ranges for $k\ge1$ and then over $i=1,\ldots,l$ we obtain, according to~\eqref{eq:def_g}, \eqref{eq:T} and $\sum_{k=1}^\infty\varphi_{k}\upp{i}(\mb{x})=1$,
	\begin{align*}
	\lnorm{Y}{g(\mb{x}+\mb{z})-g(\mb{x})-T(\mb{z})}&\leq\sum_{i=1}^{l}\sum_{k=1}^{K_i}\br{\frac{\const\varepsilon_{k}\upp{i}\lnorm{X^{*}}{T_{i}}}{1-\varepsilon_{k}\upp{i}}+\lnorm{X^*}{T_i}\varepsilon_{k}\upp{i}\lip(\varphi_{k}\upp{i})}\lnorm{Y}{\mb{w}_{i}}\lnorm{X}{\mb{z}}\\
	&\leq \theta\lnorm{X}{\mb{z}},
	\end{align*}
	where the last inequality holds due to~\eqref{epsik}. Using that $g\colon X\to T(X)$ is a locally Lipschitz mapping between finite-dimensional spaces, and hence $Dg(\mb{u})$ exists for almost all $\mb{u}\in X$, and the fact that the
	inequality above has been obtained for an arbitrary pair of $\mb{x}\in H=U_{l}$ and $\mb{z}\in B_{X}(\mb{0}_X,r)$, we establish~\eqref{Dg_close_to_T}. We also note that the last inequality implies
	\begin{equation}\label{eq:lip_g_H}
	\lnorm{Y}{g(\mb{x}+\mb{z})-g(\mb{x})}\leq \br{\opnorm{T}+\theta}\lnorm{X}{\mb{z}}
	\end{equation}
	in the first case from~\eqref{eq:cases}.
	
	In the remaining case from~\eqref{eq:cases}, $\mb{x}\in U_{j-1}\setminus \cl{U_{j}}$, $1\le j\le l$, note that $\mb{x},\mb{x}+\mb{z}\in B_{X}(\mb{x},r)\subseteq U_{j-1}\setminus \cl{U_{j}}$. Then, by~\eqref{varphiik} and~\eqref{KiUi}, we have that $\varphi\upp{i}_{k}(\mb{x})=\varphi\upp{i}_{k}(\mb{x}+\mb{z})=0$ for all $i\geq j+1$ and $k\ge1$. Therefore,
	\begin{equation}\label{eq:ibeyondj}
	\abs{\br{\varphi_{k}\upp{i}(\mb{x}+\mb{z})g_{k}\upp{i}(\mb{x}+\mb{z})-\varphi_{k}\upp{i}(\mb{x})g_{k}\upp{i}(\mb{x})}}=0 \qquad\text{for all $i\geq j+1$ and $k\ge1$.}
	\end{equation}
	For $i=j$ and each $k\in\N$ we may consider the quantity $\lambda\upp{j}_{k}=\lambda\upp{j}_{k}(\mb{x},\mb{z})$ from~\eqref{gik-lam} and apply inequality~\eqref{gik-lam} in~\eqref{eq:slope_bound_g} with $\lambda=\lambda\upp{j}_{k}$ to get, for all $k\ge1$,
	\begin{multline}\label{eq:iequalj}
	\abs{\br{\varphi_{k}\upp{j}(\mb{x}+\mb{z})g_{k}\upp{j}(\mb{x}+\mb{z})-\varphi_{k}\upp{j}(\mb{x})g_{k}\upp{j}(\mb{x})}-\lambda\upp{j}_{k}\varphi_{k}\upp{j}(\mb{x})T_{j}(\mb{z})}\\
	\leq \br{\frac{\const\varepsilon\upp{j}_{k}\lnorm{X^{*}}{T_{j}}}{1-\varepsilon\upp{j}_{k}}+\lnorm{X^*}{T_j}\varepsilon_{k}\upp{j}\lip(\varphi\upp{j}_{k})}\lnorm{X}{\mb{z}}.\\
	\end{multline}
	For $i=1,\ldots,j-1$ we again use $\lnorm{X}{\mb{z}}<r$ and $[\mb{x},\mb{x}+\mb{z}]\subseteq B_{X}(\mb{x},r)\subseteq  U_{j-1}\subseteq U_i\subseteq \bigcap_{k=1}^{K_{i}}G\upp{i}_{k}$ by~\eqref{KiUi}, to conclude, by~\eqref{gik-lam-one} and~\eqref{eq:slope_bound_g} with $\lambda=1$, that
	\begin{multline}\label{eq:ibeforej1}
	\abs{\br{\varphi_{k}\upp{i}(\mb{x}+\mb{z})g_{k}\upp{i}(\mb{x}+\mb{z})-\varphi_{k}\upp{i}(\mb{x})g_{k}\upp{i}(\mb{x})}-\varphi_{k}\upp{i}(\mb{x})T_{i}(\mb{z})}\\
	\leq\br{\frac{\const\varepsilon_{k}\upp{i}\lnorm{X^{*}}{T_{i}}}{1-\varepsilon_{k}\upp{i}}+\lnorm{X^*}{T_i}\varepsilon_{k}\upp{i}\lip(\varphi_{k}\upp{i})}\lnorm{X}{\mb{z}}  
	\text{ for $i=1,\ldots,j-1$ and $k=1,\ldots,K_{i}$.}
	\end{multline}
	Moreover, from~\eqref{KiUi} it follows that $\varphi\upp{i}_{k}|_{U_{j-1}}$ is constant $0$ for each $i=1,\ldots,j-1$ and $k>K_{i}$. Therefore
	\begin{multline}\label{eq:ibeforej2}
	\abs{\br{\varphi_{k}\upp{i}(\mb{x}+\mb{z})g_{k}\upp{i}(\mb{x}+\mb{z})-\varphi_{k}\upp{i}(\mb{x})g_{k}\upp{i}(\mb{x})}-\varphi_{k}\upp{i}(\mb{x})T_{i}(\mb{z})}=0 \\ \text{for $i=1,\ldots,j-1$ and $k>K_{i}$.}
	\end{multline}
	Setting $\lambda\upp{j}(\mb{x},\mb{z}):=\sum_{k\in\N}\lambda\upp{j}_{k}(\mb{x},\mb{z})\varphi_{k}\upp{j}(\mb{x})\in [0,1]$ and using $\sum_{k=1}^{\infty}\varphi_k\upp{i}(\mb{x})=1$ for each $1\le i\le j-1$, from~\eqref{varphiik}, we may now 
	multiply the expressions under the modulus sign on the left-hand sides of~\eqref{eq:ibeyondj}, \eqref{eq:iequalj}, \eqref{eq:ibeforej1} and~\eqref{eq:ibeforej2} by $\mb{w}_i$ and sum over respective ranges of $k\ge1$ and then over $i=1,\ldots,l$, to obtain
	\begin{multline*}
	\lnorm{Y}{g(\mb{x}+\mb{z})-g(\mb{x})
		-\br{\sum_{i=1}^{j-1}T_{i}(\mb{z})\mb{w}_{i}+\lambda\upp{j}(\mb{x},\mb{z})T_{j}(\mb{z})\mb{w}_{j}}}\\
	\leq \sum_{i=1}^{j}\sum_{k=1}^{\infty}\br{\frac{\const\varepsilon_{k}\upp{i}\lnorm{X^{*}}{T_{i}}}{1-\varepsilon_{k}\upp{i}}+\lnorm{X^*}{T_i}\varepsilon_{k}\upp{i}\lip(\varphi_{k}\upp{i})}\lnorm{Y}{\mb{w}_{i}}\lnorm{X}{\mb{z}}\leq \theta\lnorm{X}{\mb{z}},
	\end{multline*}
	where the last inequality holds due to~\eqref{epsik}. Therefore, in the latter case of~\eqref{eq:cases},
	\begin{align}
	\lnorm{Y}{g(\mb{x}+\mb{z})-g(\mb{x})}&\leq \br{\opnorm{\sum_{i=1}^{j-1}T_{i}(\cdot)\,\mb{w}_{i}+\lambda\upp{j}(\mb{x},\mb{z})T_{j}(\cdot)\,\mb{w}_{j}}+\theta}\lnorm{X}{\mb{z}} \nonumber \\
	&\leq \br{\max\set{\opnorm{\sum_{i=1}^{j-1}T_{i}(\cdot)\,\mb{w}_{i}},\opnorm{\sum_{i=1}^{j}T_{i}(\cdot)\,\mb{w}_{i}}}+\theta}\lnorm{X}{\mb{z}}
	\nonumber\\
	&\leq(\cyl(T)+\theta)\lnorm{X}{\mb{z}}
	, \label{eq:lip_g_nH}
	\end{align}
	where the penultimate inequality is due to $\lambda\upp{j}(\mb{x},\mb{z})\in[0,1]$ and the convexity of the $\opnorm{\cdot}$ norm.
	Combining~\eqref{eq:lip_g_H} and~\eqref{eq:lip_g_nH}, see also Remark~\ref{rem:cyl_constant}\eqref{rem:cyl_constant:ineq}, we conclude that the inequality 
	\begin{equation*}
	\lnorm{Y}{g(\mb{x}+\mb{z})-g(\mb{x})}\leq (\cyl(T)+\theta)\lnorm{X}{\mb{z}}
	\end{equation*}
	holds in all cases from~\eqref{eq:cases} covering all $\mb{x}\in U_0\setminus \bigcup_{j=1}^{l}\partial U_{j}$ and $\mb{z}\in B(\mb{0}_X,r(\mb{x}))$. Together with~\eqref{eq:supp}, this proves that $g$ is locally $(\cyl(T)+\theta)$-Lipschitz on the complement of the closed Lebesgue null set $\bigcup_{j=0}^{l}\partial{U_{j}}$, hence $\opnorm{Dg(\mb{x})}\le\cyl(T)+\theta$ for almost all $\mb{x}\in X$. Since we have already established that $g$ is also locally Lipschitz on $X$, we deduce~\eqref{lipTglobal}, that $g\colon X\to Y$ is a $(\cyl(T)+\theta)$-Lipschitz mapping; see Corollary~\ref{cor:lip-loc_glob}.
\end{proof}
In the following lemma we will consider mappings defined on $Q\subseteq X$. Note that although our final aim, Theorem~\ref{thm:typical_nondiff-dash}, is to prove typical non-differentiability within the space of $1$-Lipschitz mappings defined on a bounded subset $Q$ of $X$, Lemma~\ref{lemma:C1_der_close_to_psiT} works for any $Q\subseteq X$, in particular, $Q=X$. Also, in the proof of Lemma~\ref{lemma:C1_der_close_to_psiT} we will require a familiar type of smooth approximation result for Lipschitz mappings; the formal statement and proof of this result appears later on, in Section~\ref{sec:SA}.
\begin{lemma}\label{lemma:C1_der_close_to_psiT}
	Let $X$ and $Y$ be normed spaces, where $X$ is finite-dimensional. Let a closed subset $Q\subseteq X$ contain an open set $V$,  $T\in\mc{L}(X,Y)$, $g\colon Q\to W$ be a Lipschitz mapping, where $W\supseteq T(X)$ is a finite-dimensional subspace of $Y$, functions $\psi\colon V\to \R$, $\xi\colon V\to[0,\infty)$ be continuous and bounded on $V$, and $\theta>0$. Assume further that 
	\begin{equation}\label{eq:condition_Dg}
	\opnorm{Dg(\mb{x})-\psi(\mb{x})T}\le\xi(\mb{x})
	\end{equation} 
	for almost all $\mb{x}\in V$. Then there exists a 
	Lipschitz mapping $f\colon Q\to W$ such that
	\begin{enumerate}[(a)]
		\item\label{f=g_outsideH} $f(\mb{x})=g(\mb{x})$ whenever $\mb{x}\in Q\setminus \{\mb{y}\in V\colon \xi(\mb{y})>0\}$.
		\item\label{f_close_to_g} $\lnorm{Y}{f(\mb{x})-g(\mb{x})}\leq \theta$ for all $\mb{x}\in Q$.
		\item\label{f_C1_on_H} $f\in C^{1}(V,Y)$.
		\item\label{Df_close_to_psiT_n} $\opnorm{Df(\mb{x})-\psi(\mb{x})T}\leq \xi(\mb{x})(1+\theta)$ for all $\mb{x}\in V$. 
	\end{enumerate}
\end{lemma}	
\begin{proof}
	The statement of the lemma is trivial if $W=\{\mb{0}_{Y}\}$. Also, if the open set $U:=\{\mb{x}\in V\colon \xi(\mb{x})>0\}$ is empty, then $Dg$ coincides, almost everywhere on $V$, with a continuous mapping $\psi T$   by~\eqref{eq:condition_Dg}. Therefore
		$g\in C^1(V,Y)$ and $Dg(\mb{x})=\psi(\mb{x})T$ for all $\mb{x}\in V$ (see Theorem~\ref{lemma:der_equal_cts_ae_implies_evrywh}), so $f:=g$ satisfies conditions~\eqref{f=g_outsideH}--\eqref{Df_close_to_psiT_n}. Hence assume, without loss of generality, that $W\ne\{\mb{0}_{Y}\}$ and $U\ne\emptyset$. 

	Let $\mb{w}_{1},\ldots,\mb{w}_{l}\in Y$ be a basis of $W$, such that $\lnorm{Y}{\mb{w}_i}=1$ for all $1\le i\le l$ and $\mb{w}_{1}^{*},\ldots,\mb{w}_{l}^{*}\in W^*$ be the corresponding biorthogonal functionals; let $g_i=\mb{w}_i^*\circ g$, so that $g(\mb{x})=\sum_{i=1}^lg_i(\mb{x})\mb{w}_i$ for all $\mb{x}\in Q$. We also fix a basis of $X$ and use this to identify $X$ with $\R^{\dim X}$ in the standard way. This identification allows us to define the Lebesgue measure on $X$.  Accordingly all integrals on subsets of $X$ which appear in this proof should be understood via this identification.  Let also $\eqc>0$ be a constant of equivalence between the norm $\lnorm{X}{\cdot}$ and the Euclidean norm $\lnorm{\text{E}}{\cdot}$ on $\R^{\dim X}=X$ so that for all $\mb{z}\in X$
	\begin{equation}\label{eq:eqc}
	\frac1\eqc\lnorm{\text{E}}{\mb{z}}\leq\lnorm{X}{\mb{z}}\le \eqc\lnorm{\text{E}}{\mb{z}}.
	\end{equation}  
	
	Let $\rho\colon X\to [0,\infty)$ denote the standard smooth (Euclidean) mollifier in $\R^{\dim X}=X$ and for $\varepsilon>0$ let $\rho_{\varepsilon}(\mb{x})\coloneq\varepsilon^{-\dim X}\rho(\mb{x}/\varepsilon)$. In what follows we consider, for each $\mb{x}\in U$ and $\varepsilon\in(0,\varepsilon(\mb{x}))$, where $\varepsilon(\mb{x})>0$ is such that $\mb{x}+\mb{y}\in U$ for any $\lnorm{\text{E}}{\mb{y}}\le\varepsilon(\mb{x})$, the convolution   
	\begin{equation*}
	g* \rho_{\varepsilon}(\mb{x})=		\sum_{i=1}^l(g_i*\rho_{\varepsilon})(\mb{x})\mb{w}_i=\sum_{i=1}^l\br{\int_{\lnorm{\text{E}}{\mb{y}}\le\varepsilon}g_i(\mb{x-y})\rho_{\varepsilon}(\mb{y})\,d\mb{y}}\mb{w}_i,
	\end{equation*}
	where the integration is with respect to the Lebesgue measure. Note that $g*\rho_{\varepsilon}(\mb{x})\in W$ for any $\varepsilon\in(0,\varepsilon(\mb{x}))$. We also let $\inside{U}{\varepsilon}=\set{\mb{x}\in U\colon \dist_{\lnorm{\text{E}}{\cdot}}(\mb{x},X\setminus U)> \varepsilon}$, for $\varepsilon>0$.
	
	Note that for any $\varepsilon>0$ the convolution $g*\rho_{\varepsilon}$ is defined on $\inside{U}{\varepsilon}$, belongs to $\lip(\inside{U}{\varepsilon},Y)\cap C^1(\inside{U}{\varepsilon},Y)$ and approximates $g$ well: for any $\mb{x}\in\inside{U}{\varepsilon}$
	\begin{align}
	\lnorm{Y}{g* \rho_{\varepsilon}(\mb{x})-g(\mb{x})}
	&= \lnorm{Y}{\sum_{i=1}^l
		(g_i*\rho_{\varepsilon}(\mb{x})-g_i(\mb{x}))\mb{w}_i 
	} \nonumber \\
	&= \lnorm{Y}{\sum_{i=1}^l
		\br{\int_{\lnorm{\text{E}}{\mb{y}}\le\varepsilon}\br{	g_i(\mb{x}-\mb{y})-g_i(\mb{x})}\rho_{\varepsilon}(\mb{y})\,d\mb{y}}\mb{w}_i
	} \nonumber\\
	&\le			
	\sum_{i=1}^{l}\int_{\lnorm{\text{E}}{\mb{y}}\le\varepsilon}\abs{g_{i}(\mb{x}-\mb{y})-g_{i}(\mb{x})}\rho_{\varepsilon}(\mb{y})\,d\mb{y}
	\leq l\max_{1\le i\le l}\lnorm{X^*}{\mb{w}_i^*}\eqc\lip(g)\varepsilon. \label{eq:convo_approx}
	\end{align}
		
	We note, for future reference, that using $D(g_i*\rho_\varepsilon)(\mb{x})
	=(Dg_i*\rho_\varepsilon)(\mb{x})$ and~\eqref{eq:condition_Dg}, we have for all $\varepsilon>0$, $\mb{x}\in \inside{U}{\varepsilon}$ and $\lnorm{X}{\mb{v}}\le1$
\begin{align}
	&{D(g* \rho_{\varepsilon})(\mb{x})(\mb{v})-\psi(\mb{x})T(\mb{v})}\label{eq:der} ={
		\sum_{i=1}^{l}D(g_i*\rho_\varepsilon)(\mb{x})(\mb{v})\mb{w}_i-\psi(\mb{x})T(\mb{v})
	} \\
	&
	={ 
		\sum_{i=1}^{l}(Dg_i*\rho_\varepsilon)(\mb{x})(\mb{v})\mb{w}_i-\psi(\mb{x})T(\mb{v})
	}\notag
	={
		\int\limits_{\lnorm{\text{E}}{\mb{y}}\le\varepsilon}\br{ 
			Dg(\mb{x}-\mb{y})(\mb{v})-\psi(\mb{x})T(\mb{v})}\rho_{\varepsilon}(\mb{y})\,d\mb{y}
	}.  \notag
\end{align}
We now use~\eqref{eq:condition_Dg} to estimate the norm of~\eqref{eq:der}  from above as
\begin{align} 
	&\lnorm{Y}{D(g* \rho_{\varepsilon})(\mb{x})(\mb{v})-\psi(\mb{x})T(\mb{v})}\notag\\	
	&\le\int\limits_{\lnorm{\text{E}}{\mb{y}}\le\varepsilon}
	\left(
	\opnorm{Dg(\mb{x}-\mb{y})-\psi(\mb{x}-\mb{y})T}
	+
	\abs{\psi(\mb{x}-\mb{y})-\psi(\mb{x})}
	\opnorm{T}
	\right)
	\rho_{\varepsilon}(\mb{y})\,d\mb{y}
	\notag \\
	&\leq
	\xi(\mb{x})+
	\int\limits_{\lnorm{\text{E}}{\mb{y}}\le\varepsilon}
	\br{\abs{
			\xi(\mb{x}-\mb{y})-\xi(\mb{x})
		}
		+\abs{\psi(\mb{x}-\mb{y})-\psi(\mb{x})}\opnorm{T}}
	\rho_{\varepsilon}(\mb{y})\,d\mb{y}
	\notag \\
	&\le 
	\xi(\mb{x})+
	\sup_{\lnorm{\text{E}}{\mb{y}}\le\varepsilon}
	\br{\abs{
			\xi(\mb{x}-\mb{y})-\xi(\mb{x})
		}
		+\abs{\psi(\mb{x}-\mb{y})-\psi(\mb{x})}\opnorm{T}}
	.\label{eq:convo_deriv}
\end{align}
	Moreover, for any compact set $\emptyset\ne K\subseteq U$, we may use that $\xi$ and $\psi$ are continuous on $U$ and $\xi>0$ on $U$ to choose $\delta_{K}>0$ sufficiently small so that for all $\varepsilon\in (0,\delta_{K})$ we have $K\subseteq \inside{U}{\varepsilon}$ and 
	\begin{equation*}
	\sup_{\mb{x}\in K}\sup_{\lnorm{E}{\mb{y}}\le\varepsilon}\br{\abs{
			\xi(\mb{x}-\mb{y})-\xi(\mb{x})
		}
		+\abs{\psi(\mb{x}-\mb{y})-\psi(\mb{x})}\opnorm{T}}\leq \frac{\theta}{2}\min_{\mb{z}\in K}\xi(\mb{z}).
	\end{equation*}
	Combining this with~\eqref{eq:convo_deriv} we get
	\begin{equation}\label{eq:convo_deriv_final}
	\lnorm{Y}{D(g* \rho_{\varepsilon})(\mb{x})(\mb{v})-\psi(\mb{x})T(\mb{v})}\leq \xi(\mb{x})\br{1+\frac{\theta}{2}}
	\end{equation}
	for any compact $K\subseteq U$, $\mb{x}\in K$ and $\varepsilon\in (0,\delta_{K})$.
	
	Let $(\varphi_{k})_{k\in\N}$ be a smooth, locally finite partition of unity on $U=\{\mb{x}\in V\colon \xi(\mb{x})>0\}$ and for each $k\in\N$ set
	\begin{multline}\label{eq:thetak}
	\theta_{k}:=\frac{\theta\min_{\mb{z}\in \supp\varphi_{k}}\xi(\mb{z})}{2^{k}\br{1+\lip(\varphi_{k})}},\quad \varepsilon_{k}:=\frac{1}{2}\min\set{\delta_{\supp \varphi_{k}}, \frac{\theta_{k}}{l w C\br{\lip(g)+1}}},\quad A_{k}:=U_{\varepsilon_{k}},\\ h_{k}(\mb{x})=\begin{cases}
	g* \rho_{\varepsilon_{k}}(\mb{x}) & \text{if }\mb{x}\in A_{k},\\
	\mb{0}_{Y} & \text{if }\mb{x}\in U\setminus A_{k}
	\end{cases}\quad \text{so that }h_{k}\in\SLA(g,A_{k},Y,\theta_k),
	\end{multline}
	where $\eqc$ is fixed in~\eqref{eq:eqc}, $w=\max_{1\le i\le l}\lnorm{X^*}{\mb{w}_i^*}$ and $\SLA(g,A_{k},Y,\theta_{k})$ is defined as the class of smooth Lipschitz mappings $A_{k}\to Y$ which approximate $g$ uniformly within error $\theta_{k}$; a precise description of this class is given in Definition~\ref{def:sla}, and $h_k\in\SLA(g,A_{k},Y,\theta_k)$ follows from~\eqref{eq:convo_approx} and the choice of $\varepsilon_k$ above. The desired mapping $f$ may now be defined by $\sum_{k\in\N}\varphi_{k}h_{k}$ in $U$ and set equal to $g$ in $Q\setminus U$; then a `smooth approximation result' for mappings of this form will establish the remaining properties of $f$. We postpone this result, Lemma~\ref{lemma:SLA}, until Section~\ref{sec:SA}. We now describe precisely how to apply Lemma~\ref{lemma:SLA}.
	Let $h:=g$, $P(\mb{x}):=\psi(\mb{x})T\in\mc{L}(X,Y)$ and $\eta(\mb{x}):=\xi(\mb{x})\br{1+\frac{\theta}{2}}$ for each $\mb{x}\in U$. The conditions of Lemma~\ref{lemma:SLA}, including the additional condition of~\eqref{SLAii}, are now satisfied, by the definition of $U$ and~\eqref{eq:convo_deriv_final}, and we may let $f\coloneq \tilde{h}$ be given by the conclusion of Lemma~\ref{lemma:SLA}. Then conditions~\eqref{f=g_outsideH}, \eqref{f_close_to_g} of the present lemma and $f\in C^1(U,Y)$, which is a part of~\eqref{f_C1_on_H}, are satisfied. Since all values of $g$ and of $h_k$ are in $W$, we get $f\colon Q\to W$. Moreover, Lemma~\ref{lemma:SLA}\eqref{SLAii}, \eqref{eq:thetak} and~\eqref{eq:convo_deriv_final} give, for any $\mb{x}\in U$,
	\begin{multline}\label{Df-ineq}
	\opnorm{Df(\mb{x})-\psi(\mb{x})T}=\opnorm{D\tilde{h}(\mb{x})-P(\mb{x})}\\
	\leq \eta(\mb{x})+\sum_{k\in\N}\lip(\varphi_{k})\1_{\supp\varphi_{k}}(\mb{x})\theta_{k}\leq \xi(\mb{x})\br{1+\theta},
	\end{multline}
	which proves the inequality of~\eqref{Df_close_to_psiT_n} for all $\mb{x}\in U$.
	This implies, in particular, that for all $\mb{x}\in U$
	\begin{equation*}
	\opnorm{Df(\mb{x})}
	\le \psi(\mb{x})\opnorm{T}+\xi(\mb{x})	\br{1+\theta}\le 
	L:=\Psi\opnorm{T}+\Xi\br{1+\theta},
	\end{equation*}
	where $\Psi,\Xi>0$ are upper bounds of bounded functions $\psi,\xi$ respectively. Hence we conclude that $f$ is locally $L$-Lipschitz on $U$, and so using~\eqref{f=g_outsideH}, $g\in\lip(Q,W)$ and Lemma~\ref{lemma:awkward_lipschitz_bound}, we obtain $f\in\lip(Q,W)$.
	We also note from~\eqref{Df-ineq} that the mapping $\Phi\colon V\to\mc{L}(X,Y)$ defined by 
	\begin{equation*}
	\Phi(\mb{x})=\begin{cases}
	Df(\mb{x}) & \mb{x}\in U,\\
	\psi(\mb{x})T & \mb{x}\in V\setminus U
	\end{cases}
	\end{equation*}
	is continuous on $V$, as $U=\{\mb{x}\in V\colon \xi(\mb{x})>0\}$ and $\xi$ is continuous on $V$. We may now apply Theorem~\ref{lemma:dsty_pt}\eqref{dervs_equal_ae} to Lipschitz mappings $f,g\colon V\subseteq X\to W$ between finite-dimensional spaces to conclude that $Df$ and $Dg$ exist and are equal almost everywhere on $\{\mb{x}\in V\colon f(\mb{x})=g(\mb{x})\}\supseteq V\setminus U$. This, together with~\eqref{eq:condition_Dg} and the definitions of $\Phi$ and $U$ implies $Df(\mb{x})=\Phi(\mb{x})$ for almost every $\mb{x}\in V$. Therefore, by Theorem~\ref{lemma:der_equal_cts_ae_implies_evrywh}, $f\in C^1(V,Y)$ and $Df(\mb{x})=\Phi(\mb{x})$ for all $\mb{x}\in V$.
	This verifies~\eqref{f_C1_on_H} and, together with~\eqref{Df-ineq}, implies~\eqref{Df_close_to_psiT_n}.
\end{proof}
\begin{lemma}\label{lemma:psi}
	Let $X$ and $Y$ be normed spaces, where $X$ is finite-dimensional. Let $E\subseteq X$ be compact and purely unrectifiable, $\eta>0$, function $\varphi\colon X\to [0,1]$ be continuous and $T\in\mc{L}(X,Y)$. Then there exist a Lipschitz mapping $f\colon X\to T(X)$, a function $\psi\colon X\to [0,1]$ and an open set $H\subseteq X$ such that the following statements hold:
	\begin{enumerate}[(i)]
		\item\label{H_theta_close_to_E} $E\subseteq H\subseteq B_{X}(E,\eta)$.
		\item\label{f_lip_and_C1} $f\in \lip(X,Y)\cap C^{1}(H,Y)$.
		\item\label{f_small_and_supp_contained} $\sup_{\mb{x}\in X} \lnorm{Y}{f(\mb{x})}\leq \eta$ and
		$\supp f\subseteq \supp \varphi$.
		\item\label{Df_close_to_psiT} $\opnorm{Df(\mb{x})-\psi(\mb{x})T}\leq \eta$ for Lebesgue almost every $\mb{x}\in X$.
		\item\label{psi_smaller_than_phi_and_sometimes_equal} $0\leq \varphi(\mb{x})\1_{H}(\mb{x})\leq\psi(\mb{x})\leq \varphi(\mb{x})\1_{B_{X}(E,\eta)}(\mb{x})$ for all $\mb{x}\in X$.  
		\item\label{psi_continuous} $\psi(\mb{x})=\varphi(\mb{x})$ for all $\mb{x}\in H$, and so $\psi|_H$ is continuous.
	\end{enumerate}
\end{lemma}
\begin{proof}
	The statement of the lemma is clear for $T=0$, so assume, without loss of generality, that $T\ne0$ is such that $\opnorm{T}\le1$. We may do so as if $\opnorm{T}>1$ and the conclusion of this lemma holds for operators of norm less than or equal to $1$, we let $f_1,\psi$ and $H$ correspond to $E$, $\eta_1=\eta/\opnorm{T}$, $\varphi$ and $T_1=T/\opnorm{T}$ and define $f=\opnorm{T}f_1$.
	
	Let $C=\cyl(T)+5$, where $\cyl(T)$ is the constant given by Definition~\ref{def:cyl_constant},
	\begin{equation}\label{eq:def_eta_k}
	\theta:=\min\set{1,\eta/5}
	,
	\qquad k\in\N\cap\Bigl[\frac{C}{\theta},\frac{2C}{\theta}\Bigr],
	\end{equation}
	$G_{1}:=H_{1}:=B_{X}(E,\eta)$ and note that $E\subseteq H_{1}$. For each $i=2,\ldots,k-1$, whenever the open set $H_{i-1}$ containing compact $E\cap\set{\mb{x}\in X\colon\varphi(\mb{x})\geq \frac{i}{k}}$ has been defined, let 
	\begin{equation}\label{eq:defGi}
	G_{i}:=H_{i-1}\cap \set{\mb{x}\in X\colon\varphi(\mb{x})>\frac{i-1}{k}}. 
	\end{equation}
	Next apply Lemma~\ref{lemma:pu_mapping_with_prescribed_derivative} to $X$, $Y$, the compact, purely unrectifiable subset $E\cap\set{\mb{x}\in X\colon\varphi(\mb{x})\geq i/k}$ of the open set $U=G_{i}$, $\theta$ and $T$ to obtain a mapping $g_{i}\in \lip_{C}(X,T(X))$, as $C\geq \cyl(T)+\theta$,
	and an open set $H_{i}\subseteq X$ such that
	\begin{enumerate}[(a)]
		\item\label{gi_support} $\supp g_{i}\subseteq G_{i}$ and $Dg_i(\mb{x})=0$ for any $\mb{x}\notin G_i$.
		\item\label{gi_small} 
		$\sup_{\mb{x}\in X} \lnorm{Y}{g_{i}(\mb{x})}\leq \theta$.
		\item\label{Hi_between_E_and_Gi} $E\cap\set{\mb{x}\in X\colon\varphi(\mb{x})\geq \frac{i}{k}}\subseteq
		H_{i}\subseteq \cl{H_{i}}\subseteq G_{i}$. 	
		\item\label{Dgi_close_to_T} $\opnorm{Dg_{i}(\mb{x})-T}\leq \theta$ for Lebesgue almost all $\mb{x}\in H_{i}$.
	\end{enumerate}
Note that~\eqref{Hi_between_E_and_Gi} implies that  $E\cap\set{\mb{x}\in X\colon\varphi(\mb{x})\geq \frac{i+1}{k}}\subseteq H_i$ and hence we construct inductively mappings $g_{2},\ldots,g_{k-1}\in \lip_C(X,T(X))$ and nested sequences of open sets $H_{1}\supseteq\ldots\supseteq H_{k-1}$ and $G_{1}\supseteq\dots\supseteq G_{k-1}$ such that properties~\eqref{gi_support}--\eqref{Dgi_close_to_T} hold for every $2\le i\le k-1$.
	Consider $g:=\frac{1}{k}\sum_{i=2}^{k-1}g_{i}\in \lip_C(X,T(X))$. Properties~\eqref{gi_support} and~\eqref{gi_small} of $g_{2},\ldots,g_{k-1}$, and the nested property of $G_i$ lead to 
	\begin{equation}\label{eq:g_small}
	\sup_{\mb{x}\in X}\lnorm{Y}{g(\mb{x})}\leq \theta,
	\end{equation}
	and
	\begin{equation}\label{eq:supp_g_contained}
	\supp g\subseteq G_{2}\subseteq \supp\varphi.
	\end{equation}
	Define $j\colon G_{1}\to\set{1,\ldots,k-1}$ by
	\begin{equation}\label{eq:defjx}
	j(\mb{x})=\max\set{j\in\set{1,\ldots,k-1}\colon \mb{x}\in G_{j}}
	\end{equation}
	and note that for all $\mb{x}\in G_1$ we have
	$\varphi(\mb{x})\ge\frac{j(\mb{x})-1}{k}$
	due to the non-negativity of $\varphi$ when $j(\mb{x})=1$ and to the definitions~\eqref{eq:defGi} and~\eqref{eq:defjx} when $j(\mb{x})>1$. Observe also that whenever all $g_i$ are differentiable at $\mb{x}\in G_{1}$,
	\begin{equation}\label{eq:jx_form_of_g}
	Dg(\mb{x})=\frac{1}{k}\sum_{i=2}^{j(\mb{x})}Dg_{i}(\mb{x}),
	\end{equation}
	because of~\eqref{gi_support}, and the definitions of $g$ and $j(\mb{x})$. 
	Let $\psi\colon X\to [0,1]$ be given by
	\begin{equation}\label{eq:defpsi}
	\psi(\mb{x})=
	\begin{cases}
	\min\set{\frac{j(\mb{x})+2}{k},\varphi(\mb{x})},&\text{if }\mb{x}\in G_{1};\\
	0,&\text{otherwise}
	.
	\end{cases}
	\end{equation}
	Then the last inequality of~\eqref{psi_smaller_than_phi_and_sometimes_equal} is satisfied.
	Note also for future reference that for every $\mb{x}\in G_{1}$ we have 
	\begin{equation}\label{eq:jx_bounds_on_psi}
	\frac{j(\mb{x})-1}{k}
	\le\psi(\mb{x})
	\le\frac{j(\mb{x})+2}{k},
	\end{equation}
	using $\varphi(\mb{x})\ge\frac{j(\mb{x})-1}{k}$.
	Define an open set $H\subseteq X$ by
	\begin{equation*}
	H:= \bigcup_{j=1}^{k-1} S_j,
	\text{ where }
	S_j=\set{\mb{x}\in H_{j}\colon \varphi(\mb{x})<\frac{j+2}{k}}.
	\end{equation*}
	The second inclusion of~\eqref{H_theta_close_to_E}, $H\subseteq B_X(E,\eta)=H_{1}$, is now clear, due to the nested property of $H_i$. We prove the first inclusion: Let $\mb{x}\in E$. Then there is an $m\in\set{0,\ldots,k-1}$ such that $\frac{m}{k}\leq \varphi(\mb{x})\leq \frac{m+1}{k}$. If $m\leq 1$, then since $\mb{x}\in E\subseteq H_1$ and $\varphi(\mb{x})\le\frac2k<\frac3k$, we have $\mb{x}\in S_1$; for $2\le m\le k-1$ we have $\mb{x}\in E\cap \set{\mb{x}\in X\colon\varphi(\mb{x})\geq m/k}\subseteq H_{m}$ by~\eqref{Hi_between_E_and_Gi} and $\varphi(\mb{x})\leq\frac{m+1}{k}<\frac{m+2}{k}$ so $\mb{x}\in S_m$. This proves the first inclusion of~\eqref{H_theta_close_to_E}.
	
	We now verify~\eqref{psi_continuous} and the remaining inequalities of~\eqref{psi_smaller_than_phi_and_sometimes_equal}. Note first that from $H_1=G_1$, \eqref{eq:defjx} and~\eqref{Hi_between_E_and_Gi} we have, for any $1\leq j\leq k-1$, that $\mb{x}\in H_{j}$ implies $j(\mb{x})\geq j$. Fix any $\mb{x}\in H\subseteq G_{1}$ and $j\in\set{1,\ldots,k-1}$ such that $\mb{x}\in S_{j}$; then $\varphi(\mb{x})< \frac{j+2}{k}\le\frac{j(\mb{x})+2}{k}$ thus, by~\eqref{eq:defpsi}, $\psi(\mb{x})=\varphi(\mb{x})$ proving~\eqref{psi_continuous}, which together with~\eqref{eq:defpsi} gives the first two inequalities of~\eqref{psi_smaller_than_phi_and_sometimes_equal}. 
	
	Let now $\mb{x}\in G_{2}$ be any point such that all $g_i$ are differentiable at $\mb{x}$ and for any $1\le i\le k-1$ such that $\mb{x}\in H_i$ we also have the inequality~\eqref{Dgi_close_to_T}. 
	We remark that
	almost all points of $G_{2}$ are such.
	Thus, whenever $1\leq i\leq j(\mb{x})-1$, we have, from~\eqref{eq:defjx} and~\eqref{eq:defGi}, that $\mb{x}\in G_{j(\mb{x})}\subseteq H_{j(\mb{x})-1}\subseteq H_i$, hence the inequality~\eqref{Dgi_close_to_T} applies.
Therefore, we now verify that 
for almost all $\mb{x}\in G_2$,
	\begin{equation}
	\label{eq:initialDg-psiT}	
	\begin{split}
	&\opnorm{Dg(\mb{x})-\psi(\mb{x})T}
	=\opnorm{\tfrac{1}{k}\sum_{i=2}^{j(\mb{x})}Dg_{i}(\mb{x})-\psi(\mb{x})T} 
	\\
	&\le	
	\tfrac{1}{k}\sum_{i={2}}^{j(\mb{x})-1}\opnorm{Dg_{i}(\mb{x})-T}+	\tfrac1k\opnorm{Dg_{j(\mb{x})}(\mb{x})}+\abs{\tfrac{j(\mb{x})-2}{k}-\psi(\mb{x})}\opnorm{T}
	\\
	&\leq \theta+\tfrac{C}{k}+\tfrac{4}{k}
	\le\tfrac{4C}{k}
	\le 4C\min\set{\varphi(\mb{x}),\tfrac1k}.
	\end{split}	
	\end{equation}
The first equality follows from~\eqref{eq:jx_form_of_g}; the first inequality of the last line is guaranteed by~\eqref{Dgi_close_to_T}, $g_{j(\mb{x})}\in\lip_C(X,T(X))$ and~\eqref{eq:jx_bounds_on_psi}; finally, we use~\eqref{eq:def_eta_k}, $\mb{x}\in G_2$ and~\eqref{eq:defGi} for the remaining inequalities. 
	
	Using~\eqref{gi_support} of the present proof, the nested property of $G_i$ and~\eqref{eq:jx_form_of_g}, we conclude that $Dg(\mb{x})=0$ for every $\mb{x}\in X\setminus G_{2}$. 
	From~\eqref{eq:defGi} and the already verified~\eqref{psi_smaller_than_phi_and_sometimes_equal} we also have that $0\le\psi(\mb{x})\le\varphi(\mb{x})\1_{H_1}(\mb{x})\le\frac1k$ for all $\mb{x}\in X\setminus G_{2}$. Hence, for every $\mb{x}\in X\setminus G_{2}$ we have $\opnorm{Dg(\mb{x})-\psi(\mb{x})T}=\psi(\mb{x})\opnorm{T}\le\psi(\mb{x})\le\min\set{\varphi(\mb{x}),\frac1k}$. This, together with~\eqref{eq:initialDg-psiT} and $C>1$, establishes
	\begin{equation}\label{eq:Dg-psiT}
	\opnorm{Dg(\mb{x})-\psi(\mb{x})T}\leq 
	4C\min\Bigl\{\varphi(\mb{x}),\frac1k\Bigr\}
	\quad\text{for almost all }
	\mb{x}\in X.
	\end{equation}
	
	Let $f\colon  X\to T(X)\subseteq Y$ be the Lipschitz mapping given by the conclusion of Lemma~\ref{lemma:C1_der_close_to_psiT} applied to $X$, $Y$, $Q=X$, $V=H$, $T$, $W=T(X)$, $g\in\lip(X,T(X))$, continuous bounded functions $\psi|_H$, $\xi(\mb{x})=4C\min\Bigl\{\varphi(\mb{x}),\frac1k\Bigr\}$ for $\mb{x}\in H$ and $\frac{\theta}{4C}$, where the validity of~\eqref{eq:condition_Dg} is guaranteed by~\eqref{eq:Dg-psiT}. Note that from Lemma~\ref{lemma:C1_der_close_to_psiT}\eqref{f_C1_on_H} we have $f\in C^1(H,Y)$, so~\eqref{f_lip_and_C1} is satisfied.
	
	Combining Lemma~\ref{lemma:C1_der_close_to_psiT}\eqref{f_close_to_g} with~\eqref{eq:g_small} and~\eqref{eq:def_eta_k}, we deduce for every $\mb{x}\in Q=X$
	\begin{equation*}
	\lnorm{Y}{f(\mb{x})}
	\leq 
	\lnorm{Y}{f(\mb{x}) - g(\mb{x})}
	+\lnorm{Y}{g(\mb{x})}
	\le 
	2\theta \leq \eta.
	\end{equation*}
	Moreover, we deduce from Lemma~\ref{lemma:C1_der_close_to_psiT}\eqref{f=g_outsideH} that $f(\mb{x})>0$ implies either $g(\mb{x})=f(\mb{x})>0$ or $\xi(\mb{x})>0$. In both cases we conclude that $\mb{x}\in\supp(\varphi)$, in the former case due to~\eqref{eq:supp_g_contained} and in the latter due to the definition of $\xi$. This establishes~\eqref{f_small_and_supp_contained} of the present lemma. Finally, we verify the inequality of~\eqref{Df_close_to_psiT} for almost every $\mb{x}\in X$. 	
	First, for every $\mb{x}\in H$, we may apply Lemma~\ref{lemma:C1_der_close_to_psiT}\eqref{Df_close_to_psiT_n}, $\xi(\mb{x})\le4C/k$ and~\eqref{eq:def_eta_k} to obtain \begin{equation*}
	\opnorm{Df(\mb{x})-\psi(\mb{x})T}
	\le	\xi(\mb{x})\br{1+\frac{\theta}{4C}}\le\xi(\mb{x})+\theta\le\frac{4C}k+\theta\le5\theta\le\eta.
	\end{equation*}
	Further, by Lemma~\ref{lemma:C1_der_close_to_psiT}\eqref{f=g_outsideH} the set $X\setminus H$ is contained in the set where $f$ and $g$, Lipschitz mappings $X\to T(X)$ between finite-dimensional spaces, coincide. Since $Df=Dg$ almost everywhere in the latter set (see Theorem~\ref{lemma:dsty_pt}), we have $Df=Dg$ almost everywhere in $X\setminus H$.
	Hence, by~\eqref{eq:Dg-psiT} and~\eqref{eq:def_eta_k}, we have for almost every $\mb{x}\in X\setminus H$
	\begin{equation*}
	\opnorm{Df(\mb{x})-\psi(\mb{x})T}
	=\opnorm{Dg(\mb{x})-\psi(\mb{x})T}
	\le	
	\frac{4C}{k}
	\le 4\theta\le\eta.\qedhere
	\end{equation*}
\end{proof}

\begin{lemma}\label{lemma:approximation}
	Let $X$ and $Y$ be normed spaces, $E\subseteq U\subseteq Q\subseteq X$, where $E$ is compact and $U$ is open, $\eta>0$ and $g\colon Q\to Y$ be a mapping with $g\in C^{1}(U,Y)$. Then there exists $\delta\in (0,\eta)$ such that for every $h\colon Q\to Y$ with 
	\begin{equation}\label{eq:h_close_to_g}
	\sup_{\mb{z}\in Q}\lnorm{Y}{h(\mb{z})-g(\mb{z})}\leq \eta \delta/4,
	\end{equation}
	every $\mb{x}\in E$ and every $\mb{y}\in X$ with $\lnorm{X}{\mb{y}}\leq \delta$ we have
	\begin{equation*}
	\lnorm{Y}{h(\mb{x}+\mb{y})-h(\mb{x})-Dg(\mb{x})(\mb{y})}\leq \eta\delta.
	\end{equation*}
\end{lemma}
\begin{proof}
	Because $g$ is a $C^{1}$ smooth mapping and $E$ is compact, we have that $g$ is uniformly \Fre differentiable on $E$; see Lemma~\ref{lemma:C1_compact_uniform}. In other words, we may choose $\delta\in(0,\eta)$ small enough so that for all $\mb{x}\in E$ and all $\mb{y}\in X$ with $\lnorm{X}{\mb{y}}\leq \delta$ one has $B(\mb{x},\delta)\subseteq Q$ and
	\begin{equation*}
	\lnorm{Y}{g(\mb{x}+\mb{y})-g(\mb{x})-Dg(\mb{x})(\mb{y})}\leq \tfrac{\eta}{2}\lnorm{X}{\mb{y}}.
	\end{equation*}
	The conclusion of the lemma follows immediately.
\end{proof}
\begin{lemma}\label{lemma:sequence}
	Let $X$ and $Y$ be normed spaces, where $X$ is finite-dimensional.
	Let	$E\subseteq H_{0}\subseteq Q\subseteq X$ be sets, where $E$ is compact and purely unrectifiable, and $H_{0}$ is open. Let $f\dow{0}\in \lip(Q,Y)\cap C^1(H_{0},Y)$
	and $\eta\in (0,1)$. For each $k\in\N$ let $T\dow{k}\in \mc{L}(X,Y)$, $\varphi_{k}\in C(X,[0,1])$ and $\theta_k>0$. Then there is a sequence of sets $H_{j}\subseteq H_0$, Lipschitz mappings $f\dow{j}\colon Q\to Y$ and functions 
	$\psi_{j}\colon X\to [0,1]$ such that for each $j\ge1$
	\begin{enumerate}[(i)]
		\item\label{Hjs} $H_{j}$ is open, $E\subseteq H_{j}\subseteq H_{j-1}$ and $f\dow{j}\in \lip(Q, Y)\cap C^1({H}_{j}, Y)$.
		\item\label{conv} $\sup_{\mb{x}\in Q}\lnorm{ Y}{f\dow{j}(\mb{x})-f\dow{j-1}(\mb{x})}\leq \theta_j$ and $f\dow{j}(\mb{x})=f\dow{j-1}(\mb{x})$ whenever $\varphi_{j}(\mb{x})=0$.
		\item\label{psij} $\1_{H_{j}}(\mb{x})\varphi_{j}(\mb{x})\leq \psi_{j}(\mb{x})\leq \1_{H_{j-1}}\varphi_{j}(\mb{x})$ for all $\mb{x}\in X$.
		\item\label{fj_diff_ae} $f\dow{j}$ is differentiable Lebesgue a.e. on $H_{0}$.
		\item\label{der} $\opnorm{Df\dow{j}(\mb{x})-L}\leq \opnorm{Df\dow{0}(\mb{x})+\sum_{k=1}^{j}\psi_k(\mb{x})T\dow{k}-L}+\eta$ for any $L\in\mc{L}(X, Y)$ and Lebesgue almost every $\mb{x}\in H_{0}$.
	\end{enumerate}
\end{lemma}
\begin{proof}
	Suppose $j\geq 1$ and an open set $H_{j-1}\supseteq E$ and a Lipschitz mapping $f\dow{j-1}\in \lip(Q,Y)\cap C^1({H}_{j-1}, Y)$ have already been defined so that for $i=j-1$
	\begin{align}\label{eq:fj-f0_finite_dim}
	&(f\dow{i}-f\dow{0})(\mb{x})
	=\sum_{k=1}^i g\dow{k}(\mb{x})
	\quad
	\text{for all }\mb{x}\in Q,\\		
	&\text{where }
	\sum_{k=1}^i g\dow{k}\colon
	X\to\Span\br{\bigcup_{k=1}^{i}T\dow{k}(X)}
	\text{\ is a Lipschitz mapping }
	.
	\notag 
	\end{align}
	Here we interpret an empty sum as zero, an empty union as the empty set and the linear span of the empty set as $\set{0}$. Thus the conditions are met for $j=1$. Then we choose
	\begin{equation}\label{eq:etaj}
	0<\eta_{j}\leq \min\{2^{-j}\eta,\theta_{j}\}
	\end{equation}
	small enough so that $\cl{B}_{X}(E,\eta_{j})\subseteq H_{j-1}$ and apply Lemma~\ref{lemma:psi} to $X$, $Y$, $E$, $\eta_j$, $\varphi_j$, and $T\dow{j}$ to get a mapping $g\dow{j}\colon X\to T\dow{j}(X)$, a function $\psi_{j}\colon X\to [0,1]$ and an open set $H_{j}\subseteq X$ with the following properties:
	\begin{enumerate}[(a)]
		\item\label{Hj_in_between} $E\subseteq H_{j}\subseteq B_{X}(E,\eta_{j})\subseteq \cl{B}_{X}(E,\eta_{j})\subseteq H_{j-1}$.
		\item\label{gj_C1} $g\dow{j}\in \lip(X, Y)\cap C^1({H}_{j}, Y)$.
		\item\label{gj_small} $\sup_{\mb{x}\in X}\lnorm{ Y}{g\dow{j}(\mb{x})}\leq \eta_{j}$ and $\supp g\dow{j}\subseteq \supp \varphi_{j}$.
		\item\label{gj_der} $\opnorm{Dg\dow{j}(\mb{x})-\psi_j(\mb{x})T\dow{j}}\leq \eta_{j}$ for Lebesgue almost every $\mb{x}\in X$.
		\item\label{psij_phij} $0\leq \psi_{j}(\mb{x})\leq \1_{B_{X}(E,\eta_{j})}\varphi_{j}(\mb{x})$ for all $\mb{x}\in X$ and $\psi_{j}(\mb{x})=\varphi_{j}(\mb{x})$ for all $\mb{x}\in H_{j}$.
	\end{enumerate}
	Let $f\dow{j}:=f\dow{j-1}+g\dow{j}$; note that $f\dow{j}$ is defined on $Q$ and that~\eqref{eq:fj-f0_finite_dim} extends to $i=j$. Now we have that~\eqref{Hjs} and~\eqref{conv} are implied by~\eqref{Hj_in_between}, \eqref{gj_C1} and~\eqref{gj_small}, and~\eqref{psij} follows from~\eqref{psij_phij} and~\eqref{Hj_in_between}. Property~\eqref{fj_diff_ae} of $f\dow{j}$ follows from~\eqref{eq:fj-f0_finite_dim}, the finite-dimensionality of $X\supseteq Q$ and of all $T\dow{k}(X)$ in~\eqref{eq:fj-f0_finite_dim}, Rademacher's theorem and $f\dow{0}\in C^1(H_{0},Y)$. Finally, we check~\eqref{der}. Let $j\in\N$ and $\mb{x}$ be any point from $H_{0}$ lying in the intersection of the full measure sets corresponding to the mappings $g\dow{1},\ldots,g\dow{j}$ given by~\eqref{gj_der}. Let $L\in \mc{L}(X, Y)$ be arbitrary. Then we have
	\begin{multline*}
	\opnorm{Df\dow{j}(\mb{x})-L}=\opnorm{Df\dow{0}(\mb{x})+\sum_{k=1}^{j}Dg\dow{k}(\mb{x})-L}\\
	\le\opnorm{Df\dow{0}(\mb{x})+\sum_{k=1}^{j}\psi_{k}(\mb{x})T\dow{k}-L}+\sum_{k=1}^{j}\opnorm{Dg\dow{k}(\mb{x})-\psi_{k}(\mb{x})T\dow{k}}\\
	\leq \opnorm{Df\dow{0}(\mb{x})+\sum_{k=1}^{j}\psi_{k}(\mb{x})T\dow{k}-L}+\eta,
	\end{multline*}
	where, to get the last inequality, we applied~\eqref{gj_der} and~\eqref{eq:etaj}.
\end{proof}
\begin{lemma}\label{lemma:bmgame_step}
	Let $X$ and $Y$ be normed spaces, where $X$ is finite-dimensional. Let $E\subseteq H \subseteq Q\subseteq X$ be sets, where $E$ is compact and purely unrectifiable, $H$ is open and $Q$ is bounded and closed. Let $\theta>0$ and $f\in\lip_1(Q,Y)\cap C^1(H,Y)$ and $T\in\mc{L}(X,Y)$ be such that $\lip(f),\opnorm{T}<1$. Then there exist an open set $U\subseteq X$, a function $g\in \lip_1(Q,Y)\cap C^1(U,Y)$ and a positive number $\delta\in(0,\theta)$ such that
	\begin{enumerate}[(i)]
		\item\label{U_between_E_and_H} $E\subseteq U\subseteq H$
		\item\label{g_theta_close} $\sup_{\mb{x}\in Q}\lnorm{Y}{g(\mb{x})-f(\mb{x})}\leq \theta$
		\item\label{slopes_for_nearby} For every mapping $h\colon X\to Y$ with $\sup_{\mb{x}\in X}\lnorm{Y}{h(\mb{x})-g(\mb{x})}\leq \theta\delta/8$, every $\mb{x}\in E$ and every $\mb{y}\in X$ with $\norm{\mb{y}}_{X}\leq \delta$ we have
		\begin{equation*}
		\norm{h(\mb{x}+\mb{y})-h(\mb{x})-T(\mb{y})}_{Y}\leq \theta\delta.
		\end{equation*}
	\end{enumerate}
\end{lemma}
\begin{proof}
	Choose $\rho>0$ so that $\cl{B}_{X}(E,\rho)\subseteq H$. Next, exploit the uniform continuity of the partial derivatives of $f$ on the compact set $\cl{B}_{X}(E,\rho)$ to find $\tau\in(0,\rho)$ such that
	\begin{equation*}
	\mb{x},\mb{y}\in \cl{B}_{X}(E,\rho),\, \norm{\mb{y}-\mb{x}}_{X}\leq 2\tau\qquad \Rightarrow\qquad \opnorm{Df(\mb{y})-Df(\mb{x})}\leq \zeta,
	\end{equation*}
	where 
	\begin{equation}\label{eq:zeta}
	\zeta:=\min\left\{\frac{1-\max\set{\lip(f),\opnorm{T}}}{3},\frac{\theta}{4}\right\}.
	\end{equation}
	Let $(\gamma_{k})_{k\in\N}$ be a locally finite, smooth partition of unity subordinated to the family $\set{B_{X}(\mb{x},\tau)\colon \mb{x}\in E}$ and for each $k\in\N$ choose $\mb{x}_{k}\in E$ such that $\supp \gamma_{k}\subseteq B_{X}(\mb{x}_{k},\tau)$. Apply Lemma~\ref{lemma:sequence} to $X,\ Y,\ E,\ H_0=H,\ Q,\ f\dow{0}=f,\ \eta=\zeta$, $T\dow{2k}=T$, $T\dow{2k-1}=-Df(\mb{x}_{k})$, $\varphi_{2k-1}=\varphi_{2k}=\gamma_{k}$ and $\theta_k=2^{-k}\theta$ for each $k\in\N$
	to obtain sequences of open sets $(H_{j})_{j\in\N}$, Lipschitz mappings $(f\dow{j}\colon Q\to Y)_{j\in\N}$ and functions $(\psi_{j}\colon X\to[0,1])_{j\in\N}$. 
	By Lemma~\ref{lemma:sequence}\eqref{conv} and by the choice of $\theta_j$, the sequence of Lipschitz mappings $(f\dow{j})_{j\in\N}$ converges, in $(C(Q),\lnorm{\infty}{\cdot})$, to a continuous mapping $g\colon Q\to Y$ satisfying~\eqref{g_theta_close} of the present lemma. For each $\mb{x}\in H$ and for each $j\ge1$ we may write
	\begin{equation}\label{eq:def_aj_bj}
	Df(\mb{x})+\sum_{m=1}^{2j}\psi_{m}(\mb{x})T\dow{m}=a_jDf(\mb{x})+b_jT+P_j(\mb{x}),
	\end{equation} 
	where
	\begin{equation}\label{eq:a_j_b_j}
	a_j=1-\sum_{m=1}^{j}\psi_{2m-1}(\mb{x}),
	\qquad 
	b_j=\sum_{m=1}^{j}\psi_{2m}(\mb{x}) 
	\end{equation}
	and
	\begin{equation*}
	P_j(\mb{x})=\sum_{m=1}^{j}\psi_{2m-1}(\mb{x})(Df(\mb{x})-Df(\mb{x}_{m}))\in \mc{L}(X,Y).
	\end{equation*} 
	We now derive a bound on $\opnorm{P_j(\mb{x})}$. From Lemma~\ref{lemma:sequence}\eqref{psij},
\begin{equation*}
\opnorm{P_j(\mb{x})}\le\sum_{m=1}^{j}\varphi_{2m-1}(\mb{x})\opnorm{Df(\mb{x})-Df(\mb{x}_{m})}=\sum_{m=1}^{j}\gamma_{m}(\mb{x})\opnorm{Df(\mb{x})-Df(\mb{x}_{m})}. 
\end{equation*}
If the $m$th term of the latter sum is not equal to $0$, then 
$\mb{x}\in \supp \gamma_{m}\subseteq B_{X}(\mb{x}_{m},\tau)$ and therefore the choice of $\tau$ implies an upper bound of $\gamma_m(\mb{x})\zeta$ for the $m$th term. Using that $\gamma_m$ is a partition of unity, we conclude that 
	\begin{equation}\label{eq:norm_Pj}
	\opnorm{P_j(\mb{x})}\leq 
	\sum_{m=1}^{j}\gamma_{m}(\mb{x})\zeta	
	\leq \zeta.
	\end{equation}
	To estimate the norms of remaining terms on the right hand side of~\eqref{eq:def_aj_bj} we first show that
	\begin{equation}\label{eq:bounds_aj_bj}
	a_{j},b_{j}\geq0, \qquad a_{j}+b_{j}\leq 1\qquad \text{for all $j\in\N$.}
	\end{equation}
	It is easy to see $b_j\ge0$ as $\psi_{2m}(\mb{x})\ge0$ for all $m$ and all $\mb{x}$. Moreover, $a_j\ge0$ holds because $\gamma_m$ form a partition of unity, hence $\sum_{m=1}^j\psi_{2m-1}(\mb{x})\le\sum_{m=1}^j\varphi_{2m-1}(\mb{x})=\sum_{m=1}^j\gamma_{m}(\mb{x})\le1$. To see $a_j+b_j\le1$, we 
	use both inequalities of Lemma~\ref{lemma:sequence}\eqref{psij} and $\varphi_{2m}(\mb{x})=\varphi_{2m-1}(\mb{x})$ for each $m\ge1$ to obtain from~\eqref{eq:a_j_b_j}
	\begin{equation*}
	a_j+b_j
	=1+\sum_{m=1}^j\br{\psi_{2m}(\mb{x})-\psi_{2m-1}(\mb{x})}
	\le
	1+	\sum_{m=1}^j\1_{H_{2m-1}}(\mb{x})\br{\varphi_{2m}(\mb{x})-\varphi_{2m-1}(\mb{x})}
	=1.
	\end{equation*}
	Next, we apply 
	Lemma~\ref{lemma:sequence}\eqref{fj_diff_ae} and~\eqref{der} with $L=0$, followed by~\eqref{eq:def_aj_bj}, \eqref{eq:bounds_aj_bj}, \eqref{eq:norm_Pj} and~\eqref{eq:zeta}
	to get, for Lebesgue almost all $\mb{x}\in H$,
	\begin{align*}
	\opnorm{Df\dow{2j}(\mb{x})}
	&\leq \opnorm{Df(\mb{x})+\sum_{m=1}^{2j}\psi_{m}(\mb{x})T\dow{m}}+\zeta
	=\opnorm{a_jDf(\mb{x})+b_jT+P_j(\mb{x})}+\zeta
	\\
	&\leq a_j\opnorm{Df(\mb{x})}+b_j\opnorm{T}+\opnorm{P_j(\mb{x})}+\zeta\\
	&\leq a_j\lip(f)+b_j\opnorm{T}+2\zeta 
	\leq \max\set{\lip(f),\opnorm{T}}+2\zeta\leq 1.
	\end{align*}
	Since $\opnorm{Df\dow{2j}(\mb{x})}\le1$ for almost every $\mb{x}\in H$ and $f\dow{2j}\in\lip(Q,Y)$ we have that $f\dow{2j}$ is $1$-Lipschitz on each open ball $B\subseteq H$; see Corollary~\ref{cor:lip-loc_glob}. Hence $f\dow{2j}$ is locally $1$-Lipschitz on the open set $H$. Moreover, $f\dow{2j}|_{Q\setminus H}=f\dow{0}=f$ by Lemma~\ref{lemma:sequence}\eqref{conv} and the fact that $\supp \varphi_{2m-1}=\supp\varphi_{2m}=\supp \gamma_{m}\subseteq B_{X}(\mb{x}_{m},\tau)\subseteq B_{X}(E,\rho)\subseteq H$ for each $1\le m\le j$. Therefore $\lip(f\dow{2j}|_{Q\setminus H})\leq \lip(f|_{Q\setminus H})\leq 1$. Since $Q$ is closed, it follows that $f\dow{2j}\in\lip_1(Q,Y)$; see Lemma~\ref{lemma:awkward_lipschitz_bound}. Since $j\in\N$ was arbitrary, we deduce that $g=\lim f\dow{2j}\in\lip_1(Q,Y)$ too. 
	
	As $(\gamma_k)_{k\in\N}$ is a locally finite partition of unity,
	for each $\mb{x}\in E$ there exist an open set $U_{\mb{x}}\subseteq H$ and a number $k_{\mb{x}}\in\N$ such that $\mb{x}\in U_{\mb{x}}$ and $\supp \gamma_{k}\cap U_{\mb{x}}=\emptyset$ for all $k>k_{\mb{x}}$. For each $\mb{x}\in E$ let $V_{\mb{x}}=U_{\mb{x}}\cap {H}_{k_{\mb{x}}}$, where ${H}_{k_{\mb{x}}}$ is defined by Lemma~\ref{lemma:sequence}\eqref{Hjs}. Then $V_{\mb{x}}$ is an open set containing ${\mb{x}}$ and contained in $U_{\mb{x}}\subseteq H$.By the second half of Lemma~\ref{lemma:sequence}\eqref{conv}, $g|_{V_{\mb{x}}}=f\dow{k}|_{V_{\mb{x}}}=f\dow{k_{\mb{x}}}|_{V_{\mb{x}}}$ for every $\mb{x}\in E$ and $k\ge k_{\mb{x}}$. Therefore,
	\begin{equation}\label{eq:Dg}
	Dg({\mb{x}})=Df\dow{k}({\mb{x}}) 
	\text{ for all }
	{\mb{x}}\in E \text{ and } k\ge k_{\mb{x}}
	\end{equation}
	and $g\in C^1(U,Y)$, where the open set
	\begin{equation*}
	U:=\bigcup_{\mb{x}\in E}V_{\mb{x}}
	\end{equation*}
	satisfies $E\subseteq U\subseteq H$.
	Note now that~\eqref{U_between_E_and_H} of the present lemma is satisfied and that we have established $g\in \lip_1(Q,Y)\cap C^1(U,Y)$.
	
	Apply now Lemma~\ref{lemma:approximation}
	to $X$, $Y$, $E\subseteq U\subseteq Q$, the mapping $g\colon Q\to Y$ and 
	$\eta=\theta/2$, to find $\delta\in(0,\theta/2)$ such that whenever $h\colon Q\to Y$ satisfies 
	$\sup_{\mb{x}\in Q}\lnorm{Y}{h(\mb{x})-g(\mb{x})}\leq \theta \delta/8$, for
	every $\mb{x}\in E$ and every $\mb{y}\in X$ with $\lnorm{X}{\mb{y}}\leq \delta$ we have
	\begin{equation}\label{eq:lemma_approximation}
	\lnorm{Y}{h(\mb{x}+\mb{y})-h(\mb{x})-Dg(\mb{x})(\mb{y})}\leq \theta\delta/2.
	\end{equation}
	Our aim now is to replace $Dg(\mb{x})$, for $\mb{x}\in E$, by $T$ in~\eqref{eq:lemma_approximation}, to get~\eqref{slopes_for_nearby}.
	For this, fix $\mb{x}\in E$ and use $E\subseteq H_{j}$ for all $j\ge0$ (from Lemma~\ref{lemma:sequence}\eqref{Hjs}) and Lemma~\ref{lemma:sequence}\eqref{psij}  to conclude that 
	\begin{equation}\label{eq:psi=phi_on_E}
	\psi_i(\mb{x})=\varphi_i(\mb{x}) \qquad \text{for all $i\geq 1$.}
	\end{equation}
	Hence, letting $j=k_{\mb{x}}$ in~\eqref{eq:a_j_b_j} we get, using $\sum_{m=1}^{k_{\mb{x}}}\gamma_{m}(\mb{x})=\sum_{m=1}^{\infty}\gamma_{m}(\mb{x})
	=1$, 
	\begin{align*}
	a_{k_{\mb{x}}}
	&=1-\sum_{m=1}^{k_{\mb{x}}}\psi_{2m-1}(\mb{x})
	=1-\sum_{m=1}^{k_{\mb{x}}}\varphi_{2m-1}(\mb{x})
	=1-\sum_{m=1}^{k_{\mb{x}}}\gamma_{m}(\mb{x})
	=0,\\
	\text{and}\\
	b_{k_{\mb{x}}}
	&=\sum_{m=1}^{k_{\mb{x}}}\psi_{2m}(\mb{x})=\sum_{m=1}^{k_{\mb{x}}}\varphi_{2m}(\mb{x})
	=\sum_{m=1}^{k_{\mb{x}}}\gamma_{m}(\mb{x})
	=1.
	\end{align*}
	Substituting, for $\mb{x}\in E$, $a_{k_{\mb{x}}}=0$, $b_{k_{\mb{x}}}=1$ and~\eqref{eq:psi=phi_on_E} into~\eqref{eq:def_aj_bj} we obtain
	\begin{equation*}
	Df(\mb{x})+\sum_{m=1}^{2{k_{\mb{x}}}}\varphi_{m}(\mb{x})T\dow{m}-T
	=P_{k_{\mb{x}}}(\mb{x}).	
	\end{equation*}
	Thus by~\eqref{eq:Dg}, Lemma~\ref{lemma:sequence}\eqref{der} applied to $L=T$, \eqref{eq:psi=phi_on_E}, \eqref{eq:norm_Pj} and, finally,~\eqref{eq:zeta}  
	\begin{align*} 
	\opnorm{Dg(\mb{x)}-T}
	&=\opnorm{Df\dow{2k_{\mb{x}}}(\mb{x})-T}
	\le
	\opnorm{Df(\mb{x})+\sum_{m=1}^{2k_{\mb{x}}}\varphi_{m}(\mb{x})T\dow{m}-T
	}+\zeta\\
	&=\opnorm{P_{k_{\mb{x}}}}+\zeta
	\le2\zeta\le\theta/2,
	\end{align*}
	which together with~\eqref{eq:lemma_approximation}, implies~\eqref{slopes_for_nearby}.
\end{proof}
We are now ready to prove Theorem~\ref{thm:typical_nondiff-dash} which we restate again. 
\typicalnondiff*
\begin{proof}
	Let $\Lip=\lip_1(Q,Y)$. 	
	Since $Y$ is Banach, we may assume without loss of generality that $Q$ is closed; the mapping $\iota\colon \lip_{1}(\cl{Q},Y)\to\Lip$, $f\mapsto f|_{Q}$ defines a surjective isometry $\lip_{1}(\cl{Q},Y)\to \Lip$ and so a set $\mc{R}\subseteq \lip_{1}(\cl{Q},Y)$ is residual in $\lip_{1}(\cl{Q},Y)$ if and only if $\iota(\mc{R})$ is residual in $\Lip$. Note also that $\mc{D}_{f}(\mb{x})\supseteq \Ball_{W}$ for $f\in\lip_{1}(\cl{Q},Y)$ and $\mb{x}\in E$ if and only if $\mc{D}_{\iota(f)}(\mb{x})\supseteq \Ball_{W}$. 
	
	Using that an intersection of countably many residual subsets of $(\Lip,\lnorm{\infty}{\cdot})$ is residual, we may assume without loss of generality that $E$ is a compact purely unrectifiable set.
	Further, since $\mc{D}_{f}(\mb{y})$ is closed by Lemma~\ref{lemma:Df_separable} and contained in $\Ball_{\mc{L}(X,Y)}$ for every $f\in\Lip$ and $\mb{y}\in \inter Q$, and $\inter\Ball_{\mc{L}(X,Y)}$ contains a countable subset dense in $\Ball_{W}$, it suffices to prove that the set $S_{L}=\{f\in\Lip
	\colon L\in\mc{D}_{f}(\mb{x})\text{ for each }\mb{x}\in E\}$ is residual for each $L\in \inter \Ball_{\mc{L}(X,Y)}$, i.e.\ $S_L$ is residual in $(\Lip,\lnorm{\infty}{\cdot})$ whenever $\opnorm{L}<1$. So, fixing $L\in{\mc{L}(X,Y)}$ with $\opnorm{L}<1$, we now describe a winning strategy for Player~II for the Banach-Mazur game in $(\Lip,\lnorm{\infty}{\cdot})$ with the target set $S=S_{L}$. 
	
	In the $n$th round of the game, Player~I and Player~II will construct open balls $B_\Lip(f_n,r_n)$ and $B_\Lip(g_n,d_n)$ respectively, centred at $f_n,g_n\in\Lip$, such that
	\begin{equation*}
	B_\Lip(f_{n+1},r_{n+1})\subseteq B_\Lip(g_n,d_n)\subseteq B_\Lip(f_n,r_n) 
	\text{ for each }n\ge1.
	\end{equation*}
	We define Player II's winning strategy as follows. Let $n\ge1$ be fixed and assume that Player~I has made their $n$th move $B_\Lip(f_n,r_n)$. Let $f_{n}\upp{1}\in B_{\Lip}(f_{n},r_{n})$ be such that $\lip(f_{n}\upp{1})<1$; such $f_{n}\upp{1}$ may be taken of the form $qf_{n}$ for $q\in (0,1)$ chosen sufficiently close to $1$. Next, we apply a smoothing result for Lipschitz mappings on finite-dimensional spaces, Lemma~\ref{lemma:smooth_PlayerI's_mapping_around_E}, to $f=f_{n}\upp{1}\in B_{\Lip}(f_{n},r_{n})$ to find an open set $H\subseteq X$ with $E\subseteq H\subseteq Q$ and a mapping 
	$f_n\upp{2}\in \lip(Q,Y)\cap C^1(H,Y)$ such that $\lip(f_n\upp{2})<1$, thus $f_n\upp{2}\in\Lip$, and $f_n\upp{2}\in B_{\Lip}(f_{n},r_{n})$. Fix any \[0<\eta_n<\min\{2^{-n},r_n/3,1-\lip(f_n\upp{2}),1-\opnorm{L},r_n-\lnorm{\infty}{f_n-f_n\upp{2}}\} \]
	and apply Lemma~\ref{lemma:bmgame_step} with $X$, $Y$, $E$, $H$, $Q$, $\theta=\eta_n$,  $f=f_n\upp{2}$ and $T=L$ to find
	$g_n\coloneqq g\in\lip_1(Q,Y)$ and $\delta_n\coloneqq \delta$. Then $g_n\in B_{\Lip}(f_n\upp{2},\eta_n)\subseteq B_{\Lip}(f_n,r_n)$. Fix $d_n\in(0,\theta\delta/8)=(0,\eta_n\delta_n/8)$ such that $\cl{B}_\Lip(g_n,d_n)\subseteq B_\Lip(f_n,r_n)$. Player~II plays $B_\Lip(g_n,d_n)$ as their $n$th move. 
	
	Player II's strategy ensures that $\cl{B}_{\Lip}(g_{i},d_{i+1})\subseteq B_{\Lip}(g_{i},d_{i})$ for all $i\in \N$, where $d_{i}\to 0$ as $i\to\infty$. Therefore, the intersection $\bigcap_{n=1}^\infty {B}_\Lip(g_n,d_n)$ is a single function $h\in\Lip$. We now show that $h\in S_L$, i.e.\ $L\in\mc{D}_{h}(\mb{x})$ for each $\mb{x}\in E$. 
	
	For an arbitrary $\varepsilon>0$ let $n\ge1$ be such that $\eta_n<\varepsilon$. Since $h\in B_\Lip(g_n,d_n)\subseteq B_\Lip(g_n,\eta_n\delta_n/8)$, we conclude, by Lemma~\ref{lemma:bmgame_step}\eqref{slopes_for_nearby}, that for all $\mb{x}\in E$
	\[
	\sup_{\mb{y}\in\cl{B}_{X}(\mb{0}_X,\delta_n)}
	\frac{\lnorm{Y}{h(\mb{x}+\mb{y})-h(\mb{x})-L(\mb{y})
		}
	}{\delta_n}	
	\le\eta_n<\varepsilon.
	\]
	Thus for all $\mb{x}\in E$
	\[
	\liminf_{\delta\to0+}
	\sup_{\mb{y}\in\cl{B}_{X}(\mb{0}_X,\delta)}
	\frac{\lnorm{Y}{h(\mb{x}+\mb{y})-h(\mb{x})-L(\mb{y})
		}
	}{\delta}	=0,
	\]
	and so $L\in\mc{D}_{h}(\mb{x})$. Thus $h\in S_L$, which finishes the proof that $S_L$ is residual in $\Lip$.
\end{proof}
\section{Smooth approximation of Lipschitz mappings.}\label{sec:SA}
This section is devoted to results which guarantee existence of smooth approximations of Lipschitz mappings. In particular, in Lemma~\ref{lemma:SLA} we show that it is possible to `assemble' such an approximation from many pieces related to partition of unity, and in Theorem~\ref{thm:johanis} we extend a result of~\cite{johanis2003approximation}, in order to make it applicable in more general settings, in particular in the proof of Theorem~\ref{thm:typical_nondiff-dash}. 
\begin{define}\label{def:sla}
	Let $X,Y$ be normed spaces, $U\subseteq X$ be open, $h\in\lip(U,Y)$, $\theta>0$. We define the set $\SLA(h,U,Y,\theta)$ of smooth Lipschitz approximations of $h$ over $U$ with error $\theta$ as the collection of mappings
	$g\in\lip(U,Y)\cap C^1(U,Y)$ such that
	$\lnorm{Y}{g(\mb{x})-h(\mb{x})}\le\theta$ for every $\mb{x}\in U$.
\end{define}
\begin{lemma}\label{lemma:SLA}
	Let $X$ and $Y$ be normed spaces. Let $Q\subseteq X$ be closed, $h\in\lip(Q,Y)$
	and $U\subseteq Q$ be open and such that it admits a locally finite, $C^{1}$-smooth partition of unity $(\varphi_{k})_{k\in\N}$ with $\supp(\varphi_k)\subseteq U$ for all $k\ge1$. Let $\theta>0$ and $\theta_k>0$ be such that $\sum_{k\ge1}(1+\lip(\varphi_k))\theta_k\le\theta$.
	Let $A_k$ be open sets such that $\supp(\varphi_k)\subseteq A_k\subseteq U$ for all $k\ge1$, and for each $k\ge1$, let  
	$h_k\colon U\to Y$ be a mapping such that $h_{k}\in \SLA(h,A_{k},Y,\theta_{k})$. Then the mapping $\tilde{h}\colon Q\to Y$,
		\begin{equation}\label{eq:def_H}
			\tilde{h}(\mb{x})=\begin{cases}
				\sum_{k\in\N}\varphi_{k}(\mb{x})h_{k}(\mb{x}) & \text{if }\mb{x}\in U,\\
				h(\mb{x}) & \text{if }\mb{x}\in Q\setminus U
			\end{cases}
		\end{equation}
has the following properties.
	\begin{enumerate}[(i)]
		\item\label{SLAi} We have
 $\tilde{h}\in C(Q,Y)\cap C^1(U,Y)$ and $\lnorm{Y}{\tilde{h}(\mb{x})-h(\mb{x})}\le\theta$ for all $\mb{x}\in Q$, and 
		$\lip(\tilde{h})\le\max(\lip(h),\theta+\sup_{k\ge1}\lip(h_{k}|_{A_k}))$.
		In particular, if $\sup_{k\ge1}\lip(h_{k}|_{A_k})$ is finite, then $\tilde{h}\in\lip(Q,Y)$.
		\item\label{SLAii} If, additionally, for each $\mb{x}\in U$ there is $P(\mb{x})\in \mc{L}(X,Y)$ and $\eta(\mb{x})>0$ such that 
		\begin{equation*}
		\opnorm{Dh_{k}(\mb{x})-P(\mb{x})}\leq \eta(\mb{x})\qquad \text{for all }\mb{x}\in A_{k}\text{ and }k\in\N, 
		\end{equation*}
		then 
		$\opnorm{D\tilde{h}(\mb{x})-P(\mb{x})}\leq \eta(\mb{x})+\sum_{k\in\N}\lip(\varphi_{k})\1_{\supp (\varphi_{k})}(\mb{x})\theta_{k}$ for all $\mb{x}\in U$.
	\end{enumerate}
\end{lemma}
\begin{proof}
Observe that due to the local finiteness property of the partition of unity, for each point $\mb{x}\in U$ there is an open neighbourhood $V(\mb{x})\subseteq U$ such that
	the set 
	\begin{equation}\label{eq:defI(x)}
	I(\mb{x})=\set{k\ge1\colon  V(\mb{x})\cap\supp(\varphi_k)\ne\emptyset}
	\end{equation}
	is finite. Hence for any $\mb{x}\in U$ 
	\begin{equation}\label{eq:HisC1}
	\tilde{h}(\mb{y})=\sum_{k\in I(\mb{x})}\varphi_{k}(\mb{y})h_{k}(\mb{y}) 
	\quad \text{for all}\quad 
	\mb{y}\in V(\mb{x}),
	\end{equation}
	and thus $\tilde{h}|_{V(\mb{x})}$ is a finite sum of mappings $C^1$-smooth on $V(\mb{x})$.
	Hence by the arbitrariness of $\mb{x}\in U$ we conclude $\tilde{h}\in C^{1}(U,Y)$. 
	We now show that $\tilde{h}$ is continuous on $Q$. The only points of the domain at which continuity of $\tilde{h}$ is unclear are those in $\partial U$. Let $\mb{x}_{0}\in \partial U\cap Q$ and let $\theta>0$ be arbitrary. There exists $N\in\N$ large enough so that 
	\begin{equation}\label{eq:choice_N_new}
	\sum_{k=N+1}^{\infty}\theta_{k}<\theta/2. 
	\end{equation}
	Next, choose 
	\begin{equation}\label{eq:choice_eta_new}
	0<\eta<\frac{\theta}{2\br{\lip(h)+1}}
	\end{equation}
	sufficiently small so that 
	\begin{equation}\label{eq:evading_supports_new}
	B_{X}(\mb{x}_{0},\eta)\cap \supp (\varphi_{i})=\emptyset\qquad  \text{for $i=1,\ldots,N$,}
	\end{equation}
	which is possible because each $\supp(\varphi_{i})$ is contained in the open set $U$, whilst $\mb{x}_{0}\notin U$, so $\mb{x}_0\notin\supp(\varphi_i)$. Let $\mb{x}\in B_{X}(\mb{x}_{0},\eta)\cap Q$ be arbitrary. 
	
	If $\mb{x}\in Q\setminus U$ we have $\lnorm{Y}{\tilde{h}(\mb{x})-\tilde{h}(\mb{x}_{0})}=\lnorm{Y}{h(\mb{x})-h(\mb{x}_{0})}\leq \lip(h)\eta\leq \theta$. Now assume that $\mb{x}\in U$. Then, using~\eqref{eq:evading_supports_new} and that $(\varphi_{k})_{k\in\N}$ are a partition of unity on $U$, we get
	\begin{multline*}
	\tilde{h}(\mb{x})-\tilde{h}(\mb{x}_{0})=\sum_{k=1}^{\infty}\varphi_{k}(\mb{x})(h_{k}(\mb{x})-h(\mb{x}_{0}))=\sum_{k=N+1}^{\infty}\varphi_{k}(\mb{x})(h_{k}(\mb{x})-h(\mb{x}_{0}))\\
	=\sum_{k=N+1}^{\infty}\varphi_{k}(\mb{x})(h_{k}(\mb{x})-h(\mb{x}))+\sum_{k=N+1}^{\infty}\varphi_{k}(\mb{x})(h(\mb{x})-h(\mb{x}_{0})).
	\end{multline*}
	Using this with~\eqref{eq:choice_N_new} and~\eqref{eq:choice_eta_new}, as well as $h_k\in\SLA(h,A_k,Y,\theta_k)$ which implies  $\lnorm{Y}{h_{k}(\mb{x})-h(\mb{x})}\leq \theta_{k}$ for $\mb{x}\in \supp(\varphi_k)\subseteq A_k$,  we derive
	\begin{align*}
	\lnorm{Y}{\tilde{h}(\mb{x})-\tilde{h}(\mb{x}_{0})}
	&\leq\sum_{k=N+1}^{\infty}\varphi_{k}(\mb{x})\lnorm{Y}{h_{k}(\mb{x})-h(\mb{x})}+\sum_{k=N+1}^{\infty}\varphi_{k}(\mb{x})\lnorm{Y}{h(\mb{x})-h(\mb{x}_{0})}\\
	&\leq  \sum_{k=N+1}^{\infty}\theta_{k}+\lip(h)\eta\leq \theta.
	\end{align*}
	Thus, we have established $\lnorm{Y}{\tilde{h}(\mb{x})-\tilde{h}(\mb{x}_{0})}\leq \theta$ for all $\mb{x}\in B_{X}(\mb{x}_{0},\eta)\cap Q$. Since $\theta>0$ was arbitrary, this verifies the continuity of $\tilde{h}$ at $\mb{x}_{0}$. This completes the proof that $\tilde{h}$ is continuous on $Q$.
	
	For each $\mb{x}\in Q\setminus U$ we have $\lnorm{Y}{\tilde{h}(\mb{x})-h(\mb{x})}=0$ and for each $\mb{x}\in U$ we have, using that $(\varphi_{k})_{k\in\N}$ is a partition of unity on $U$, $\lnorm{Y}{h_{k}(\mb{x})-h(\mb{x})}\leq \theta_{k}$ for $\mb{x}\in\supp(\varphi_k)$ and $\sum\theta_k\le\theta$,
	\begin{align*}
	\lnorm{Y}{\tilde{h}(\mb{x})-h(\mb{x})}&=\lnorm{Y}{\sum_{k\in\N}\varphi_{k}(\mb{x})(h_{k}(\mb{x})-h(\mb{x}))}\\
	&\leq \sum_{k\in\N}\varphi_{k}(\mb{x})\lnorm{Y}{h_{k}(\mb{x})-h(\mb{x})}\leq \sum_{k\in\N}\theta_{k}\leq \theta.
	\end{align*}
	Therefore $\lnorm{Y}{\tilde{h}(\mb{x})-h(\mb{x})}\leq \theta$ for all $\mb{x}\in Q$.
	
	It only remains to verify the desired bound on the Lipschitz constant of $\tilde{h}$. Consider first $\mb{x}\in U$. Using that $(\varphi_{k})_{k\in\N}$ is a locally finite smooth partition of unity, so $\sum_{k\in\N}\varphi_k$ is identically $1$ on $U$, we conclude $\sum_{k\in\N}D\varphi_{k}(\mb{x})=0$. However $D\varphi_{k}(\mb{x})=0$ for each $k\notin I(\mb{x})$, see~\eqref{eq:defI(x)}, so we also conclude $\sum_{k\in I(\mb{x})}D\varphi_k(\mb{x})=0$.
	Recall that $\tilde{h}$ can be written as a finite sum of  $C^1$-smooth terms in an  open neighbourhood of $\mb{x}$, see~\eqref{eq:HisC1}; thus
	\begin{multline}\label{eq:DH}
	D\tilde{h}(\mb{x})=\sum_{k\in I(\mb{x})} \varphi_{k}(\mb{x})Dh_{k}(\mb{x})+\sum_{k\in I(\mb{x})}D\varphi_{k}(\mb{x})h_{k}(\mb{x})\\
	=\sum_{k\in I(\mb{x})}\varphi_{k}(\mb{x})Dh_{k}(\mb{x})+\sum_{k\in I(\mb{x})}D\varphi_{k}(\mb{x})h_k(\mb{x})-\sum_{k\in I(\mb{x})}D\varphi_{k}(\mb{x})h(\mb{x}).
	\end{multline}
	We deduce, using the properties of $h_{k}\in\SLA(h,A_k,Y,\theta_k)$ for each $k\in I(\mb{x})$, that
	\begin{align*}
	\opnorm{D\tilde{h}(\mb{x})}
	&\leq \sum_{k\in I(\mb{x})}\varphi_{k}(\mb{x}) \opnorm{Dh_{k}(\mb{x})} +\sum_{k\in I(\mb{x})}\lnorm{X^*}{D\varphi_{k}(\mb{x})}\lnorm{Y}{h_{k}(\mb{x})-h(\mb{x})}\\
	&\leq \sup_{k\in I(\mb{x})}\lip(h_{k}|_{A_k})+\sum_{k\in I(\mb{x})}
	\lip(\varphi_{k})\theta_{k}\leq \sup_{k\ge1}\lip(h_{k}|_{A_k})+\theta.
	\end{align*}
	Since this inequality has been established for an arbitrary point $\mb{x}$ in the open set $U$, we conclude that $\tilde{h}$ is locally $\br{\sup_{k\ge1}\lip(h_{k}|_{A_k})+\theta}$-Lipschitz on $U$, if this constant is finite. In summary, we have now established that  $\tilde{h}$ is continuous, $\tilde{h}$ is locally $\br{\sup_{k\ge1}\lip(h_{k}|_{A_k})+\theta}$-Lipschitz on $U$, whilst by definition of $\tilde{h}$, $\lip(\tilde{h}|_{Q\setminus U})\leq \lip (h)$. Therefore, $\tilde{h}$ is Lipschitz with $\lip(\tilde{h})\leq \max(\lip(h),\sup_{k\ge1}\lip(h_{k}|_{A_k})+\theta)$; see Lemma~\ref{lemma:awkward_lipschitz_bound}. This finishes the proof of~\eqref{SLAi}.
	
	To prove~\eqref{SLAii}, we use again~\eqref{eq:DH} to get, for every $\mb{x}\in U$, 
	\begin{align*}
	&\opnorm{
		D\tilde{h}(\mb{x})-P(\mb{x})
	}\\
	&\le 
	\sum_{k\in I(\mb{x})}\varphi_{k}(\mb{x})\opnorm{Dh_{k}(\mb{x})-P(\mb{x})}+\sum_{k\in I(\mb{x})}\lnorm{X^*}{D\varphi_{k}(\mb{x})}\lnorm{Y}{
		h_k(\mb{x})-h(\mb{x})
	}\\
	&\le 
	\eta(\mb{x})
	+\sum_{k\in\N}\lip(\varphi_{k})\1_{\supp (\varphi_{k})}(\mb{x})\theta_{k}.\qedhere
	\end{align*}
\end{proof}

\begin{define}\label{def:unif-gat}
	If $X$ and $Y$ are normed spaces, $U\subseteq X$ is open and $f\colon U\to Y$ is \Gat differentiable at every point of $U$, we say that $f$ is uniformly \Gat differentiable on $U$ if for each pair of $\mb{h}\in \Sph_{X}$ and $\varepsilon>0$ there exists $\delta=\delta(\mb{h},\varepsilon)>0$ such that
	\begin{equation}\label{eq:unif-gat}
		\lnorm{Y}{f(\mb{x}+t\mb{h})-f(\mb{x})-tD_{G}f(\mb{x})(\mb{h})}\leq \varepsilon\abs{t}
	\end{equation}
	whenever $\mb{x}\in U$ and  $\abs{t}\le \min\set{\dist(\mb{x},X\setminus U),\delta}$.
\end{define}	
\begin{lemma}\label{lemma:uni_gat_implies_smooth}
	Let $X$ and $Y$ be Banach spaces, $U\subseteq X$ be open and $f\colon U\to Y$ be uniformly \Gat differentiable. Then $f\in C^{1}(U,Y)$. 
\end{lemma}
\begin{proof}
	Fix $\mb{x}_{0}\in U$ and choose $r>0$ so that $B_{X}(\mb{x}_{0},r)\subseteq U$. We verify continuity of the \Gat derivative $D_{G}f$ of $f$ at $\mb{x}_{0}$. By~\cite[Prop~4.2]{benyamini1998geometric}, this will imply \Fre differentiability of $f$ on $U$ and therefore that $f\in C^{1}(U)$. Given $\mb{u}\in\Sph_{X}$ and $\varepsilon\in(0,1)$, exploit the uniform \Gat differentiability of $f$ on $U$, see~\eqref{eq:unif-gat} and let $\delta=\min(\delta(\mb{u},\varepsilon/4),r/2)$ 
	Then for all $\mb{x}\in B_{X}(\mb{x}_{0},r/2)$ and $t\in [-\delta,\delta]$,
	\begin{equation*}
		\lnorm{Y}{f(\mb{x}+t\mb{u})-f(\mb{x})-tD_{G}f(\mb{x})(\mb{u})}
		\leq \frac{\varepsilon \abs{t}}{4}
	\end{equation*}
	Then for every $\mb{x}\in U$ with $\lnorm{X}{\mb{x}-\mb{x}_{0}}\leq \frac{\varepsilon\delta}{4\br{\lip(f)+1}}$ we use the above inequality with $t=\delta$ to get
	\begin{align*}
		\delta&\lnorm{Y}{\br{D_{G}f(\mb{x})-D_{G}f(\mb{x}_{0})}(\mb{u})}\\ &\leq \lnorm{Y}{\delta D_{G}f(\mb{x})(\mb{u})-\br{f(\mb{x}+\delta\mb{u})-f(\mb{x})}}
		+\lnorm{Y}{f(\mb{x}+\delta\mb{u})-f(\mb{x}_{0}+\delta\mb{u})}\\&+\lnorm{Y}{f(\mb{x}_{0})-f(\mb{x})} +\lnorm{Y}{f(\mb{x}_{0}+\delta\mb{u})-f(\mb{x}_{0})-\delta D_{G}f(\mb{x}_{0})(\mb{u})}\\
		&\leq 2\cdot\frac{\varepsilon\delta}{4}+2\lip(f)\lnorm{X}{\mb{x}-\mb{x}_{0}}\leq \varepsilon\delta.
	\end{align*}
	Since $\mb{u}\in\Sph_{X}$ was arbitrary, this shows $\opnorm{D_{G}f(\mb{x})-D_{G}f(\mb{x}_{0})}\leq \varepsilon$.
\end{proof}

The next theorem is a generalisation of a smooth approximation result of Johanis~\cite{johanis2003approximation}. The difference to~\cite{johanis2003approximation} is that the following statement treats Lipschitz mappings defined only on a subset of a Banach space, whereas~\cite{johanis2003approximation} provides smooth approximations only for mappings defined on the whole space $X$, and Lipschitz mappings defined on subsets of Banach spaces may not necessarily be extended to the whole space. In other words, the result of~\cite{johanis2003approximation} is the special case of the following theorem, where we set $Q=X$. This local version of Johanis's smooth approximation is a useful statement, independent of the present paper. The proof is an adaptation of the argument of Johanis~\cite{johanis2003approximation}, see also~\cite[Theorem~3.1]{FWZ}, but for completeness we include the full argument.

\begin{thm}\label{thm:johanis}
	Let $X$ and $Y$ be Banach spaces, where $X$ is separable, $Q\subseteq X$ and $\varepsilon>0$ be such that
	\begin{equation}\label{eq:qeps}
	\emptyset\neq\inside{Q}{\varepsilon}\coloneq\set{\mb{x}\in Q\colon \dist_{\lnorm{X}{-}}(\mb{x},X\setminus Q)> \varepsilon}, 
	\end{equation}
	and let $f\colon Q\to Y$ be a Lipschitz mapping.
	Then there exists $g\in\lip(Q_{\varepsilon},Y)\cap C^{1}(Q_{\varepsilon},Y)$ such that 
	$g$ is uniformly \Gat differentiable on $\inside{Q}{\varepsilon}$, 
	$\lip(g)\leq \lip(f)$ and $\lnorm{Y}{g(\mb{x})-f(\mb{x})}\leq \varepsilon$ for all $\mb{x}\in Q_{\varepsilon}$. 
\end{thm}
\begin{proof}
	We follow the argument of~\cite{johanis2003approximation} 
	and make only small adjustments. Let $(\mb{h}_{i})_{i\in \N}$ be a countable dense subset of $\Sph_{X}$; for each $i\in\N$ let $\varphi_{i}\in C^{\infty}(\R)$ be such that $\varphi_{i}(t)\geq 0$ for all $t\in\R$, $\displaystyle\int_{\R}\varphi_{i}(t)\,dt=1$ and 
	\begin{equation*}
	\supp\varphi_{i}\subseteq J_{i}:= \sqbr{\frac{-\varepsilon}{2\br{\lip(f)+1}2^{i}},\frac{\varepsilon}{2\br{\lip(f)+1}2^{i}}}.
	\end{equation*}
	For each $n\in\N$ we let $P_{n}:=\prod_{i=1}^{n}J_{i}$
	and observe that $\mb{x}-\sum_{i=1}^{n}t_{i}\mb{h}_{i}\in Q$ whenever $\mb{x}\in \inside{Q}{\varepsilon}$ and $(t_{1},\ldots,t_{n})\in P_{n}$. For  $n\in\N$ define mappings $g_{n}\colon \inside{Q}{\varepsilon}\to Y$ by
	\begin{equation*}
	g_{n}(\mb{x})=\int_{P_{n}}f\bbr{\mb{x}-\sum_{i=1}^{n}t_{i}\mb{h}_{i}}\prod_{i=1}^{n}\varphi_{i}(t_{i})\,d\lambda_{n},
	\end{equation*}
	where the integral is the Bochner integral and $\lambda_{n}$ denotes the $n$-dimensional Lebesgue measure on $\R^{n}\supseteq P_{n}$. We note, for further reference, that for any $\mb{x}\in\inside{Q}{\varepsilon}$ and $m\ge n$
	\begin{equation}\label{eq:g_n_m}	
	g_{n}(\mb{x})
	=\int_{P_{m}}f\bbr{\mb{x}-\sum_{i=1}^{n}t_{i}\mb{h}_{i}}\prod_{i=1}^{m}\varphi_{i}(t_{i})\,d\lambda_{m}
	=\int_{\R^m}f\bbr{\mb{x}-\sum_{i=1}^{n}t_{i}\mb{h}_{i}}\prod_{i=1}^{m}\varphi_{i}(t_{i})\,d\lambda_{m}.		
	\end{equation}	
	For each $n\in\N$, the following inequalities show that $g_{n}\colon\inside{Q}{\varepsilon}\to Y$ is Lipschitz with $\lip(g_n)\leq \lip(f)$: whenever $\mb{x},\mb{y}\in \inside{Q}{\varepsilon}$,
	\begin{align*}
	\lnorm{Y}{g_{n}(\mb{x})-g_{n}(\mb{y})}&\leq \Jint \lnorm{Y}{f\bbr{\mb{x}-\sum_{i=1}^{n}t_{i}\mb{h}_{i}}-f\bbr{\mb{y}-\sum_{i=1}^{n}t_{i}\mb{h}_{i}}}\prod_{i=1}^{n}\varphi_{i}(t_{i})\,d\lambda_{n}\\
	&\leq \lip(f)\lnorm{X}{\mb{y}-\mb{x}}\Jint \prod_{i=1}^{n}\varphi_{i}(t_{i})\,d\lambda_{n}=\lip(f)\lnorm{X}{\mb{x}-\mb{y}}
	.
	\end{align*} 
	For any $m,n\in\N$ with $m>n$ and any $\mb{x}\in \inside{Q}{\varepsilon}$ we observe, using~\eqref{eq:g_n_m},
	\begin{multline*}
	\lnorm{Y}{g_{m}(\mb{x})-g_{n}(\mb{x})}\leq \Jintm \lnorm{Y}{f\br{\mb{x}-\sum_{i=1}^{m}t_{i}\mb{h}_{i}}-f\br{\mb{x}-\sum_{i=1}^{n}t_{i}\mb{h}_{i}}}\prod_{i=1}^{m}\varphi_{i}(t_{i})\,d\lambda_{m}\\
	\leq \lip(f)\Jintm \lnorm{X}{\sum_{i=n+1}^{m}t_{i}\mb{h}_{i}}\prod_{i=1}^{m}\varphi_{i}(t_{i})\,d\lambda_{m}\leq \lip(f)\sum_{i=n+1}^{\infty}\frac{\varepsilon}{2\br{\lip(f)+1}2^{i}}\leq \frac{\varepsilon}{2^{n}},
	\end{multline*}
	and conclude from this that the sequence of Lipschitz mappings $(g_{n})_{n\in\N}$ with $\lip(g_{n})\leq \lip(f)$ converges uniformly on $\inside{Q}{\varepsilon}$ to a Lipschitz mapping $g\colon \inside{Q}{\varepsilon}\to Y$ with $\lip(g)\le\lip(f)$ and such that  $\lnorm{Y}{g_n(\mb{y})-g(\mb{y})}\le\frac{\varepsilon}{2^{n}}$ for each $\mb{y}\in\inside{Q}{\varepsilon}$ and each $n\in\N$. To see that $g$ is a good approximation of $f$, observe for $\mb{x}\in Q_\varepsilon$ that
	
	\begin{multline*}
	\lnorm{Y}{f(\mb{x})-g(\mb{x})}\leq \lnorm{Y}{f(\mb{x})-g_{1}(\mb{x})}+\lnorm{Y}{g_{1}(\mb{x})-g(\mb{x})} \\
	\leq \int_{P_1} \lnorm{Y}{f(\mb{x})-f\br{\mb{x}-t_{1}\mb{h}_{1}}}\varphi_{1}(t_{1})\,d\lambda_{1}+\frac{\varepsilon}{2}
	\leq \lip(f)\int_{P_1} \abs{t_{1}}\varphi_{1}(t_{1})\,d\lambda_{1}+\frac{\varepsilon}{2}\leq \varepsilon.
	\end{multline*}
	We are now only left to check that $g\in C^1(\inside{Q}{\varepsilon},Y)$ and $g$ is uniformly \Gat differentiable on $\inside{Q}{\varepsilon}$. Note that by Lemma~\ref{lemma:uni_gat_implies_smooth}, the latter implies the former. We first show that $g$ is \Gat differentiable at every $\mb{x}\in\inside{Q}{\varepsilon}$, and then verify that condition~\eqref{eq:unif-gat} of Definition~\ref{def:unif-gat} is satisfied.
	
	Let us start by using~\eqref{eq:g_n_m} to compute the directional derivatives of $g_{n}$ at $\mb{x}\in\inside{Q}{\varepsilon}$ in the direction of the vectors $\mb{h}_{i}$ for $i,n\in \N$ with $n\geq i$ as follows:
	\begin{align}\label{eq:Dhign}
	\begin{split}
	g_n'(\mb{x};\mb{h}_i)
	&=\lim_{\tau\to 0}\frac{g_{n}(\mb{x}+\tau\mb{h}_{i})-g_{n}(\mb{x})}{\tau}\\
	&=\lim_{\tau\to 0}\frac{1}{\tau}\int_{\R^n}\Bbr{f\bbr{(\mb{x}-\sum_{j=1}^{n}t_{j}\mb{h}_{j})+\tau \mb{h}_{i}}-f\bbr{\mb{x}-\sum_{j=1}^{n}t_{j}\mb{h}_{j}}}\prod_{j=1}^{n}\varphi_{j}(t_{j})\,d\lambda_{n}\\
	&=\lim_{\tau\to 0}\int_{\R^n} f\bbr{\mb{x}-\sum_{j=1}^{n}t_{j}\mb{h}_{j}}\cdot \frac{\varphi_{i}(t_{i}+\tau)-\varphi_{i}(t_{i})}{\tau}\cdot \prod_{1\le j\neq i\le n}\varphi_{j}(t_{j})\,d\lambda_{n}\\
	&=\Jint f\bbr{\mb{x}-\sum_{j=1}^{n}t_{j}\mb{h}_{j}}\varphi_{i}'(t_{i})\prod_{1\le j\neq i\le n}\varphi_{j}(t_{j})\,d\lambda_{n},
	\end{split}
	\end{align}
	The penultimate equality is a standard application of the Dominated Convergence Theorem for the Bochner integral, and the last equality follows from $\supp\varphi_i'\subseteq J_i$. It is also important to observe that for a fixed pair of $\mb{x}\in\inside{Q}{\varepsilon}$ and $i\ge1$ the limits in~\eqref{eq:Dhign} are uniform with respect to $n$: this may be verified by applying the Mean Value Theorem to $\varphi_{i}$ and recalling that $\varphi_{i}'\in C^{\infty}(\R)$ with bounded support and is therefore Lipschitz. For any $\theta>0$ and $\mb{x},i$ as above we let $\tau_0=\tau_0(\theta,\mb{x},i)>0$  be such that $B_{X}(\mb{x},\tau_{0})\subseteq \inside{Q}{\varepsilon}$ and 
	\begin{equation}\label{eq:tau0}
	\lnorm{Y}{
		g_n'(\mb{x};\mb{h}_i)-\frac{{g_{n}(\mb{x}+\tau\mb{h}_{i})-g_{n}(\mb{x})}}{\tau}
	}\leq\theta
	\qquad \text{whenever $n\geq i$ and $0<\abs{\tau}<\tau_{0}$.}
	\end{equation}	
	
	We will now argue that for each $i\in\N$ and $\mb{x}\in \inside{Q}{\varepsilon}$, the sequence of directional derivatives $g_n'(\mb{x};\mb{h}_i)$, $n\geq i$, computed above, converges to the directional derivative $g'(\mb{x};\mb{h}_i)$. To this end, let $i\in\N$ and $\mb{x}\in \inside{Q}{\varepsilon}$ and observe, using~\eqref{eq:g_n_m}, for $m>n\geq i$
	\begin{align*}
	&\lnorm{Y}{g_m'(\mb{x};\mb{h}_i)-g_n'(\mb{x};\mb{h}_i)}\\
	&\leq \int_{P_{m}}\lnorm{Y}{f\br{\mb{x}-\sum_{j=1}^{m}t_{j}\mb{h}_{j}}-f\br{\mb{x}-\sum_{j=1}^{n}t_{j}\mb{h}_{j}}}\abs{\varphi_{i}'(t_{i})}\prod_{1\le j\neq i\le m}\varphi_{j}(t_{j})\,d\lambda_{m}\\
	&\leq \lip(f)\int_{P_{m}}\sum_{j=n+1}^{m}\abs{t_{j}}\abs{\varphi_{i}'(t_{i})}\prod_{1\le j\neq i\le m}\varphi_{j}(t_{j})\,d\lambda_{m}\\
	&\leq \lip(f)\sum_{j=n+1}^{m}\frac{\varepsilon}{2\br{\lip(f)+1}2^{j}}\int_{J_{i}}\abs{\varphi_{i}'(t_{i})}\,d\lambda_{1}(t_i)\leq \frac{\lip(\varphi_{i})\varepsilon}{2^{n}}.
	\end{align*}
	Hence, $(g_n'(\mb{x};\mb{h}_i))_{n\in\N}$ is a Cauchy sequence in $(Y,\lnorm{Y}{-})$; let $D_i(\mb{x})\coloneq\lim_{n\to\infty}g_n'(\mb{x};\mb{h}_i)$. Fix an arbitrary $\eta>0$, let $0<\abs{\tau}<\tau_0(\eta/3,\mb{x},i)$ and choose $N\ge i$ large enough such that 
	\begin{equation}\label{eq:N}
	\lnorm{Y}{g_{N}(\mb{y})-g(\mb{y})}
	\leq \frac{\eta\tau}{6}
	\text{ for all } \mb{y}\in\inside{Q}{\varepsilon}
	\qquad\text{ and }\qquad\lnorm{Y}{g_N'(\mb{x};\mb{h}_i)-D_i(\mb{x})}
	\le \frac{\eta}{3}
	.
	\end{equation}
	Then we may combine~\eqref{eq:tau0} and~\eqref{eq:N} to deduce
	\begin{align}\label{eq:lim_gn_g}
	\begin{split}	
	&\lnorm{Y}{\tfrac{g(\mb{x}+\tau \mb{h}_{i})-g(\mb{x})}{\tau}-D_i(\mb{x})}
	\leq \lnorm{Y}{\tfrac{g(\mb{x}+\tau \mb{h}_{i})-g(\mb{x})}{\tau}-\tfrac{g_{N}(\mb{x}+\tau \mb{h}_{i})-g_{N}(\mb{x})}{\tau}}\\
	&+\lnorm{Y}{\tfrac{g_{N}(\mb{x}+\tau \mb{h}_{i})-g_{N}(\mb{x})}{\tau}-g_N'(\mb{x};\mb{h}_i)}
	+\lnorm{Y}{g_N'(\mb{x};\mb{h}_i)-D_i(\mb{x})}\\
	&\leq \frac{\eta}{3}+\frac{\eta}{3}+\frac{\eta}{3}=\eta.
	\end{split}
	\end{align}
	This establishes that $g'(\mb{x};\mb{h}_i)$ exists and equals $D_i(\mb{x})=\lim_{n\to\infty} g_n'(\mb{x};\mb{h}_i)$ for all $\mb{x}\in\inside{Q}{\varepsilon}$ and $i\ge1$, and so inequality~\eqref{eq:tau0} also holds with $g_n$ replaced by $g$, for all $0<\abs{\tau}<\tau_0(\theta,\mb{x},i)$.
	
	Let an arbitrary $\mb{x}\in\inside{Q}{\varepsilon}$ be fixed. Since $g\in\lip(\inside{Q}{\varepsilon})$, the mapping $g'(\mb{x};\cdot)$ is Lipschitz on $\{\mb{h}_i\colon i\ge1\}$, as for any $\mb{h}_i\ne\mb{h}_j$ and any $\eta>0$
	\begin{equation*}
	\lnorm{Y}{g'(\mb{x};\mb{h}_{j})-g'(\mb{x};\mb{h}_{i})}
	\leq 
	2\eta\lnorm{X}{\mb{h}_{j}-\mb{h}_{i}}
	+
	\lip(g)\lnorm{X}{\mb{h}_{j}-\mb{h}_{i}}
	\end{equation*}
	using~\eqref{eq:tau0} with $0<\abs{\tau}<\tau_0(\eta\lnorm{X}{\mb{h}_{j}-\mb{h}_{i}},\mb{x},i)$ and $g$ instead of $g_n$. 
	
	Let $\Phi_{\mb{x}}\colon \Sph_{X}\to Y$ denote the unique Lipschitz extension to $\Sph_{X}$ of the mapping $g'(\mb{x};\cdot)\colon\set{\mb{h}_i\colon i\ge1}\to Y$.  We now verify that the directional derivative $g'(\mb{x};\mb{h})$ exists and equals $\Phi_{\mb{x}}(\mb{h})$ for all $\mb{h}\in \Sph_{X}$ and $\mb{x}\in \inside{Q}{\varepsilon}$. Indeed, given $\mb{h}\in \Sph_{X}$ and $\eta>0$, choose $i\in\N$ such that $\lnorm{X}{\mb{h}_{i}-\mb{h}}\leq \eta/3\br{\lip(\Phi_{\mb{x}})+\lip(g)+1}$. Then 
	\begin{multline}\label{eq:Dhgx}
	\lnorm{Y}{\frac{g(\mb{x}+\tau \mb{h})-g(\mb{x})}{\tau}-\Phi_{\mb{x}}(\mb{h})}\leq \lip(g)\lnorm{X}{\mb{h}-\mb{h}_{i}}+\frac{\eta}{3}
	+\lip(\Phi_{\mb{x}})\lnorm{X}{\mb{h}_{i}-\mb{h}}\leq \eta,
	\end{multline}
	whenever $0<\abs{\tau}<\tau_0(\eta/3,\mb{x},i)$, using~\eqref{eq:tau0} for $g$ instead of $g_n$ and $\Phi_{\mb{x}}(\mb{h}_i)=g'(\mb{x};\mb{h}_i)$. Extending $\Phi_{\mb{x}}$ now to the whole of $X$ via the formula $\Phi_{\mb{x}}(t\mb{h})=t\Phi_{\mb{x}}(\mb{h})$, $t\in\R$, $\mb{h}\in \Sph_{X}$, it is readily verified that $\Phi_{\mb{x}}$ remains Lipschitz and we get
	\begin{equation}\label{eq:Phix}
	g'(\mb{x};\mb{v})=\Phi_{\mb{x}}(\mb{v}) \qquad \text{for all }\mb{x},\mb{v}\in X.
	\end{equation} 
	We finally verify that $\Phi_\mb{x}$ is a linear operator for each $\mb{x}\in\inside{Q}{\varepsilon}$. Together with~\eqref{eq:Phix} and Lipschitzness of $g$ this will establish that $g$ is \Gat differentiable on $\inside{Q}{\varepsilon}$, with \Gat derivative $Dg(\mb{x})=\Phi_\mb{x}$ of norm $\opnorm{\Phi_\mb{x}}\le\lip(g)$ at every $\mb{x}\in\inside{Q}{\varepsilon}$. To show that $\Phi_\mb{x}$ is a linear operator, it is enough to check linearity of $\Phi_{\mb{x}}$ on $\{\mb{h}_i\colon i\ge1\}$. For this, note that a
	calculation similar to~\eqref{eq:Dhign} shows, for $i,j\in\N$ and $\alpha_{i},\alpha_{j}\in\R$, that for $n\ge i,j$ the directional derivative $g_{n}'(\mb{x};{\alpha_{i}\mb{h}_{i}+\alpha_{j}\mb{h}_{j}})$ exists and equals
	\begin{equation*}	
	\Jint f\bbr{\mb{x}-\sum_{k=1}^{n}t_{j}\mb{h}_{j}}\Bbr{\alpha_{i}\varphi_{i}'(t_{i})\prod_{k\neq i}\varphi_{k}(t_{k})+\alpha_{j}\varphi_{j}'(t_{j})\prod_{k\neq j}\varphi_{k}(t_{k})}\,d\lambda_{n},
	\end{equation*}
	and so 
	$g_{n}'(\mb{x};{\alpha_{i}\mb{h}_{i}+\alpha_{j}\mb{h}_{j}})=
	\alpha_{i}g_{n}'(\mb{x};{\mb{h}_{i}})+\alpha_{j}g_{n}'(\mb{x};{\mb{h}_{j}})$.
	Taking limits in this identity  as $n\to\infty$ gives, similarly to~\eqref{eq:lim_gn_g},
	\begin{equation*}
	\Phi_{\mb{x}}(\alpha_{i}\mb{h}_{i}+\alpha_{j}\mb{h}_{j})=\alpha_{i}\Phi_{\mb{x}}(\mb{h}_{i})+\alpha_{j}\Phi_{\mb{x}}(\mb{h}_{j}).
	\end{equation*}

	The only thing left is to check that $g$ is uniformly \Gat differentiable on $\inside{Q}{\varepsilon}$. 
	Assume that $\mb{h}\in\Sph_{X}$ and $\varepsilon>0$ are fixed. Choose $i\ge1$ such that $\lnorm{X}{\mb{h}_i-\mb{h}}<\varepsilon/(4\lip(g)+1)$.

	Then, for each $n\geq i$ and all $\mb{y}_1,\mb{y}_2\in \inside{Q}{\varepsilon}$ we can use~\eqref{eq:Dhign} to derive
	\begin{equation*}
	\lnorm{Y}{g_n'(\mb{y}_1;\mb{h}_{i})-g_n'(\mb{y}_2;\mb{h}_{i})}\leq \lip(f)\int_{J_{i}} \abs{\varphi_{i}'(t)}\,d\lambda_{1}(t)\cdot \lnorm{X}{\mb{y}_1-\mb{y}_2}
	=:L_i\lnorm{X}{\mb{y}_1-\mb{y}_2}
	.
	\end{equation*}
	Taking a limit of the above inequality as $n\to\infty$ we obtain, for any $\mb{y}_1,\mb{y}_2\in\inside{Q}{\varepsilon}$,  
	\begin{equation*}
	\lnorm{Y}{g'(\mb{y}_1;\mb{h}_{i})-g'(\mb{y}_2;\mb{h}_{i})}
	\le 
	L_i\lnorm{X}{\mb{y}_1-\mb{y}_2}
	,
	\end{equation*}
	
	which implies for every $\mb{x}\in\inside{Q}{\varepsilon}$ and any $0<\abs{\tau}<\dist(\mb{x},X\setminus \inside{Q}{\varepsilon})$, 
	\begin{multline*}
	\lnorm{Y}{\frac{g(\mb{x}+\tau \mb{h}_{i})-g(\mb{x})}{\tau}-g'(\mb{x};{\mb{h}_{i}})}
	=\frac{1}{\abs{\tau}}\lnorm{Y}{\int_{(0,\tau)} g'(\mb{x}+t\mb{h}_{i};{\mb{h}_{i}})-g'(\mb{x};{\mb{h}_{i}})\,d\lambda_{1}(t)}\\
	\leq \frac{L_{i}}{\abs{\tau}}\int_{(0,\tau)}\abs{t}\,d\lambda_{1}(t)= L_{i}\abs{\tau}/2.
	\end{multline*}
	Let $\delta=\varepsilon/(L_i+1)$; consider any $\mb{x}\in\inside{Q}{\varepsilon}$ and $0<\abs{\tau}<\min(\delta,\dist(\mb{x},X\setminus\inside{Q}{\varepsilon}))$. We now verify that condition~\eqref{eq:unif-gat} of Definition~\ref{def:unif-gat} is satisfied. Indeed, we readily have $\lnorm{Y}{\frac{g(\mb{x}+\tau \mb{h}_{i})-g(\mb{x})}{\tau}-g'(\mb{x};{\mb{h}_{i}})}\le\varepsilon/2$ from above, 
	\begin{equation*}
	\lnorm{Y}{g'(\mb{x};\mb{h})-g'(\mb{x};\mb{h}_i)}
	=\lnorm{Y}{\Phi_{\mb{x}}(\mb{h}-\mb{h}_i)}
	\le\opnorm{\Phi_{\mb{x}}}\lnorm{X}{\mb{h}-\mb{h}_i}
	\le\lip(g)\lnorm{X}{\mb{h}-\mb{h}_i}
	<\varepsilon/4
	\end{equation*}
	and
	\begin{equation*}
	\lnorm{Y}{
		\frac{g(\mb{x}+\tau \mb{h})-g(\mb{x})}{\tau}-
		\frac{g(\mb{x}+\tau \mb{h}_{i})-g(\mb{x})}{\tau}
	}
	\le
	\lip(g)\lnorm{X}{\mb{h}-\mb{h}_i}
	<\varepsilon/4.
	\end{equation*}
	
\end{proof}
\begin{lemma}\label{lemma:smooth_PlayerI's_mapping_around_E}
	Let $X$ and $Y$ be Banach spaces, where $X$ is separable. Let $Q\subseteq X$, $\emptyset\neq E\subseteq X$ and $r>\rho>0$ be such that $B_{X}(E,r)\subseteq Q$ and $H=B_{X}(E,\rho)$ admits a locally finite $C^{1}$-smooth partition of unity with supports in $H$. Let $f\in\lip(Q,Y)$ and $\varepsilon>0$. Then there exists a mapping $g\in \lip(Q,Y)\cap C^{1}(H,Y)$ such that $g|_{Q\setminus H}=f|_{Q\setminus H}$,  $\lnorm{Y}{g(\mb{y})-f(\mb{y})}\leq \varepsilon$ for all $\mb{y}\in Q$ and $\lip(g)\leq \lip(f)+\varepsilon$.
\end{lemma}
\begin{proof}
	Since $Y$ is Banach, we may assume without loss of generality that $Q$ is closed. We may assume that $\varepsilon<(r-\rho)/2$ so that $E\subseteq H\subseteq Q_{2\varepsilon}$, where $Q_{2\varepsilon}$ is defined by~\eqref{eq:qeps} in Theorem~\ref{thm:johanis}. Let $(\varphi_{k})_{k\in\N}$ be a smooth, locally finite partition of unity on $H$ with $\supp(\varphi_k)\subseteq H$ for each $k\ge1$. Choose any $\varepsilon_k\in(0,\varepsilon)$ such that $\sum_{k\ge1}(1+\lip(\varphi_k))\varepsilon_k\le\varepsilon$.
	
	By Theorem~\ref{thm:johanis} we have that $\SLA(f,H,Y,\varepsilon_k)\supseteq \SLA(f,Q_{\varepsilon_k},Y,\varepsilon_k)\ne\emptyset$ for each $k\ge1$ and, moreover, for each $k\ge1$ the set $\SLA(f,Q_{\varepsilon_k},Y,\varepsilon_k)$ contains a mapping $h_{k}$ with $\lip(h_{k})\leq\lip(f)$. To complete the proof, we let $\theta=\varepsilon$, $\theta_{k}=\varepsilon_{k}$, $U=H$ and $A_k=H$ for all $k\ge1$, and finally take $g$ as the mapping $\tilde{h}$ given by the conclusion of Lemma~\ref{lemma:SLA}.
\end{proof}
\begin{appendices}
\numberwithin{equation}{section} 
\renewcommand{\theequation}{\Alph{section}.\arabic{equation}}

	\appendixpage
\section{Local to global Lipschitz estimates.}
\begin{lemma}\label{lem:lip-loc_glob} 
	Let $X,Y$ be normed spaces, $F\subseteq U\subseteq X$ where $U$ is open and convex, and suppose that for any $\varepsilon>0$ and any $\mb{x},\mb{y}\in U$  there exist $\mb{x}',\mb{y}'\in U$ such that $\lnorm{X}{\mb{x}-\mb{x}'},\lnorm{X}{\mb{y}-\mb{y}'}<\varepsilon$ and $[\mb{x}',\mb{y}']\cap F$ has $1$-dimensional Hausdorff measure $0$.
	Let
	$g\colon U\to Y$ be locally Lipschitz on $U$ and suppose that $g$ has at least one of the following properties:
		\begin{enumerate}[(i)]
			\item $g$ is locally $L$-Lipschitz on $U\setminus F$.
			\item \label{it:lip-loc_glob:der} 
for every $\mb{x}\in U\setminus F$, the derivative $Dg(\mb{x})$ exists and satisfies 
			$\opnorm{Dg(\mb{x})}\leq L$.
		\end{enumerate}
		Then $g\colon U\to Y$ is $L$-Lipschitz.
\end{lemma}
\begin{proof}
		In order to show that $g$ is $L$-Lipschitz, we fix an arbitrary $\mb{w}^*\in \Sph_{Y^*}$ and show that the function $g_{\mb{w}^*}=\mb{w}^*\circ g\colon U\to\R$ is $L$-Lipschitz.
		Let $\mb{x},\mb{y}\in U$ be any pair of distinct points. Let $\varepsilon\in(0,\lnorm{X}{\mb{y}-\mb{x}}/2)$ be arbitrary; find $\mb{x}',\mb{y}'\in U$ as guaranteed by the hypothesis of the lemma. Then, since $[\mb{x}',\mb{y}']\subseteq U$ and $g_{\mb{w}^*}|_{[\mb{x}',\mb{y}']}$ is locally Lipschitz as a mapping $[\mb{x}',\mb{y}']\to Y$,
		\begin{equation*}
		g_{\mb{w}^*}(\mb{y}')-g_{\mb{w}^*}(\mb{x}')
		=\int_0^{\lnorm{X}{\mb{y}'-\mb{x}'}} g_{\mb{w}^*}'(\mb{x}'+t\mb{v};\mb{v})\,dt,
		\end{equation*}
		where $\mb{v}=\frac{\mb{y}'-\mb{x}'}{\lnorm{X}{\mb{y}'-\mb{x}'}}$, $g_{\mb{w}^*}'(\mb{x}'+t\mb{v};\mb{v})$ is the directional derivative of $g_{\mb{w}^*}$ at $\mb{x}'+t\mb{v}$ in the direction of $\mb{v}$ which exists for Lebesgue almost all $t\in[0,\lnorm{X}{\mb{y}'-\mb{y}}]$.
		Recall that $\mc{H}^{1}$-almost every point of $[\mb{x}',\mb{y}']$ belongs to $U\setminus F$, hence $|g_{\mb{w}^*}'(\mb{x}'+t\mb{v};\mb{v})|\le L$ for almost all $t\in[0,\lnorm{X}{\mb{y}'-\mb{x}'}]$, implying $\abs{		g_{\mb{w}^*}(\mb{y}')-g_{\mb{w}^*}(\mb{x}')
		}\le L\lnorm{X}{\mb{y}'-\mb{x}'}$. 
		Passing to a limit when $\varepsilon\to0$ gives  $\abs{g_{\mb{w}^*}(\mb{y})-g_{\mb{w}^*}(\mb{x})
		}\le L\lnorm{X}{\mb{y}-\mb{x}}$ which, in turn, due to arbitrariness of $\mb{x},\mb{y}\in U$ and $\mb{w}^*\in\Sph_{Y^*}$ implies the statement.
\end{proof}
\begin{cor}\label{cor:lip-loc_glob}
		Let $X$, $Y$ be normed spaces, where $X$ is finite-dimensional, let $U\subseteq X$ be open and convex, $g\colon U\to Y$ be locally Lipschitz on $U$ and suppose that 
$\opnorm{Dg(\mb{x})}\leq L$ for Lebesgue a.e. $\mb{x}\in U$.
		Then $g\colon U\to Y$ is $L$-Lipschitz.
\end{cor}
	\begin{proof}
		Defining $F$ as a Borel Lebesgue null set containing the Lebesgue null set $U\setminus S$, where $S$ is the set of $\mb{x}\in U$ for which 
		$\opnorm{Dg(\mb{x})}\leq L$.
		A standard application of Fubini's Theorem shows that the conditions of Lemma~\ref{lem:lip-loc_glob} with~\eqref{it:lip-loc_glob:der} are met.
	\end{proof}
	
	\begin{lemma}\label{lemma:awkward_lipschitz_bound}
		Let $X$ and $Y$ be normed spaces, $U\subseteq Q\subseteq X$ be sets where $U$ is open and $Q$ is closed, let $f\colon Q\to Y$ be a continuous function, which is locally $L$-Lipschitz on $U$ and is Lipschitz on $Q\setminus U$. Then $f\in\lip(Q,Y)$ and
		\begin{equation*}
		\lip(f)\leq \max\set{L,\lip(f|_{Q\setminus U})}.
		\end{equation*} 	
	\end{lemma}
	\begin{proof}
		Fix distinct points $\mb{x}_{1},\mb{x}_{2}\in Q$ and set $L_1:=\max\set{L,\lip(f|_{Q\setminus U})}$. We show that 
		\begin{equation}\label{eq:to_show}
		\lnorm{Y}{f(\mb{x}_{2})-f(\mb{x}_{1})}\leq L_1\lnorm{X}{\mb{x}_{2}-\mb{x}_{1}}.
		\end{equation}
		This inequality is clear if both $\mb{x}_1,\mb{x}_2\in Q\setminus U$. Assume without loss of generality $\mb{x}_1\in U$.
		Let $\mb{e}:=\mb{x}_{2}-\mb{x}_{1}$, $U_1=U\cap \br{\mb{x}_1+\R\mb{e}}$ and $Q_1=Q\cap \br{\mb{x}_1+\R\mb{e}}$. Then $\mb{x}_1\in U_1\subseteq Q_1\subseteq \br{\mb{x}_1+\R\mb{e}}$ and $U_1$ is a relatively open subset of the line $\br{\mb{x}_1+\R\mb{e}}$, hence can be written as a disjoint union of open intervals. Let $I$ be the open interval containing $\mb{x}_1$. If $\mb{x}_2\in I\subseteq U_1$, then~\eqref{eq:to_show} is trivially satisfied, even with $L$ instead of $L_1$ in the right-hand side. Hence assume $I$ has a right endpoint $\mb{b}\in \mb{x}_{1}+\R\mb{e}$ lying between $\mb{x}_{1}$ and $\mb{x}_{2}$, implying $\lnorm{X}{\mb{x}_{2}-\mb{b}}+\lnorm{X}{\mb{b}-\mb{x}_{1}}=\lnorm{X}{\mb{x}_{2}-\mb{x}_{1}}$ and $\mb{b}\in \cl{U_1}\setminus U_{1}\subseteq Q_1\setminus U_{1}$. If $\mb{x}_2\notin U_1$, then~\eqref{eq:to_show} follows from 
		\begin{equation*}
		\lnorm{Y}{f(\mb{x}_{2})-f(\mb{b})}
		+	
		\lnorm{Y}{f(\mb{b})-f(\mb{x}_{1})}
		\leq
		\lip(f|_{Q\setminus U})\lnorm{X}{\mb{x}_2-\mb{b}}+
		L\lnorm{X}{\mb{b}-\mb{x}_1},
		\end{equation*}
		establishing the $L_1$-Lipschitzness of $f$ between points from $U_1$ and $Q_1\setminus U_1$.
		Therefore if
$\mb{x}_2\in U_1\setminus I$, 
then~\eqref{eq:to_show}  follows from
		\begin{equation*}
		\lnorm{Y}{f(\mb{x}_{2})-f(\mb{b})}
		+	
		\lnorm{Y}{f(\mb{b})-f(\mb{x}_{1})}
		\leq
		L_1\lnorm{X}{\mb{x}_2-\mb{b}}+
		L\lnorm{X}{\mb{b}-\mb{x}_1},
		\end{equation*}
which
holds due to $\mb{x}_1,\mb{x}_2\in U_1$, $\mb{b}\in Q_1\setminus U_1$. 
\end{proof}
\section{Derivatives of Lipschitz mappings.}
\begin{thm}\label{lemma:dsty_pt}
	Let $(X,\lnorm{X}{-})$ and $(Y,\lnorm{Y}{-})$ be normed spaces, where $X$ is finite-dimensional, $\myV\subseteq X$ be open, $f,g\colon \myV\to Y$ be locally Lipschitz mappings and $A=\set{\mb{x}\in \myV\colon f(\mb{x})=g(\mb{x})}$. Then the following statements hold:
	\begin{enumerate}[(i)]
		\item\label{diff_dnsty_pt} At each Lebesgue density point $\mb{x}$ of $A$ the mapping $f$ is Fr\'echet differentiable if and only if the mapping $g$ is Fr\'echet differentiable and $Df(\mb{x})=Dg(\mb{x})$.
		\item\label{dervs_equal_ae} If $Y$ is finite-dimensional then $Df(\mb{x})$ and $Dg(\mb{x})$ exist and are equal for Lebesgue almost every $\mb{x}\in A$.
	\end{enumerate}
\end{thm}
\begin{proof}
	Statement~\eqref{dervs_equal_ae} is a consequence of~\eqref{diff_dnsty_pt}, Stepanov's Theorem and the Lebesgue Density Theorem. Indeed, by~\eqref{diff_dnsty_pt}, we have $Df(\mb{x})=Dg(\mb{x})$ everywhere in the set
	\begin{equation*}
	\set{\mb{x}\in A\colon \text{$\mb{x}$ is a Lebesgue denisty point of $A$}}\cap  \set{\mb{x}\in A\colon \text{$f$ is Fr\'echet differentiable at $\mb{x}$}}.
	\end{equation*}
	The first set in this intersection has full Lebesgue measure in $A$ by the Lebesgue Density Theorem and, for finite-dimensional $Y$, the second set in the intersection is also of full Lebesgue measure in $A$, by Stepanov's Theorem. 
	
	We now prove~\eqref{diff_dnsty_pt}: Let $\mb{x}$ be a Lebesgue density point of $A$, choose $r>0$ such that $B:=B_{X}(\mb{x},r)\subseteq \myV$ and $f|_{B}$ and $g|_{B}$ are Lipschitz and assume that $f$ is differentiable at $\mb{x}$. Fix $\varepsilon\in (0,1)$. Let $d:=\dim(X)$, set 
	\begin{equation}\label{eq:eta}
	\eta=\br{\frac{\varepsilon}{16\br{\lip(f|_{B})+\lip(g|_{B})+1}}}^{d}
	\end{equation}
	and choose $\delta\in (0,r)$ so that 
	\begin{equation}\label{eq:Df_threshold}
	\lnorm{Y}{f(\mb{x}+\mb{y})-f(\mb{x})-Df(\mb{x})(\mb{y})}\leq \eta\lnorm{X}{\mb{y}}\qquad \text{for all $\mb{y}\in \cl{B}_{X}(\mb{0},\delta)$,}
	\end{equation}
	and
	\begin{equation}\label{eq:leb_density_threshold}
	\leb_{X}\br{B_{X}(\mb{x},t)\cap A}\geq (1-\eta)\leb_{X}\br{B_{X}(\mb{x},t)} \qquad \text{for all $t\in [0, \delta]$.}
	\end{equation}
	Let $\mb{y}\in B_{X}(\mb{0},\delta/2)\setminus\set{\mb{0}}$ and set $t=2\lnorm{X}{\mb{y}}$. Observe, using $\eta<2^{-d-1}$ and $\delta<r$, that
	\begin{multline*}
	B_{X}(\mb{x+\mb{y}},2^{1/d}\eta^{1/d}t)\subseteq B_{X}(\mb{x},t)\subseteq B\subseteq \myV\quad \text{and}\\ \leb_{X}\br{B_{X}(\mb{x+\mb{y}},2^{1/d}\eta^{1/d}t)}=2\eta\leb\br{B_{X}(\mb{x},t)}.
	\end{multline*}
	Therefore, by~\eqref{eq:leb_density_threshold} we have that $B_{X}(\mb{x}+\mb{y},2^{1/d}\eta^{1/d}t)\cap A\cap B\neq \emptyset$. Let $\mb{z}\in B_{X}(\mb{x}+\mb{y},2^{1/d}\eta^{1/d}t)\cap A\cap B$ and set $\mb{y}':=\mb{z}-\mb{x}$. Then we have 
	\begin{align}\label{eq:y'}
	\lnorm{X}{\mb{y}'-\mb{y}}
	&=\lnorm{X}{\mb{z}-(\mb{x}+\mb{y})}\leq 2^{1/d}\eta^{1/d}t=2^{1+\frac{1}{d}}\eta^{1/d}\lnorm{X}{\mb{y}}\leq 4\eta^{1/d}\lnorm{X}{\mb{y}} \quad \text{so that}\nonumber\\ \lnorm{X}{\mb{y'}}&\leq \br{1+4\eta^{1/d}}\lnorm{X}{\mb{y}}\leq2\lnorm{X}{\mb{y}}<\delta.
	\end{align}
	We may now use the hypothesis $f|_{A}=g|_{A}$, \eqref{eq:y'}, \eqref{eq:Df_threshold} and  $\mb{x}+\mb{y}'=\mb{z},\mb{x}\in A\cap B$, $\mb{x}+\mb{y}\in B$ to write
	\begin{multline*}
	\lnorm{Y}{g(\mb{x}+\mb{y})-g(\mb{x})-Df(\mb{x})(\mb{y})}\leq\\ \lnorm{Y}{g(\mb{x}+\mb{y})-g(\mb{x}+\mb{y}')}+\lnorm{Y}{f(\mb{x}+\mb{y}')-f(\mb{x})-Df(\mb{x})(\mb{y}')}+\lnorm{Y}{Df(\mb{x})(\mb{y}'-\mb{y})}\\
	\leq \lip(g|_{B})4\eta^{1/d}\lnorm{X}{\mb{y}}+\eta\br{1+4\eta^{1/d}}\lnorm{X}{\mb{y}}+\lip(f|_{B})4\eta^{1/d}\lnorm{X}{\mb{y}}\leq \varepsilon\lnorm{X}{\mb{y}},
	\end{multline*}
	where we apply~\eqref{eq:eta} to get the final inequality (bound the three coefficients in order by $1/4$, $1/2$ and $1/4$). Since $\varepsilon>0$ and $\mb{y}\in B_{X}(\mb{0},\delta/2)\setminus\set{\mb{0}}$ were arbitrary, this establishes the Fr\'echet differentiability of $g$ at $\mb{x}$ with $Dg(\mb{x})=Df(\mb{x})$. Since the roles of $f$ and $g$ in the above argument are symmetric, this proves the if and only if statement of~\eqref{diff_dnsty_pt}.
\end{proof}

\begin{lemma}\label{lemma:closure_good_directions}
	Let $(X,\lnorm{X}{-})$, $(Y,\lnorm{Y}{-})$ be normed spaces, $H\subseteq X$ be open, $L\in\mc{L}(X,Y)$, $f\colon H\to Y$ be a Lipschitz mapping and $\mb{z}\in H$. 
	
	Then the set $D:=\set{\mb{u}\in \Sph_{X}\colon f'(\mb{z},\mb{u})\text{ exists and equals }L(\mb{u})}$ is closed.
\end{lemma}
\begin{proof}
	Let $(\mb{u}_{j})_{j\in\N}$ be a sequence in $D$ with limit $\mb{u}=\lim_{j\to\infty}\mb{u}_{j}\in X$. We show that $\mb{u}\in D$. Given $\varepsilon>0$ we choose choose $k\in\N$ large enough so that 
	\begin{equation*}
	\lnorm{X}{\mb{u}_{k}-\mb{u}}\leq \frac{\varepsilon}{3\br{\lip(f)+\opnorm{L}+1}}.
	\end{equation*}
	Next choose $\delta>0$ small enough so that 
	\begin{equation*}
	\lnorm{Y}{f(\mb{z}+t\mb{u}_{k})-f(\mb{z})-tL(\mb{u}_{k})}\leq \frac{\varepsilon\abs{t}}{3}\qquad \text{ for all $t\in (0,\delta)$.}
	\end{equation*}
	Then, for all $t\in (0,\delta)$ we have
	\begin{multline*}
	\lnorm{Y}{f(\mb{z}+t\mb{u})-f(\mb{z})-tL(\mb{u})}\leq \lnorm{Y}{f(\mb{z}+t\mb{u})-f(\mb{z}+t\mb{u}_{k})}\\
	+\lnorm{Y}{f(\mb{z}+t\mb{u}_{k})-f(\mb{z})-tL(\mb{u}_{k})}+\abs{t}\lnorm{Y}{L(\mb{u}_{k})-L(\mb{u})}\\
	\leq \frac{\varepsilon\abs{t}}{3}+\frac{\varepsilon\abs{t}}{3}+\frac{\varepsilon\abs{t}}{3}=\varepsilon\abs{t}.
	\end{multline*}
	We conclude that $\mb{u}\in D$.
\end{proof}

\begin{thm}\label{lemma:der_equal_cts_ae_implies_evrywh}
	Let $(X,\lnorm{X}{-})$, $(Y,\lnorm{Y}{-})$ be normed spaces, where $X$ is finite-dimensional, $H\subseteq X$ be open, $\Phi\colon H\to \mc{L}(X,Y)$ be continuous and $f\colon H\to Y$ be a Lipschitz mapping such that $Df(\mb{x})=\Phi(\mb{x})$ for Lebesgue almost every $\mb{x}\in H$. Then $f\in C^{1}(H,Y)$ and $Df(\mb{x})=\Phi(\mb{x})$ for every $\mb{x}\in H$.
\end{thm}
\begin{proof} The proof goes by induction on $d=\dim X$. For $d=1$ it suffices to observe that
	\begin{equation*}
	f(y+t)=f(t)+\int_{0}^{t}\Phi(y+s)(1)\,ds
	\end{equation*}
	for all $y\in H$ and $t\in (0,\dist(y,X\setminus H))$.
	
	Assume now that $d\geq 2$ and the theorem is valid whenever the domain space has dimension less than $d$. Suppose $\dim X=d$, fix $\mb{x}\in H$ and $\mb{v}\in \Sph_{X}$. We complete the proof by showing that the directional derivative $f'(\mb{x},\mb{v})$ exists and equals $\Phi(\mb{x})(\mb{v})$. Given any $(d-1)$-dimensional subspace $U$ of $X$, not containing $\mb{v}$ we have, by Fubini's Theorem, that $\mc{H}^{1}$-a.e. $\mb{z}\in \br{\mb{x}+\R\mb{v}}\cap H$ has the property that for $\mc{H}^{d-1}$-a.e. $\mb{y}\in \br{\mb{z}+U}\cap H$ the mapping $f$ is differentiable at $\mb{y}$ and $Df(\mb{y})=\Phi(\mb{y})$. By the induction hypothesis, we get that $\mc{H}^{1}$-a.e. $\mb{z}\in \br{\mb{x}+\R\mb{v}}\cap H$ is such that the directional derivatives $f'(\mb{y},\mb{u})$ exist and equals $\Phi(\mb{y})(\mb{u})$ for every $\mb{y}\in \br{\mb{z}+U}\cap H$ and every $\mb{u}\in U$, in particular for $\mb{y}=\mb{z}$. We conclude, by applying this argument to each $U$ from a countable dense subset of the set of $(d-1)$-dimensional subspaces of $X$ not containing $\mb{v}$, that $\mc{H}^{1}$-a.e. $\mb{z}\in \br{\mb{x}+\R\mb{v}}\cap H$ has the property that all directional derivatives $f'(\mb{z},\mb{u}_{j})$ for a dense sequence
	$(\mb{u}_{j})_{j\in\N}$ in $\Sph_{X}\setminus\set{\mb{v},-\mb{v}}$ exist and are given by the formula $f'(\mb{z},\mb{u}_{j})
	=\Phi(\mb{z})(\mb{u}_{j})$. By Lemma~\ref{lemma:closure_good_directions} this formula extends to all $\mb{u}\in \cl{\set{\mb{u}_{j}\colon j\in\N}}=\Sph_{X}$. In particular, it extends to $\mb{u}=\mb{v}$, giving $f'(\mb{z},\mb{v})=\Phi(\mb{z})(\mb{v})$ for $\mc{H}^{1}$-a.e. $\mb{z}\in \br{\mb{x}+\R\mb{v}}\cap H$. Finally, by the induction hypothesis, this implies $f'(\mb{z},\mb{v})=\Phi(\mb{z})(\mb{v})$ for all $\mb{z}\in \br{\mb{x}+\R\mb{v}}\cap H$, in particular for $\mb{z}=\mb{x}$.
\end{proof}
\section{Miscellaneous.}\label{sec:misc}
The following lemma verifies that the minimum in the definition of $\cyl(T)$, see Definition~\ref{def:cyl_constant}, is attained.
\begin{lemma}\label{lemma:cyl_constant}
	Let $X$ and $Y$ be normed vector spaces and $T\in\mc{L}(X,Y)\setminus\{\mb{0}_{\mc{L}(X,Y)}\}$ be of finite rank $l$. Then the infimum
	\begin{equation*}
	\cyl(T):=\inf\set{\max_{1\leq j\leq l}\opnorm{\sum_{i=1}^{j}\mb{w}_{i}^{*}\circ T(\cdot)\,\mb{w}_{i}}
		\colon 
		\begin{array}{c}
		W=(\mb{w}_{1},\ldots,\mb{w}_{l})\text{ is a basis of $T(X)$},\\
		\mb{w}_{1}^{*},\ldots,\mb{w}_{l}^{*}\in T(X)^{*}\text{ is dual to }W
		\end{array}}
	\end{equation*}
	is attained, so it is in fact a minimum.
\end{lemma}
\begin{proof}
	Whenever $(\mb{w}_{1},\ldots,\mb{w}_{l})$ and $\mb{w}_{1}^{*},\ldots,\mb{w}_{l}^{*}$ contribute to the set over which the infimum defining $\cyl(T)$ is defined, the operators $\mb{w}_{i}^{*}\circ T(\cdot)\mb{w}_{i}\in \mc{L}(X,Y)$, for $1\leq i\leq l$, are invariant under rescaling of $\mb{w}_{i}$. Therefore, the set in the definition of $\cyl(T)$ is unchanged if we only allow contributions from bases $W$ with all vectors $\mb{w}_{i}$ of norm $1$ for all $1\leq i\leq l$. We will work with this equivalent definition of $\cyl(T)$ in the present proof.
	Let $Z:=T(X)\subseteq Y$ and $\bas\subseteq(\Sph_{Z})^{l}$ be the collection of ordered bases of $Z$, consisting of vectors of norm $1$.
	For each $1\le s\le l$ and
	$W=(\mb{w}_{1},\ldots,\mb{w}_{l})\in\bas$, let 
	\begin{equation*}
	\cyl(T,W,s)=\max_{1\leq j\leq s}\opnorm{\sum_{i=1}^{j}\mb{w}_{i}^{*}\circ T(\cdot)\,\mb{w}_{i}}
	.
	\end{equation*}
	Then $\cyl(T)=\inf_{W\in\bas} \cyl(T,W,l)$.
	Let $W_n\in\bas$ be such that $\cyl(T,W_n,l)\to\cyl(T)$. Let $W=(\mb{u}_1,\dots,\mb{u}_l)$ be the limit, in ${(\Sph_{Z})}^{l}$, of a convergent subsequence of $W_n$.
	Assume  $W\notin\bas$, i.e. the vectors are linearly dependent. Let $k\le l$ be the smallest index such that $\mb{u}_1,\dots,\mb{u}_k$ are linearly dependent, $a_{1},\ldots,a_{k-1}\in \R$ be such that $\mb{u}_k=\sum_{1\le i\le k-1} a_i\mb{u}_i$ and denote $A=\sum_{1\le i\le k-1}|a_i|$. Note that $k\ge2$ as $\lnorm{Y}{\mb{u}_1}=1$.
	For each $\alpha>0$, let $W\upp{\alpha}=(\mb{w}_i\upp\alpha)\in\bas$ be such that
	\begin{equation}\label{eq:est-cyl}
	\cyl(T,W\upp\alpha,l)<\cyl(T)+\alpha 
	\end{equation}
	and
	$\lnorm{Y}{\mb{w}_i\upp\alpha-\mb{u}_i}<\alpha$ for all $1\le i\le l$. Note that each $W^{(\alpha)}$ may be chosen from the sequence $(W_{n})$. Letting $C_{\alpha}=\lnorm{Z^*}{(\mb{w}_k\upp\alpha)^*}$, we get
	\begin{align*}
	1&=(\mb{w}_k\upp\alpha)^*(\mb{w}_k\upp\alpha)
	\le(\mb{w}_k\upp\alpha)^*(\mb{u}_k)+\alpha C_{\alpha}
	=(\mb{w}_k\upp\alpha)^*(\sum_{1\le i\le k-1} a_i\mb{u}_i)+\alpha C_{\alpha}\\
	&\le (\mb{w}_k\upp\alpha)^*(\sum_{1\le i\le k-1} a_i\mb{w}_i\upp\alpha)+\alpha A C_{\alpha}+\alpha C_{\alpha}
	=0+\alpha(A+1)C_{\alpha}=\alpha(A+1)C_{\alpha}.
	\end{align*}
	Hence $\lnorm{Z^*}{(\mb{w}_k\upp\alpha)^*}=C_{\alpha}\to\infty$ as $\alpha\to0$. 
	
	Fix a null sequence $\alpha_n\in(0,1)$, let $s\le k$ be the smallest index such that $\lnorm{Z^*}{(\mb{w}_{s}\upp{\alpha_{n}})^*}$ is unbounded and let $\beta_n\to0$ be a subsequence of $(\alpha_n)$ such that $\lnorm{Z^*}{(\mb{w}_{s}\upp{\beta_n})^*}\to\infty$. Then $\opnorm{{(\mb{w}_{s}\upp{\beta_n})^*}\circ T(\cdot) \mb{u}_s}=\lnorm{X^*}{(\mb{w}_{s}\upp{\beta_n})^*\circ T(\cdot)}\to\infty$, while $\opnorm{{(\mb{w}_{i}\upp{\beta_n})^*}\circ T(\cdot) \mb{u}_i}$ are bounded for each $1\le i\le s-1$, implying $\cyl(T,W\upp{\beta_{n}},s)\to \infty$. Thus, by~\eqref{eq:est-cyl}, 
	\begin{equation*}
	\cyl(T)\ge
	\cyl(T,W\upp{\beta_n},l)-\beta_n 
	\ge \cyl(T,W\upp{\beta_n},s)-1 
	\to\infty
	\end{equation*} 
	a contradiction. 
\end{proof}
The following theorem is a version of an observation in~\cite{alberti2016differentiability}. However, there are several differences in notation and terminology in~\cite{alberti2016differentiability} compared to the present paper and it requires some careful reading in order to obtain Theorem~\ref{thm:compact_pu} from what is written in~\cite{alberti2016differentiability}. Therefore, in this section we explain how to navigate~\cite{alberti2016differentiability} in order to verify Theorem~\ref{thm:compact_pu}.
\begin{thm}\label{thm:compact_pu}
	Let $X$ be a finite-dimensional normed space, $E\subseteq X$ be a compact purely unrectifiable set and $T\in X^{*}\setminus\set{0}$. Then for every $\varepsilon>0$ there exists an open set $G\subseteq X$ such that $E\subseteq G$ and 
	\begin{multline*}
	\sup\left\{\mc{H}^{1}(G\cap \gamma(\R))\colon \gamma\in\lip(\R,X),\right.\\
	\left. T(\gamma'(t))\geq \varepsilon\lnorm{X}{\gamma'(t)}\lnorm{X^{*}}{T}\text{for Lebesgue a.e.\ $t\in\R$}\right\}\leq \varepsilon.
	\end{multline*}
\end{thm}
\begin{proof}

	The statement is obtained by applying~\cite[Step~1 (inside the proof of Lemma~4.12)]{alberti2016differentiability} to $K=E$. We let $n:=\dim X$, identify $X$ and $X^*$ with $\R^{n}$ and write $\mb{e}_{1},\ldots,\mb{e}_{n}$ for the standard basis vectors of $X$.  Here, we identify $L\in X^{*}$ with the vector $L\in\R^{n}$ satisfying $L\mb{x}=\skp{L,\mb{x}}$ for all $\mb{x}\in X$. Let $\lnorm{E}{-}$ denote the Euclidean norm on $X\leftrightarrow \R^{n}\leftrightarrow X^{*}$. Then, by equivalence of norms on finite-dimensional spaces, there is a constant $M>0$ such that
	\begin{align*}
	\frac{1}{M}\lnorm{E}{\mb{x}}&\leq \lnorm{X}{\mb{x}}\leq M\lnorm{E}{\mb{x}},\qquad \text{for all $\mb{x}\in X$, and}\\
	\frac{1}{M}\lnorm{E}{L}&\leq \lnorm{X^{*}}{L}\leq M\lnorm{E}{L} \qquad \text{for all $L\in X^{*}$.}
	\end{align*}
	We may assume that $\varepsilon<1/M^{2}$. In the notation of~\cite{alberti2016differentiability} we take $\alpha=\cos^{-1}\br{M^{2}\varepsilon}$ and $e=\frac{1}{\lnorm{E}{T}}(T\mb{e}_{1},T\mb{e}_{2},\ldots,T\mb{e}_{n})$. We also note that the notion of $C$-null for $C=C(e,\alpha)$, in~\cite[4.11, Lemma~4.12]{alberti2016differentiability}, is weaker than pure unrectifiability. Applying~\cite[Step 1, Proof of L.~4.12]{alberti2016differentiability} we obtain an open set $G\subseteq X$ such that $E\subseteq G$ and $\mc{H}^{1}(G\cap \gamma(J))\leq \varepsilon$ for every compact interval $J\subseteq \R$ and $\gamma\in\lip(J,X)$ with $T(\gamma'(t))\geq M^{2}\varepsilon \lnorm{E}{\gamma'(t)}\lnorm{E}{T}\geq \varepsilon \lnorm{X}{\gamma'(t)}\lnorm{X^*}{T}$. This implies the conclusion of the theorem.
\end{proof}
\begin{lemma}\label{lemma:C1_compact_uniform}
	Let $X$ and $Y$ be normed spaces, $E\subseteq U\subseteq X$ be sets where $E$ is compact and $U$ is open, and let $g\in C^{1}(U,Y)$ and $\theta>0$. Then there exists $\delta\in (0,\theta)$ such that for every $\mb{x}\in E$ and every $\mb{y}\in X$ with $\lnorm{X}{\mb{y}}\leq \delta$ we have
	\begin{equation*}
		\lnorm{Y}{g(\mb{x}+\mb{y})-g(\mb{x})-Dg(\mb{x})(\mb{y})}\leq \theta \lnorm{X}{y}.
	\end{equation*}
\end{lemma}
\begin{proof}
	For each $\mb{x}\in E$ choose $\delta_{\mb{x}}>0$ small enough so that
	\begin{equation*}
		\opnorm{Dg(\mb{z})-Dg(\mb{x})}\leq \theta \qquad \text{for all }\mb{z}\in B_{X}(\mb{x},2\delta_{\mb{x}})\subseteq U.
	\end{equation*}
	The collection of sets $(B_{X}(\mb{x},\delta_{\mb{x}}))_{\mb{x}\in E}$ is an open cover of the compact set $E$; it therefore admits a finite subcover $B_{X}(\mb{x}_{1},\delta_{1}),B_{X}(\mb{x}_{2},\delta_{2}),\ldots,B_{X}(\mb{x}_{N},\delta_{N})$ for some $N\in\N$, where for $j=1,\ldots,N$ we relabel $\delta_{\mb{x}_{j}}$ as $\delta_{j}$. 
	
	Let $\delta:=\min\set{\delta_{1},\ldots,\delta_{N}}>0$, $\mb{x}\in E$ and $\mb{y}\in X$ with $0<\lnorm{X}{\mb{y}}\leq \delta$. Then there exists $i\in\set{1,\ldots,N}$ such that $\mb{x}\in B_{X}(\mb{x}_{i},\delta_{i})$ and so $[\mb{x},\mb{x}+\mb{y}]\subseteq B_{X}(\mb{x}_{i},2\delta_{i})\subseteq U$. Set $\mb{e}:=\frac{\mb{y}}{\lnorm{X}{\mb{y}}}$ and let $\varphi\in Y^{*}$ be a functional with $\lnorm{Y^*}{\varphi}=1$. Then
	\begin{multline*}
		\abs{\varphi\Bigl(g(\mb{x}+\mb{y})-g(\mb{x})-Dg(\mb{x})(\mb{y})\Bigr)}\\=\abs{\int_{0}^{\lnorm{X}{\mb{y}}}D(\varphi\circ g)(\mb{x}+t\mb{e})(\mb{e})-D(\varphi\circ g)(\mb{x})(\mb{e})\,dt}\\
		\leq \int_{0}^{\lnorm{X}{\mb{y}}}\opnorm{Dg(\mb{x}+t\mb{e})-Dg(\mb{x})}\,dt\leq \theta \lnorm{X}{\mb{y}}.
	\end{multline*}
	Taking supremums in the above inequality over all $\varphi\in Y^{*}$ completes the proof.
\end{proof}

\begin{remark}\label{rem:lip-loc_glob}
		If $X$ is a finite-dimensional normed space, $E\subseteq U\subseteq X$ are such that $E$ is compact and $U$ is open, then then there exists an open $U_0$ such that $E\subseteq U_0\subseteq\cl{U_0}\subseteq U$ and the 	Lebesgue measure of $\partial U_0$ is zero. The latter could be obtained by choosing $U_0$ in the form of a finite union of open balls with centres in $E$.
\end{remark}

\begin{lemma}\label{lemma:commpact_partition_unity}	
	Let $X$ be a normed space, $E\subseteq V\subseteq X$ be sets where $E$ is compact and $V$ is open, $(G_{k})_{k\in\N}$ be a sequence of open subsets of $X$ which contain $E$ and let $(\varphi_{k})_{k\in\N}$ be a smooth, locally finite partition of unity with supports contained in $V$. Then there is a number $K\in\N$ and an open set $U\subseteq X$ such that
	\begin{equation*}
		E\subseteq U\subseteq \cl{U}\subseteq V\cap \bigcap_{k=1}^{K}G_{k} 
	\end{equation*}	
	and
	\begin{equation*}
		\supp(\varphi_k)\cap U=\emptyset
		\quad\text{for all}	\quad
		k>K.
	\end{equation*}
	In the case that $X$ is finite-dimensional, $U$ may be chosen so that additionally $\partial U$ has Lebesgue measure zero.
\end{lemma}
\begin{proof}
	For each point $\mb{p}\in E$ we may choose a radius $r_{\mb{p}}>0$ such that $\cl{B}_{X}(\mb{p},r_{\mb{p}})\subseteq V$ and the set
	\begin{equation*}
		M_{\mb{p}}:=\set{k\in\N\colon \supp\varphi_{k}\cap B_{X}(\mb{p},r_{\mb{p}})\neq \emptyset}
	\end{equation*}
	is finite. The collection of balls $(B_{X}(\mb{p},r_{\mb{p}}))_{\mb{p}\in E}$ is then an open cover of the compact set $E$. Accordingly, it has a finite subcover. In other words, there exists a finite subset $F$ of $E$ such that 
	\begin{equation*}
		E\subseteq \bigcup_{\mb{p}\in F}B_{X}(\mb{p},r_{\mb{p}}).
	\end{equation*}
	Let 
	\begin{equation*}
		K:=\max\bigcup_{\mb{p}\in F}M_{\mb{p}}
	\end{equation*}
	and
	\begin{equation*}
		V':=\br{\bigcup_{\mb{p}\in F}B_{X}(\mb{p},r_{\mb{p}})}\cap \bigcap_{k=1}^{K}G_{k}.
	\end{equation*}
	Finally, we apply Remark~\ref{rem:lip-loc_glob} to define $U\subseteq V'$ with the required properties. The assertions of the lemma for $K$ and $U$ are now readily verified. 
\end{proof}
\end{appendices}
\bibliographystyle{plain}
\bibliography{biblio}
\end{document}